\documentclass[reqno, a4paper]{amsart}
\usepackage{amsmath, amssymb,amsthm}

\pdfoutput=1
\usepackage{color}
\def\comment#1{}
\usepackage[subtle,wordspacing=normal,mathspacing=normal, title=normal, tracking=normal]{savetrees}

\usepackage{txfonts}
\usepackage{enumerate}
\usepackage{caption}
\usepackage{accents}

\renewcommand{\complement}{^c}
 \DeclareMathOperator{\deb}{deb}
 \DeclareMathOperator{\conv}{conv}

\newcommand{\LT}{\mathfrak{L}}

 \newcommand{\xii}{\xi}

\newcommand{\OB}{Ob{\l}{\'o}j}

\newcommand{\cost}{c}

\newtheorem{theorem}{Theorem}
\newtheorem{assumption}[theorem]{Assumption}
\newtheorem{corollary}[theorem]{Corollary}
\newtheorem{definition}[theorem]{Definition}
\newtheorem{lemma}[theorem]{Lemma}
\newtheorem{proposition}[theorem]{Proposition}

\newtheorem*{importantconvention}
{Important Convention}

\theoremstyle{remark}
\newtheorem{remark}[theorem]{Remark}
\newtheorem{example}[theorem]{Example}

 \newcommand{\eps}{\varepsilon}

  \newcommand{\F}{\mathcal{F}}
 
 \newcommand{\M}{\mathsf{M}}
 
 \newcommand{\W}{\mathcal{W}}

 \renewcommand{\phi}{\varphi}

\newcommand{\E}{\mathbb{E}}
\renewcommand{\P}{\mathbb{P}}
\newcommand{\N}{\mathbb{N}}
\newcommand{\Q}{\mathbb{Q}}
\newcommand{\R}{\mathbb{R}}

\newcommand{\law}{\mathrm{Law}}
\newcommand{\lsub}[1]{{}_{\llcorner #1}}
\newcommand{\Xs}{\mathsf{X}}
\newcommand{\Ys}{\mathsf{Y}}
\newcommand{\cl}{\textsc{cl}}
\newcommand{\op}{\textsc{op}}
\newcommand{\G}{\mathcal G}

\DeclareMathOperator{\supp}{supp}
\newcommand{\bes}{\begin{subequations}}
\newcommand{\ees}{\end{subequations}}
\newcommand{\eea}{\end{eqnarray}}

\newcommand{\cF}{\mathcal{F}}

\renewcommand{\W}{{\mathbb W}}

\newcommand{\BB}{{\mathbb B}}

\usepackage{bbm}
\newcommand{\I}{\mathbbm 1}
\newcommand{\1}{\mathbbm 1}

\renewcommand{\eps}{\varepsilon}
\renewcommand{\epsilon}{\varepsilon}

\newcommand{\Leb}{\mathcal{L}}
\newcommand{\leb}{\mathcal{L}}

\DeclareMathOperator{\id}{Id}
\DeclareMathOperator{\proj}{proj}

\newcommand{\fourIdx}[5]{%
\setbox1=\hbox{\ensuremath{^{#1}}}%
 \setbox2=\hbox{\ensuremath{_{#2}}}%
 \setbox5=\hbox{\ensuremath{#5}}%
 \hspace{\ifnum\wd1>\wd2\wd1\else\wd2\fi}%
 \ensuremath{\copy5^{\hspace{-\wd1}\hspace{-\wd5}#1\hspace{\wd5}#3}%
 _{\hspace{-\wd2}\hspace{-\wd5}#2\hspace{\wd5}#4}%
 }}

\numberwithin{equation}{section}
\numberwithin{theorem}{section}

\newcommand{\TRT}{\mathsf{JOIN}}

\newcommand{\RST}{\mathsf{RST}}

\newcommand{\TRST}{\mathsf{JOIN}}
\newcommand{\pr}{\mathsf{proj}}

\newcommand{\BP}{\mathsf{SG}}
\newcommand{\SG}{\mathsf{SG}}
\newcommand{\SGW}{{\widehat{\mathsf{SG}}}}

\newcommand{\SSG}{{\mathsf{SG}_2}}

\newcommand{\Opt}{\mathsf{Opt}}

\newcommand{\DC}{\mathsf{DC}}

\renewcommand{\ll}{\llbracket}

\renewcommand{\subset}{\subseteq}
\renewcommand{\supset}{\supseteq}

\newcommand{\ol}[1]{\bar{#1}}
\newcommand{\ul}[1]{\underaccent{\bar}{#1}}
\newcommand{\nimmdasnicht}[1]{\overline{#1}}

\renewcommand{\S}{S}

\newcommand{\CRx}{{C_x(\R_+)}}
\newcommand{\CRo}{{C_0(\R_+)}}
\newcommand{\CRR}{{C(\R_+)}}

\newcommand{\oCRo}{{{\nimmdasnicht C}_0(\R_+)}}
\renewcommand{\BB}{\ol{B}}

\newcommand{\oF}{\bar \F}
\newcommand{\oQ}{\bar \Q}
\newcommand{\oW}{\nimmdasnicht \W}
\newcommand{\oE}{\bar \E}
\usepackage{subcaption}
\usepackage{asymptote}
\usepackage{graphicx}
\begin{asydef}
usepackage("amsmath");
usepackage("txfonts");
texpreamble("\newcommand{\ul}[1]{\underline{#1}}\newcommand{\ol}[1]{\overline{#1}}\newcommand{\Rt}{\overline{R}}");
import fontsize;
defaultpen(fontsize(10pt));
\end{asydef}

\renewcommand{\llcorner}{\upharpoonright}

\begin{document}

\title{Optimal Transport and Skorokhod Embedding}

\author{Mathias Beiglb\"ock} \author{Alexander M.~G.~Cox}
\author{Martin Huesmann} \thanks{The authors thank Julio Backhoff,
  Manu Eder, Walter Schachermayer, Nizar Touzi, and the anonymous
  referees for helpful comments. The first author was supported by the
  FWF-grants P26736 and Y782, the third author by the CRC 1060.}
  \begin{abstract}

    The \emph{Skorokhod embedding problem} is to represent a given probability
    as the distribution of Brownian motion at a chosen stopping time.  Over
    the last 50 years this has become one of the important classical problems in
    probability theory and a number of authors have constructed
    solutions with particular optimality properties.  These constructions employ
    a variety of techniques ranging from excursion theory to potential 
    and PDE theory and have been used in many different branches of pure and
    applied probability.
 
    We develop a new approach to Skorokhod embedding based on ideas
    and concepts from \emph{optimal mass transport}.  In analogy to the
    celebrated article of Gangbo and McCann on the geometry of optimal
    transport, we establish a geometric characterization of Skorokhod embeddings
    with desired optimality properties.  This leads to a systematic method to
    construct optimal embeddings. It allows us, for the first time, to derive all
    known optimal Skorokhod embeddings as special cases of one unified
    construction and leads to a variety of new embeddings.  While previous
    constructions typically used particular properties of Brownian motion, our
    approach applies to all sufficiently regular Markov processes.
 
\medskip

\end{abstract}
\maketitle
\section{Introduction}
Let $B$ be a Brownian motion started in $0$ and consider a probability $\mu$ on
the real line which is centered and has second moment.  The Skorokhod embedding
problem is to construct a stopping time $\tau$ embedding $\mu$ into Brownian
motion in the sense that
\begin{align}\label{SkoSol}\tag{SEP}\textstyle
B_\tau \mbox{ is distributed according to } \mu, \quad \E[\tau] < \infty.
\end{align}
Here, the second condition is imposed to exclude certain undesirable solutions, and can be modified to extend to measures without a second moment. As already demonstrated by Skorokhod \cite{Sk61, Sk65} in the early 1960s, it is always possible to construct solutions to the problem.  Indeed, the survey article \cite{Ob04} of \OB{} classifies 21 distinct solutions to \eqref{SkoSol}, although this list (from 2004) misses many more recent contributions.  A common inspiration for many of these papers is to construct solutions to \eqref{SkoSol} that exhibit additional desirable properties or a distinct internal structure. These have found applications in different fields and various extensions of the original problem have been considered.  We refer to \cite{Ob04} (and the 120+ references therein) for a comprehensive account of the field.
  
\medskip
 
 Our aim is to develop a new approach to \eqref{SkoSol} based on ideas from optimal transport. Many of the previous developments are thus obtained as applications of one unifying principle (Theorem \ref{GlobalLocal}) and several difficult problems are rendered tractable. Moreover, our methods can easily handle a number of more general versions of the problem: for example, integrable measures, general starting distributions, and $\R^d$-valued Feller processes.

 \subsection{A motivating example --- Root's construction} \label{sec:RootIntro}

 To illustrate our approach we introduce Root's construction, \cite{Ro69}, which will serve as inspiration in the rest of the paper. Root's construction is one of the earliest solutions to \eqref{SkoSol}, and it is prototypical for many further solutions to \eqref{SkoSol} in that it has a simple \emph{geometric description} and possesses a certain \emph{optimality property} in the class of all solutions.

\begin{figure}[th]
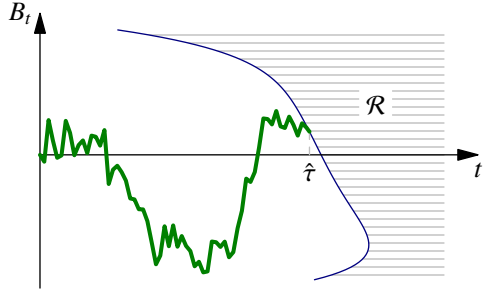

\begin{asy}[width=0.5\textwidth]
  import graph;
  import stats;
  import patterns;

  // import settings;

  // gsOptions="-P"; 

  // Construct a Brownian motion of time length T, with N time-steps
  int N = 100;
  //int N = 300;
  real T = 1.6*1.6;
  real dt = T/N;
  real B0 = 0;

  real sig = 0.7;

  real xmax = 0.8;
  real xmin = -0.8;

  real tmax = (xmax-xmin)*1.6;

  real[] B; // Brownian motion
  real[] t; // Time

  path BM;

  // Seed the random number generator. Delete for a "random" path:
  srand(103);

  B[0] = B0;
  t[0] = 0;

  BM = (t[0],B[0]);

  // Define a barrier

  real R(real y) {return 1.55 + ((y-1)**4)/8-(y**2)/1.5 -(y**6)*4;}

  int H = N+1;
  int H2 = N+1;
  int BMstop;
  int BMstop2 = N+2;

  for (int i=1; i<N+1; ++i)
  {
    B[i] = B[i-1] + sig*Gaussrand()*sqrt(dt);
    t[i] = i*dt;

    if ((H==N+1)&&(t[i]>=R(B[i])))
    {
	H = i;
	BMstop = length(BM);
       BM = BM--(R(B[i]),B[i]);
    }
    else
    {
      BM = BM--(t[i],B[i]);
    }

  }

  if (H==N+1)
  BMstop = length(BM);

  pen p = deepgreen + 1.5;
  pen p2 = lightgray + 0.25;
  //pen p2 = mediumgray + 1;

  //if (H<N+1)
  //draw(subpath(BM,BMstop,BMstop2),p2);

  pair tau = point(BM,BMstop+1);

  pen q = black + 0.5;

  real eps1 = 0.05;
  real eps2 = 0.15;

  path barrier = (graph(R,identity,xmin+eps1,xmax-eps1)--(tmax-eps2,xmax-eps1)--(tmax-eps2,xmin+eps1)--cycle);

  add("hatch",hatch(1mm,W,mediumgray));

  fill(barrier,pattern("hatch"));

  draw(graph(R,identity,xmin+eps1,xmax-eps1),NW,deepblue+0.5);

  draw((0,xmin)--(0,xmax+eps1),q,Arrow);
  draw((0,0)--((T+eps1),0),q,Arrow);

  draw((tau.x,0)--tau,mediumgray+dashed);
  label("$\hat\tau$",(tau.x,0),S);

  draw(subpath(BM,0,BMstop+1),p);

  label("$t$",(T+eps1,0),S);
  label("$B_t$",(0,(xmax+eps1)),(-1,0));

  label("$\mathcal{R}$",(2,0.3),UnFill(0.5mm));

  //label("$D_{\Rt}$",(2.2,1.35),UnFill(0.5mm));
\end{asy}

 \caption{Root's solution of \eqref{SkoSol}.}
  \label{fig:RootSolution}
\end{figure}

Root established that there exists a \emph{barrier} $\mathcal R$ (which is
essentially unique) such that the Skorokhod embedding problem is solved by the
stopping time \begin{align}\label{RootType}\tau_{\text{Root}}=\inf\{ t \geq 0 :
  (t,B_t)\in\mathcal R\}.
\end{align} 
A barrier is a Borel set $\mathcal R \subseteq \R_+\times \R$ such that
$(s,x)\in \mathcal R$ and $s < t$ implies $(t,x)\in \mathcal R$.  The Root
construction is distinguished by the following {optimality property}: among all
solutions to \eqref{SkoSol} for a fixed terminal distribution $\mu$, it
minimizes $\E[\tau^2]$. For us, the optimality property will be the
starting point from which we deduce a geometric characterization of
$\tau_{\text{Root}}$.  To this end, we now formalize the corresponding
optimization problem.

\subsection{Optimal Skorokhod Embedding Problem}

We consider the set of \emph{stopped paths}
\begin{align}\label{StoppedPaths}
S =\{(f,s): f:[0,s] \to \R \mbox{ is continuous, $f(0)=0$}\}.
\end{align} 
Throughout the paper we  consider  a function
$$\gamma:S\to \R.$$ 
We fix a stochastic basis $\Omega=(\Omega,\G,(\G_t)_{t\geq0},\P)$ which is sufficiently rich to support a Brownian motion $B$ and a uniformly distributed $\G_0$-random variable, independent of $B$.  The \emph{optimal Skorokhod embedding problem} is to construct a stopping time $\tau $ on $\Omega$ which optimizes
\begin{align}\label{IntPri}\tag{OptSEP}
  P_\gamma=\inf\Big\{ \E\big[\gamma\big((B_t)_{t\leq
    \tau},\tau\big)\big]: \mbox{ $\tau$ solves \eqref{SkoSol} }\Big\}.
\end{align}
We emphasize that \eqref{IntPri} does not depend on the particular choice of the underlying basis as long as it is rich enough in the above sense, cf.\ Lemma \ref{lem:equivOpt} / Section \ref{sec:optim-probl-dual-start}. We will usually assume that \eqref{IntPri} is \emph{well posed} in the sense that $\E\big[\gamma\big((B_t)_{t\leq \tau},\tau\big)\big]$ exists with values in $(-\infty,\infty]$ for all $\tau$ which solve \eqref{SkoSol} and is finite for one such $\tau$.

The Root stopping time solves \eqref{IntPri} in the case where $\gamma(f,s)= s^2$. Other examples where the solution is known include functions depending on the running maximum $\gamma((f,s)):= \bar f(s):= \max_{t\leq s} f(t)$ or functions of the local time at $0$.

The solutions to \eqref{SkoSol} have their origins in many different branches of probability theory, and in many cases, the original derivation of the embedding occurred separately from the proof of the corresponding optimality properties. Moreover, the optimality of a given construction is often not immediate; for example, the optimality property of the Root embedding was first conjectured by Kiefer \cite{Ki72} and subsequently established by Rost \cite{Ro76}.

In contrast to existing work, we will start with the optimization problem \eqref{IntPri} and we seek a systematic method to determine the minimizer for a given function $\gamma$.  To develop a general theory for this optimization problem we interpret stopping times in terms of a transport plan from the Wiener space $(\CRo,\W)$ to the target measure $\mu$, i.e.\ we want to think of a stopping time $\tau$ as transporting the mass of a trajectory $(B_t(\omega))_{t \in \R_+}$ to the point $B_{\tau(\omega)}(\omega)\in\R.$ Note that this is \emph{not a coupling} between $\W$ and $\mu$ in the usual sense and one cannot directly apply optimal transport theory.  Nevertheless the transport perspective provides a powerful intuition that guides us to develop an analogous theory, which in particular accounts for the adaptedness properties of stopping times.  To this end, it is necessary to combine ideas and results from optimal transport with concepts and techniques from stochastic analysis.

As in optimal transport, it is crucial to consider \eqref{IntPri} in a suitably relaxed form, i.e.\ in \eqref{IntPri} we will optimize over \emph{randomized stopping times} (see Definition~\ref{def:RST} below). These can be viewed as usual stopping times on a possibly enlarged probability space but in our context it is more natural to interpret them as stopping times of `Kantorovich-type' (in the sense of optimal transport), i.e.\ stopping times which terminate a given path not at a single deterministic time instance but according to a distribution.

This relaxation will allow us to transfer many of the convenient properties of classical transport theory to our probabilistic setup.  Exactly as in classical transport theory, \eqref{IntPri} can be viewed as a linear optimization problem. The set of couplings in mass transport is compact and similarly the set of all randomized stopping times solving \eqref{SkoSol} on Wiener space is compact in a natural sense.  Under the standing assumption that $B$ is defined on a sufficiently rich stochastic basis, these considerations allow us to prove:
\begin{theorem}\label{MinimizerExists}
  Let $\gamma:S\to \R$ be lsc and bounded from below.  Then \eqref{IntPri} admits a minimizing stopping time $\tau$. 
\end{theorem}
Here we can talk about the continuity properties of $\gamma$ since $S$ possesses a natural Polish topology (cf.\ \eqref{STop}).

In the language of linear optimization, Theorem \ref{MinimizerExists} is a primal problem. It is therefore natural to expect that there exists a corresponding dual problem, and our second main result concerns this duality:
\begin{theorem}\label{DualityIntro}
  Let $\gamma: S \to \R$ be lsc and bounded from
  below, and set   $$D_\gamma =\sup\left\{\int \psi(y)\, d\mu(y): \psi \in C(\R), \exists M, 
    \begin{array}{l}
      M\mbox{ is a continuous $\G$-martingale}, M_0 = 0 \\
      \P-\mbox{a.s.}, \forall t \ge 0,  M_t + \psi(B_t) \leq \gamma((B_s)_{s \le
      t},t) 
    \end{array}
  \right\}
  $$
  where $M, \psi$ 
  satisfy $|M_t| \leq a + bt+c B_t^2 $, $|\psi(y)| \leq a+ b y^2$
  for some $a,b,c>0$.   Then we have the duality relation 
  \begin{align}
    \label{HedgingDualEq2Intro} 
    P_\gamma = D_\gamma.
  \end{align}
\end{theorem}

We will prove this result in Section~\ref{sec:optim-probl-dual}, and variants of this result will prove to be important in establishing later results. Theorem~\ref{DualityIntro} has close analogues in the literature. In particular, using Hobson's time change argument (\cite{Ho98a,Ho11}), Theorem \ref{DualityIntro} is comparable to the work of Dolinsky and Soner \cite{DoSo13,DoSo14}. Similar duality results in a discrete time framework are established by Bouchard and Nutz \cite{BoNu13} among others.

\subsection{Geometric Characterization of Optimizers --- Monotonicity Principle}

A fundamental idea in optimal transport is that the optimality of a transport plan is reflected by the geometry of its support set. Often this is key to understanding the transport problem. On the level of support sets, the relevant notion is \emph{$c$-cyclical monotonicity}.  The relevance of this concept for the theory of optimal transport has been fully recognized by Gangbo and McCann \cite{GaMc96}, based on earlier work of Knott and Smith \cite{KnSm84} and R\"uschendorf \cite{Ru91,Ru95} among others.

Inspired by these results, we establish a \emph{monotonicity principle} which links the optimality of a stopping time $\tau$ with `geometric' properties of $\tau$. Combined with Theorem~\ref{MinimizerExists}, this principle will turn out to be surprisingly powerful. For the first time, \emph{all} the known solutions to \eqref{SkoSol} with optimality properties can be established through one unifying principle. Moreover, the monotonicity principle allows us to treat the optimization problem \eqref{IntPri} in a systematic manner, generating further embeddings as a by-product.

Our third main result states: 
\begin{theorem}[Monotonicity Principle]\label{GlobalLocal} 
  Let $\gamma:S\to\R$ be Borel measurable.  Suppose that \eqref{IntPri} is well posed and $\tau$ is an optimizer.  Then there exists a \emph{$\gamma$-monotone} (cf.\ Definition \ref{def:Gamma1} below) Borel set $\Gamma\subset S$ such that $\P$-a.s.
  \begin{align}\label{GammaSupport} ((B_t)_{t\leq\tau},\tau)\in\Gamma\;.\end{align}
\end{theorem}
If \eqref{GammaSupport} holds, we will loosely say that $\Gamma$ \emph{supports} $\tau$. The significance of Theorem~\ref{GlobalLocal} is that it links the optimality of the stopping time $\tau$ with a particular property of the set $\Gamma$, i.e. $\gamma$-monotonicity. In applications, the latter turns out to be much more tangible. We emphasize that we do not require continuity assumptions on $\gamma$ in this result. This will be important when we apply our results.

To link the optimality of a stopping time with properties of the set $\Gamma$ we consider the minimization problem \eqref{IntPri} on a pathwise level.  Consider two paths $(f,s), (g,t)\in S$ which end at the same value, i.e.\ $f(s)=g(t)$. We want to determine which of the two paths should be \emph{stopped} and which one should be allowed to \emph{go on} further, bearing in mind that we try to minimize $\E[\gamma((B_s)_{s\leq \tau}, \tau)]$. To make this definition formal, we need to perform an operation at the level of individual paths.  We will write $f\oplus h$ for the concatenation of the two paths $(f,s), (h,u) \in S$, specifically:
\begin{equation*}
  (f\oplus h)(r) := 
  \begin{cases}
    f(r) & r \le s\\
    f(s) + h(r-s) & r \in (s,s+u]
  \end{cases}.  
\end{equation*}
Then we set
\begin{align}\label{eq:-f}
&\gamma^{(f,s)\oplus }(h,u) := \gamma(f\oplus h,s+u).
\end{align}
We will call $\big((f,s), (g,t)\big)$ a \emph{stop-go} pair if it is
advantageous to \emph{stop} $(f,s)$ and to \emph{go on} after $(g,t)$
in the following sense:
\begin{definition}\label{def:SG}   The pair $\big((f,s), (g,t)\big)\in S\times S$   is a \emph{stop-go
    pair}, written $\big((f,s), (g,t)\big)\in\BP$, iff $f(s)=g(t)$ and
  \begin{align}
    \label{BPIneqProb-intro}
    \E\left[\gamma^{(f,s)\oplus }\left(\left(  B_u\right)_{u\leq
    \sigma}, \sigma\right)\right] +
    \gamma(g,t)\quad>\quad \gamma(f,s)  + 
    \E\left[\gamma^{(g,t)\oplus }\left(\left(  B_u\right)_{u\leq
    \sigma}, \sigma\right)\right]
  \end{align}
  for \emph{every} $(\F^B_t)_{t \ge 0}$-stopping time $\sigma$ which satisfies
  $0 < \E[\sigma] < \infty$ and for which both sides of
    \eqref{BPIneqProb-intro} are well defined and the left hand side
    is finite.
\end{definition}
Here $(\F_t^B)_{t \ge 0}$ denotes the natural filtration generated by the Brownian motion $B$. A consequence of considering only $(\F_t^B)_{t \ge 0}$-stopping times is that the set $\SG$ does not depend on the particular choice of the underlying stochastic basis.

\begin{figure}[th]
\resizebox{.8\textwidth}{!}{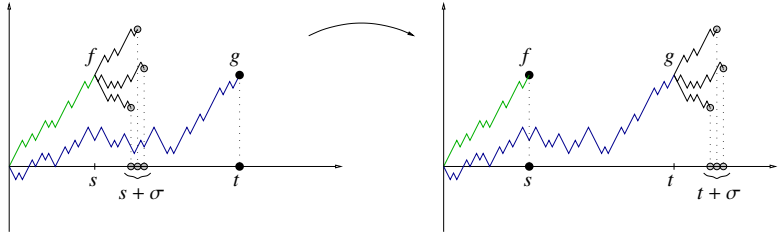}
 \caption{ The left hand side of \eqref{BPIneqProb-intro} corresponds to averaging the function $\gamma$ over the stopped paths on the left picture; the right hand side to averaging the function $\gamma$ over the stopped paths on the right picture. 
}
  \label{fig:BadPairs3-intro}
\end{figure}
We note that a swapping of paths (as illustrated in Figure~\ref{fig:BadPairs3-intro}) was used by Hobson \cite[p 34]{Ho11} to provide a heuristic derivation of the optimality properties of the Root embedding. Indeed Hobson's approach was the starting point of the present paper.

Recalling \eqref{GammaSupport}, we see that the set $\Gamma \subset S$ contains the \emph{stopped} paths: that is, a path $(g,t)$ is in $\Gamma$ if there is some possibility that the optimal stopping rule decides to stop at time $t$ having observed the path $(g(u))_{u \in [0,t]}$. In addition, we need to consider those paths which we observe as the initial section of a longer, stopped, path:
these are the \emph{going} paths
\begin{align}\label{Gammakl}
  \Gamma^<:=\big\{(f,s): \exists (\tilde f,\tilde s)\in \Gamma, s< \tilde s \mbox{ and $f\equiv \tilde f$ on $[0,s]$} \big\}\;.
 \end{align}
We can now formally introduce $\gamma$-monotonicity.
\begin{definition}\label{def:Gamma1}
 A set $\Gamma\subset S$ is called $\gamma$-monotone iff $\Gamma^< \times \Gamma$ contains no stop-go pairs, i.e.\
\begin{align}\label{eq:Gamma1} \SG\ \cap \  \big(\Gamma^< \times \Gamma \big) =\emptyset.\end{align}
\end{definition} 
By the monotonicity principle, Theorem~\ref{GlobalLocal}, an optimal stopping time is supported by a set $\Gamma$ such that $\Gamma^<\times \Gamma$ contains no stop-go pair $\big((f,s),(g,t)\big)$. Intuitively, such a pair gives rise to a possible modification, improving the given stopping rule: as $f(s) = g(t)$, we can imagine stopping the path $(f,s)$ at time $s$, and allowing $(g,t)$ to go on by transferring all paths which extend $(f,s)$, the `remaining lifetime', onto $(g,t)$, which is now \emph{going} (see Figure~\ref{fig:BadPairs3-intro}). By \eqref{BPIneqProb-intro} this guarantees an improved value of $P_\gamma$, contradicting the optimality of our stopping rule. Observe that the condition $f(s)=g(t)$ is what  guarantees that a modified stopping rule still embeds the measure $\mu$.  In Section~\ref{sec:part-embedd} below we will briefly indicate how the monotonicity principle can be used to derive existing solutions to the Skorokhod embedding problem as well as a whole family of novel solutions to the Skorokhod embedding problem; many further examples will be provided in Section~\ref{sec:embeddings-abundance}.

Importantly, the transport-based approach readily admits a number of strong generalizations and extensions. With only minor changes the existence result, Theorem~\ref{MinimizerExists}, the duality result, Theorem~\ref{DualityIntro}, and the monotonicity principle, Theorem~\ref{GlobalLocal} below, extend to general starting distributions and Brownian motion in $\R^d$, and more generally to sufficiently regular Markov processes; see Sections \ref{s:mp} and \ref{sec:feller}. This is notable since previous constructions usually exploit rather specific properties of Brownian motion.

The monotonicity principle, Theorem~\ref{GlobalLocal}, represents the culmination of the three main results, and the proof of this result will be the most complex part of this paper, requiring substantial preparation in order to combine the relevant concepts from stochastic analysis and optimal transport. The preparation and proof of this result will therefore comprise the majority of the paper. In fact the proof will automatically imply a stronger version (Theorem \ref{GlobalLocal2}) of Theorem~\ref{GlobalLocal}. For our applications, it will also be helpful to introduce a version of this result which incorporates a secondary optimization, Theorem \ref{thm:second-maxim-result}.

The `classical' optimal transport version of Theorem~\ref{GlobalLocal} can be established through fairly direct arguments, at least in a reasonably regular setting, cf.\ \cite[Thms.\ 3.2, 3.3]{AmPr03} and \cite[p. 88f]{Vi03}.  However, these approaches do not extend easily to our setup: stopping times are of course not couplings in the usual sense and there is no reason for particular combinatorial manipulations to carry over in a direct fashion. Another substantial difference is that the procedure of transferring paths described below Definition \ref{def:Gamma1} necessarily refers to a \emph{continuum} of paths while the classical notion of cyclical monotonicity is concerned with rearrangements along finite cycles.  The argument given subsequently is more in the spirit of \cite{BGMS08, BiCa10} and requires a fusion of ideas from optimal transport and stochastic analysis.

\subsection{New Horizons}

The results presented in this paper are limited to  the case of the classical Skorokhod embedding problem for Markov processes with continuous paths. However we believe that our methods are sufficiently general that a number of interesting and important extensions, which previously would have been intractable, may now be within reach:
\begin{enumerate}
\item {\bf Markov processes:} The results presented in this paper should extend
  to a more general class of Markov processes with c\`adl\`ag paths. The main
  technical issues this would present lie in the generalization of the results
  in Section~\ref{sec:prel-stopp-times}, where the specific structure of the
  space of continuous paths is exploited.
\item {\bf Multiple path-swapping:} In our monotonicity principle, Theorem~\ref{GlobalLocal}, we consider the impact of swapping mass from a single unstopped path onto a single stopped path, and argue that if this improves the objective $\gamma$ on average, then the stopping time in question was not optimal. In classical optimal transport, it is known that single swapping is not sufficient to guarantee optimality; rather, one needs to consider the impact of allowing a finite `cycle' of swaps to occur, and moreover, that this is both a necessary and sufficient condition for optimality. It is natural to conjecture that a similar result applies in the present setup.

\item {\bf Multiple marginals:} A natural generalization of the Skorokhod
  embedding problem is to consider the case where a sequence of measures,
  $\mu_1, \mu_2, \dots, \mu_n$ are given, and the aim is to find a sequence of
  stopping times $\tau_1 \le \tau_2 \le \dots \le \tau_n$ such that $B_{\tau_k}
  \sim \mu_k$, and such that the chosen sequence of stopping times minimizes
  $\E[\gamma((B_t)_{t \le \tau_n},\tau_1,\dots,\tau_n)]$ for a suitable
  function $\gamma$. In this setup, it is natural to ask whether there exists
  a suitable monotonicity principle, corresponding to
  Theorem~\ref{GlobalLocal}.

\item {\bf Constrained embedding problems:} In this paper, we consider classical embedding problems, where the optimization is carried out over the class of solutions to \eqref{SkoSol}. However, in many natural applications, one needs to further consider the class of {\it constrained} embedding problems: for example, where one minimizes some function over the class of embeddings which also satisfy a restriction on the probability of stopping after a given time. It would be natural to derive generalizations of our duality results, and a corresponding monotonicity principle for such problems.
\end{enumerate}

\subsection{Background}

Since the first solution to \eqref{SkoSol} by Skorokhod \cite{Sk65} the embedding problem has received frequent attention in the literature, with new solutions appearing regularly, and exploiting a number of different mathematical tools. Many of these solutions also prove to be, by design or accident, solutions of \eqref{IntPri} for a particular choice of $\gamma$, e.g.\ \cite{Ro69, Ro76,AzYo79,Ja88,Va83,Pe86}. The survey \cite{Ob04} is a comprehensive account of all the solutions to \eqref{SkoSol} up to 2004 and references many articles which use or develop solutions to the Skorokhod embedding problem.  
More recently, novel twists on the classical Skorokhod embedding problem have been investigated by: Last et.~al.~\cite{LaMo14}, who consider the closely related problem of finding unbiased shifts of Brownian motion (and where there are also natural connections to optimal transport); Hirsch et.~al.~\cite{HiPr11}, who have used solutions to the Skorokhod embedding problem to construct Peacocks; and Gassiat et.~al.~\cite{GaMiOb14}, who have exploited particular properties of Root's solution to construct efficient numerical schemes for SDEs.

The Skorokhod embedding problem has also recently received substantial attention from the mathematical finance community. This goes back to an idea of Hobson \cite{Ho98a}: through the Dambis-Dubins-Schwarz Theorem, the optimization problems \eqref{IntPri} are related to the pricing of financial derivatives, and in particular to the problem of \emph{model-risk}. We refer the reader to the survey article \cite{Ho11} for further details.

Recently there has been much interest in optimal transport problems where the transport plan must satisfy additional martingale constraints. Such problems arise naturally in the financial context, but are also of independent mathematical interest, for example --- mirroring classical optimal transport --- they have important consequences for the study of martingale inequalities (see e.g.\ \cite{BoNu13,HeObSpTo12,ObSp14}).  The first papers to study such problems include \cite{HoNe12, BeHePe12, GaHeTo12, DoSo12}, and this field is commonly referred to as \emph{martingale optimal transport}. The Skorokhod embedding problem has been considered in this context by Galichon et.~al.~ in \cite{GaHeTo12}; through a stochastic control problem they recover the Az\'ema-Yor solution of the Skorokhod embedding problem. Notably, their approach is very different from the one pursued in the present paper.

\subsection{Outline of the Article} In Section~\ref{sec:part-embedd} we establish the Root and the Rost embeddings as a consequence of Theorems~\ref{MinimizerExists} and \ref{GlobalLocal}, as well as constructing a family of new embeddings.  The results presented in this section are intended as a motivation for the rest of the paper.  In the derivation of these embeddings we highlight the interplay between arguments of a probabilistic nature, and arguments relating to the pathwise space $S$ introduced in \eqref{StoppedPaths}.  A major benefit of working in these two separate domains is that it is typically relatively easy to prove pointwise statements in the setup of the space $S$; on the other hand, the associated probabilistic arguments are usually straightforward.  However neither set of arguments naturally transfers to the other setup.

The link between these distinct domains is provided by Theorems \ref{MinimizerExists} and \ref{DualityIntro}, and in particular the monotonicity principle Theorem \ref{GlobalLocal} which we establish in Sections 3 to 5. In Section~\ref{sec:prel-stopp-times}, we introduce a framework that allows us to view classical probabilistic concepts on the pathwise space $S$ and establish a number of auxiliary results that will be needed later on.  In Section~\ref{sec:optim-probl-dual} we prove our first two main results. As in the transport case, Theorem~\ref{MinimizerExists} will be a simple consequence of lower semi-continuity plus compactness of the set of solutions to the Skorokhod problem.  To establish Theorem~\ref{DualityIntro}, we use classical duality results from optimal transport.  In Section~\ref{s:mp} we prove Theorem~\ref{GlobalLocal} based on a combination of arguments from optimal transport with Choquet's capacitability theorem and ingredients from stochastic analysis.

In Section~\ref{sec:embeddings-abundance} we use our results to establish all  known solutions to \eqref{IntPri} as well as further embeddings. We also give an example in which \eqref{IntPri} admits only optimizers depending on additional randomization.
For readers who are mainly interested in these applications, it should be possible to read this section immediately after Section~\ref{sec:part-embedd}. 

 In Section~\ref{sec:feller} we describe a number of extensions of our previous results. In particular we consider general starting distributions and show that our main results extend to  continuous Feller processes under certain assumptions which we are able to verify for a large class of processes. As a special case of the results in this section, we also show that, as usual, the  moment condition on $\mu$ can be dropped when the second condition in \eqref{SkoSol} is recast in terms of uniform integrability resp.\ minimality (cf.\ \eqref{eq:minimalDefn}).

\subsection{Frequently used notation}
\begin{itemize}
\item The set of (sub-)probability measures on a space $\Xs$ is denoted by $\mathcal P(\Xs)$ / $\mathcal P^{\leq 1}(\Xs)$.
\item For a measure $\xi$ on $\Xs$ we write $f(\xi)$ for the push-forward of $\xi$ under $f:\Xs\to \Ys$.
\item We use $\xi(f)$ as well as $\int f~ d\xi$ to denote the integral of a function $f$ against a measure $\xi$.
\item Stochastic processes are usually denoted by capital letters like $X,Y,Z$. 
\item $\CRx$ denotes the continuous functions starting in $x$; $\CRR=\bigcup_{x\in\R}\CRx$.  
\item The set of stopped paths is $
S =\{(f,s): f:[0,s] \to \R \mbox{ is continuous, $f(0)=0$}\}
$ and we define $r:\CRo\times\R_+\to S$ by $r(\omega, t):= (\omega_{\upharpoonright [0,t]},t)$.
\item For $\Gamma \subseteq S$ we set $\Gamma^<:=\{(f,s): \exists (\tilde f,\tilde s)\in \Gamma, s< \tilde s \mbox{ and $f\equiv \tilde f$ on $[0,s]$}\}.$
\item For   $(f,s) \in S$ we write $\ol{f} = \sup_{r \le s} f(r)$, $\ul{f} = \inf_{r \le s} f(r)$ and $f^* = \sup_{r \le s} |f(r)|$.
\item We use $\oplus $ for the concatenation of paths: depending on the context the arguments may be elements of $S$, $\CRo$ or $\CRo\times\R_+$. 
\item 
  If $F$ is  a function on $S$ resp.\ $\CRo\times \R_+$ and $(f,s)\in S$ we set $F^{(f,s)\oplus} (y):= F((f,s)\oplus y)$, where $y$ may be an element of $ S$, $\CRo$, or $\CRo\times \R_+$.
\item $\W$ denotes Wiener measure; $\F^0$  ($\F^a$) the natural (augmented) filtration on $\CRo$.
\item Two commonly used probability spaces are $(\Omega, \G, (\G_t)_{t \ge 0}, \P)$, which is an arbitrary probability space, on which there exists a process $B$ which is Brownian motion, and sometimes also a $\G_0$-random variable $Y$ which is uniformly distributed on $[0,1]$. On this space, the natural filtration generated by the process $B$ is denoted by $(\F^B_t)_{t \ge 0}$ In addition, we sometimes refer to the space $(\oCRo, \oF, (\oF_t)_{t \ge 0}, \oW)$, which is the product space $\oCRo = \CRo \times [0,1]$ equipped with a suitable filtration (see the discussion above Theorem~\ref{thm:equiv RST} for further details) and the product measure $\oW=\W\otimes\mathcal L$ of Wiener and Lebesgue measure.
\end{itemize}

\section{Particular embeddings}
\label{sec:part-embedd} 
In this section we explain how Theorem \ref{GlobalLocal} can be used to derive particular solutions to the Skorokhod embedding problem, \eqref{SkoSol}, using the optimization problem \eqref{IntPri}. For much of the paper, we consider \eqref{SkoSol} for measures $\mu$ where $\int x^2\, \mu(dx) < \infty$. This constraint can be weakened to require only the first moment to be finite, subject to the restriction that the stopping time is \emph{minimal}: that is, if $\tau$ is a stopping time such that $B_{\tau} \sim \mu$, then for any stopping time $\tau'$,
\begin{equation}
  \label{eq:minimalDefn}
  B_{\tau'} \sim \mu \text{ and } \tau' \le \tau \text{ implies } \tau' = \tau \text{ a.s.}
\end{equation}
In the case where $\mu$ has a second moment, minimality and $\E[\tau]< \infty$ are equivalent.  We emphasize that, \emph{mutatis mutandis}, all of our results are valid in this more general setup, see Section~\ref{sec:feller}. Recall that we are working on a stochastic basis which is rich enough to support a Brownian motion and a uniformly distributed random variable.

\subsection{The Root embedding} \label{sec:root-embedding}

We recall the definition of the Root embedding, $\tau_{\text{Root}}$, from \eqref{RootType}, and we wish to recover Root's result (\cite{Ro69}) from an optimization problem. Remember that, according to Root's terminology, a (closed) set $\mathcal R \subseteq \R_+ \times \R$ is a \emph{barrier} if $(s,x) \in \mathcal R$ implies $(t,x) \in \mathcal R$ whenever $t>s$. Then Root's construction of a solution to the Skorokhod embedding problem can be summarized as follows:
\begin{theorem}\label{RootThm} Let $\gamma(f,t)= h(t)$, where $h:\R_+\to \R$ is a strictly convex function such that \eqref{IntPri} is well posed. Then a minimizer  of \eqref{IntPri} exists, and moreover for any minimizer $\hat \tau$, there exists a barrier $\mathcal R$ such that $\hat \tau=\inf\{ t \geq 0 : (t,B_t)\in\mathcal R\}$. In particular the Skorokhod embedding problem has a solution of barrier type as in \eqref{RootType}.
\end{theorem}
\begin{proof}
  \noindent {\bf Step 1.} We first pick --- by Theorem~\ref{MinimizerExists} --- a stopping time $\hat\tau$ which attains $P_\gamma.$   
   By Theorem~\ref{GlobalLocal} there exists a set $\Gamma \subseteq S$ such that $\left( \left(B_{s}\right)_{s \le\hat\tau}, \hat\tau\right) \in \Gamma$ almost surely, and such that $(\Gamma^{<} \times \Gamma) \cap \BP = \emptyset$.

  \noindent {\bf Step 2.} Next, consider paths $(f,s),(g,t)\in S$ such that $f(s)=g(t)$. We consider when $\big((f,s),(g,t)\big)\in \BP,$ i.e.\ under which conditions $(f,s)$ should be stopped and Brownian motion should continue to go after $(g,t)$. In the present case \eqref{BPIneqProb-intro} amounts to
\begin{align}\label{BPRoot}
  \E\big[h(s+\sigma)\big]\ + h(t) \quad>\quad
  h(s) \ +\ \E\big[h(t+\sigma)\big]. 
\end{align} 
Thus, by strict convexity of $h$, $\big((f,s),(g,t)\big) \in \SG$ iff $t <s$. 
We define two barriers by
\begin{align*}
\mathcal R_\textsc{cl}:=\{(s,x):\exists (g,t)\in\Gamma, g(t)=x, t\leq s\},\\
\mathcal R_\textsc{op}:=\{(s,x):\exists (g,t)\in\Gamma, g(t)=x, t< s\}.
\end{align*}
Fix $(g,t) \in \Gamma$. Then we have $(t,g(t)) \in \mathcal{R}_{\cl}$. 
Suppose for contradiction that $\inf\{s \in [0,t]: (s,g(s)) \in \mathcal{R}_{\op}\} < t$. Then there exists $s<t$ such that $(f,s) := \left(g_{\llcorner [0,s]},s\right) \in \Gamma^{<}$ and $(s, f(s)) \in \mathcal{R}_{\op}$. By definition of $\mathcal{R}_{\op}$, it follows that there exists another path $(k,u) \in \Gamma$ such that $k(u) = f(s)$ and $u  < s$. But then $\big( (f,s), (k,u)\big) \in \SG\cap \big(\Gamma^<\times \Gamma\big)$ which cannot be the case. Hence,
\begin{equation*}
  (g,t) \in \Gamma \implies \inf\{s \in [0,t]: (s,g(s)) \in \mathcal{R}_{\cl}\} \le t
  \le \inf\{s \in [0,t]: (s,g(s)) \in \mathcal{R}_{\op}\}.
\end{equation*}

\noindent {\bf Step 3.} Now consider $\omega \in \Omega$ such that
$(g,t) = \left( \left(B_{s}(\omega)\right)_{s \le\hat \tau(\omega)},
 \hat \tau(\omega)\right) \in \Gamma$. Then it follows immediately that:
\begin{equation}\label{eq:RootSandwich}
   \tau_{\cl}(\omega):= \inf\{s : (s,B_s(\omega)) \in \mathcal{R}_{\cl}\} \le \hat\tau(\omega) \le \inf\{s  : (s,B_s(\omega)) \in \mathcal{R}_{\op}\} =: \tau_{\op}(\omega).
 \end{equation}
 We finally observe that $\tau_\cl = \tau_\op$ a.s.\ by the strong Markov property, and the fact that one-dimensional Brownian motion immediately returns to its starting point.
\end{proof}

A consequence of this proof is that (on a given stochastic basis) there exists exactly one solution of the Skorokhod embedding problem which minimizes $\E[h(\tau)]$; this property was first established in \cite{Ro76}, together with the optimality property of Root's solution. To see this, assume that minimizers $\tau_1$ and $\tau_2$ are given. Then we can use an independent coin-flip to define a new minimizer $\bar \tau$ which is with probability $1/2$ equal to $\tau_1$ and with probability $1/2$ equal to $\tau_2$. By Theorem \ref{RootThm}, $\bar \tau$ is of barrier type and hence $\tau_1=\tau_2$. 
\begin{remark} \label{rem:ProofStructure}
  We highlight here the nature of the proof of Theorem~\ref{RootThm}. The proof divides into three steps, two of these steps (Steps 1 and 3) being probabilistic in nature, making arguments about random variables on a particular probability space. The second step, however, is purely a pointwise argument about the properties of subsets of $\Gamma$ in relation to the function $\gamma$ which we look to optimize. The latter arguments are \emph{not} probabilistic in nature.
\end{remark}

\begin{remark} \label{rem:Loynes} The following argument, due to Loynes~\cite{Lo70}, can be used to argue that barriers are unique in the sense that if two barriers solve \eqref{SkoSol}, then their hitting times must be equal. Suppose that $\mathcal{R}$ and $\mathcal{S}$ are both \emph{closed} barriers which embed $\mu$. Note that we can take the closed barriers without altering the stopping properties. Consider the barrier $\mathcal{R} \cup \mathcal{S}$: let $A \subset \Omega_{\mathcal{R}} := \{x: (t,x) \in \mathcal{S} \implies (t,x) \in \mathcal{R}\}$. Then $\P(B_{\tau_{\mathcal{R} \cup \mathcal{S}}} \in A) \le \P(B_{\tau_{\mathcal{R}}} \in A) = \mu(A)$. Similarly, for $A' \subset \Omega_{\mathcal{S}} := \{x: (t,x) \in \mathcal{R} \implies (t,x) \in \mathcal{S}\}$, $\P(B_{\tau_{\mathcal{R} \cup \mathcal{S}}} \in A') \le \P(B_{\tau_{\mathcal{S}}} \in A') = \mu(A')$. Since $\mu(\Omega_{\mathcal{R}} \cup \Omega_\mathcal{S}) = 1$, $\tau_{\mathcal{R} \cup \mathcal{S}}$ embeds $\mu$.

  It is known (see Monroe~\cite{Mo72}) that, when $\mu$ has a second moment, the second condition in \eqref{SkoSol}, $\E[\tau] < \infty$ is equivalent to minimality of the stopping time (recall \eqref{eq:minimalDefn}). It immediately follows from the argument above that if the barriers $\mathcal{R}$ and $\mathcal{S}$ solve \eqref{SkoSol}, then $\tau_{\mathcal{R}} = \tau_{\mathcal{S}}$ a.s.
  With minor modifications the argument of Loynes also applies to the Rost solution discussed below as well as to a number of further classical embeddings presented in Section \ref{sec:embeddings-abundance} below.
\end{remark}
In Section~\ref{sec:root-rost-embeddings} we will prove generalizations of Theorem~\ref{RootThm} which admit similar conclusions in $\R^d$ and for general initial distributions.

\subsection{The Rost embedding} \label{sec:rost-embedding}

\begin{figure}
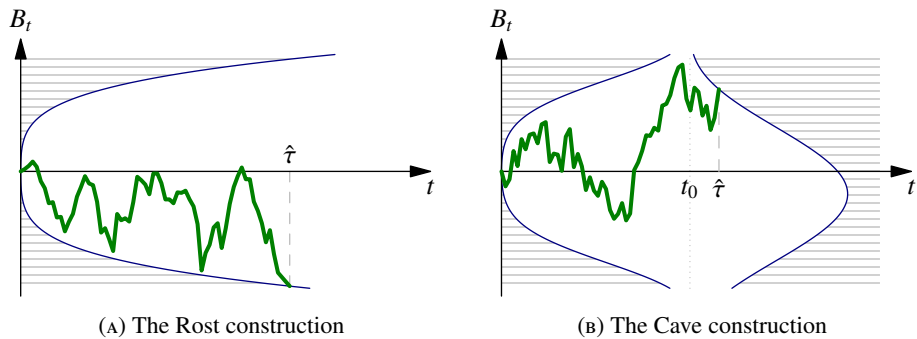

\begin{subfigure}{.5\textwidth}
\centering
\begin{asy}[width=0.9\textwidth]
  import graph;
  import stats;
  import patterns;

  // import settings;

  // gsOptions="-P"; 

  // Construct a Browniam motion of time length T, with N time-steps
  int N = 100;
  //int N = 300;
  real T = 1.6*1.6;
  real dt = T/N;
  real B0 = 0;

  real sig = 0.7;

  real xmax = 0.8;
  real xmin = -0.8;

  real tmax = (xmax-xmin)*1.6;

  real[] B; // Brownian motion
  real[] t; // Time

  path BM;

  // Seed the random number generator. Delete for a "random" path:
  srand(91);

  B[0] = B0;
  t[0] = 0;

  BM = (t[0],B[0]);

  // Define a barrier

  real R(real y) {return (y**4)*5.5 + (y**2)/6+(y**2)*(y+0.5)**2*0.19;}

  int H = N+1;
  int H2 = N+1;
  int BMstop;
  int BMstop2 = N+2;

  for (int i=1; i<N+1; ++i)
  {
    B[i] = B[i-1] + sig*Gaussrand()*sqrt(dt);
    t[i] = i*dt;

    if ((H==N+1)&&(t[i]<R(B[i])))
    {
	H = i;
	BMstop = length(BM);
       while (R(B[i]) > (i+2)*dt)
       {
         B[i] = (0.95*B[i]+0.05*B[i-1]);
       }
       BM = BM--(R(B[i]),B[i]);
    }
    else
    {
      BM = BM--(t[i],B[i]);
    }

  }

  if (H==N+1)
  BMstop = length(BM);

  pen p = deepgreen + 1.5;
  pen p2 = lightgray + 0.25;
  //pen p2 = mediumgray + 1;

  //if (H<N+1)
  //draw(subpath(BM,BMstop,BMstop2),p2);

  pair tau = point(BM,BMstop+1);

  pen q = black + 0.5;

  real eps1 = 0.05;
  real eps2 = 0.15;

  path barrier = (graph(R,identity,xmin+eps1,xmax-eps1)--(0,xmax-eps1)--(0,xmin+eps1)--cycle);

  add("hatch",hatch(1mm,W,mediumgray));

  fill(barrier,pattern("hatch"));

  draw(graph(R,identity,xmin+eps1,xmax-eps1),NW,deepblue+0.5);

  draw((tau.x,0)--tau,mediumgray+dashed);
  label("$\hat\tau$",(tau.x,0),(0,1));

  draw((0,xmin)--(0,xmax+eps1),q,Arrow);
  draw((0,0)--((T+eps1),0),q,Arrow);

  draw(subpath(BM,0,BMstop+1),p);

  label("$t$",(T+eps1,0),S);
  label("$B_t$",(0,(xmax+eps1)),(0,1));

  //label("$B_{\Rt}$",(3.5,0.5),UnFill(0.5mm));

  //label("$D_{\Rt}$",(2.2,1.35),UnFill(0.5mm));
\end{asy}

\caption{The Rost construction}
\end{subfigure}%
\begin{subfigure}{.5\textwidth}
\centering
\begin{asy}[width=0.9\textwidth]
  import graph;
  import stats;
  import patterns;

  // import settings;

  // gsOptions="-P"; 

  // Construct a Browniam motion of time length T, with N time-steps
  int N = 100;
  //int N = 300;
  real T = 1.6*1.6;
  real dt = T/N;
  real B0 = 0;

  real sig = 0.7;

  real xmax = 0.8;
  real xmin = -0.8;

  real tmax = (xmax-xmin)*1.6;

  real t0 = 1.2;

  real[] B; // Brownian motion
  real[] t; // Time

  path BM;

  // Seed the random number generator. Delete for a "random" path:
  srand(89);

  B[0] = B0;
  t[0] = 0;

  BM = (t[0],B[0]);

  // Define a barrier

  real R1(real y) {return t0 + ((y-1)**4)*((y+1.2)**2)*((y+1.4)**2)/3;}
  real R2(real y) {return (y**2)*exp(-3*y**2)/2 + t0*(1-exp(-6*y**4));}

  int H = N+1;
  int H2 = N+1;
  int BMstop;
  int BMstop2 = N+2;

  for (int i=1; i<N+1; ++i)
  {
    B[i] = B[i-1] + sig*Gaussrand()*sqrt(dt);
    t[i] = i*dt;

    if ((H==N+1)&&(t[i]>=R1(B[i]))||((H==N+1)&&(t[i]<=R2(B[i]))))
    {
	H = i;
	BMstop = length(BM);
       while (R1(B[i])<=t[i-1])
       {
         B[i] = 0.75*B[i]+0.25*B[i-1];
       }
       BM = BM--(R1(B[i]),B[i]);
    }
    else
    {
      BM = BM--(t[i],B[i]);
    }

  }

  if (H==N+1)
  BMstop = length(BM);

  pen p = deepgreen + 1.5;
  pen p2 = lightgray + 0.25;
  //pen p2 = mediumgray + 1;

  //if (H<N+1)
  //draw(subpath(BM,BMstop,BMstop2),p2);

  pair tau = point(BM,BMstop+1);

  pen q = black + 0.5;

  real eps1 = 0.05;
  real eps2 = 0.15;

  path barrier1 = (graph(R1,identity,xmin+eps1,xmax-eps1)--(tmax-eps2,xmax-eps1)--(tmax-eps2,xmin+eps1)--cycle);
  path barrier2 = (graph(R2,identity,xmin+eps1,xmax-eps1)--(0,xmax-eps1)--(0,xmin+eps1)--cycle);

  add("hatch",hatch(1mm,W,mediumgray));

  fill(barrier1,pattern("hatch"));
  fill(barrier2,pattern("hatch"));

  draw((t0,xmin+eps1)--(t0,xmax-eps1),mediumgray+dotted);
  label("$t_0$",(t0,0),S);

  draw(graph(R1,identity,xmin+eps1,xmax-eps1),NW,deepblue+0.5);
  draw(graph(R2,identity,xmin+eps1,xmax-eps1),NW,deepblue+0.5);

  draw((0,xmin)--(0,xmax+eps1),q,Arrow);
  draw((0,0)--((T+eps1),0),q,Arrow);

  draw((tau.x,0)--tau,mediumgray+dashed);
  label("$\hat\tau$",(tau.x,0),S);

  draw(subpath(BM,0,BMstop+1),p);

  label("$t$",(T+eps1,0),S);
  label("$B_t$",(0,(xmax+eps1)),(0,1));

  //label("$B_{\Rt}$",(3.5,0.5),UnFill(0.5mm));

  //label("$D_{\Rt}$",(2.2,1.35),UnFill(0.5mm));
\end{asy}
\caption{The Cave construction}
\end{subfigure}
\caption{The barriers corresponding to the Rost and Cave embeddings}
\end{figure}

 A set $\mathcal R \subseteq \R_+\times \R$ is an \emph{inverse barrier} if $(s,x)\in \mathcal R$ and  $s > t$ implies that $(t,x)\in \mathcal R$. 
It has been shown by Rost \cite{Ro76} that under the condition $\mu(\{0\})=0$ there exists  an inverse barrier  such that the corresponding hitting time (in the sense of \eqref{RootType}) solves the Skorokhod problem. It is not hard to see that without this condition some additional randomization is required.
We derive this using an argument almost identical to the one above.

\begin{theorem}\label{RostThm} Suppose $\mu(\{0\}) = 0$. Let $\gamma(f,t)= h(t)$, where $h:\R_+\to \R_+$ is a strictly \emph{concave} function such that \eqref{IntPri} is well posed. Then a minimizer $\hat \tau$ of \eqref{IntPri} exists, and moreover for any minimizer $\hat \tau$, there exists an inverse barrier $\mathcal R$ such that $\hat \tau=\inf\{ t \geq 0 : (t,B_t)\in\mathcal R\}$. In particular the Skorokhod embedding problem has a solution which is the hitting time of an inverse-barrier.
\end{theorem}
\begin{proof} Our proof follows closely the proof of Theorem~\ref{RootThm}. In particular, Steps 1 and 2 can be carried out almost verbatim to get an optimizer $\hat \tau$ and a \emph{$\gamma$-monotone} set $\Gamma\subseteq S$ such that $\P(((B_t)_{t\leq\hat\tau},\hat \tau)\in \Gamma)=1$. By concavity of $h$, the set of stop-go pairs is now given by 
$$\BP=\{((f,s),(g,t))\in S\times S: f(s)=g(t), s<t\}.$$
We remove all paths $(f,s)$ with $f(s)=0$ from $\Gamma$, as $\mu(\{0\})=0$ this does not alter the full support property (or the $\gamma$-monotone property).
Next we define inverse barriers by
\begin{align*}
\mathcal R_\op:=\{(s,x):\exists (g,t)\in\Gamma, g(t)=x, s< t \},\\
\mathcal R_\cl:=\{(s,x):\exists (g,t)\in\Gamma, g(t)=x, s\leq t\}.
\end{align*}
 Denoting the respective hitting times by $\tau_{\op}$ and $\tau_{\cl}$ the argument familiar from the Root case yields 
 $\tau_\cl \leq \hat \tau\leq \tau_\op$ a.s.\ and
it remains to  show $\tau_\cl = \tau_\op$ a.s.\  
 The argument is slightly more involved than in the Root case but again entirely probabilistic:

We define  $b(t) := \inf\{x > 0: (t,x) \in \mathcal{R}_\cl \}$, $c(t) := \sup\{x<0: (t,x) \in \mathcal{R}_\cl\}$ and note that 
$$\ \inf\{ t > 0: B_t \not\in (c(t),b(t))\}\leq \tau_{\cl} \leq \tau_{\op} \leq \inf\{ t > 0: B_t \not\in [c(t),b(t)]\}.$$ 
Concentrating on the function $b$, we have for $\eps > 0$
$$ \underbrace{\inf\{t>0: B_t\geq b(t)\}}_{=: \sigma_b} \leq \underbrace{\inf\{t>0: B_t> b(t)\}}_{=: \sigma_b^+} \leq \underbrace{\inf\{t>0: B_t-\eps t\geq b(t)\}}_{=: \sigma_b^\eps} .
$$
By Girsanov's Theorem, $\lim_{\eps\to 0} \P(\sigma^\eps_b \leq t) = \P(\sigma_b \leq t)$ for each $t\in \R_+$ hence $ \sigma_b^+= \sigma_b$ a.s. 

Arguing likewise on $c$, we obtain $\tau_\cl = \tau_\op$ a.s.
\end{proof}

As in the case of the Root embedding we obtain that the minimizer of $\E[h( \tau)]$  is unique.

\subsection{The cave embedding}
In this section we give an example of a new embedding that can be derived from Theorem~\ref{GlobalLocal}. It can be seen as a unification of the Root and Rost embeddings.  A set $\mathcal R \subseteq \R_+\times \R$ is a \emph{cave barrier} if there exists $t_0\in\R_+$, an inverse barrier $\mathcal R^0\subseteq [0,t_0]\times\R$ and a barrier $\mathcal R^1 \subseteq [t_0,\infty)\times \R$ such that $\mathcal R=\mathcal R^0\cup \mathcal R^1.$ We will show that there exists a cave barrier such that the corresponding hitting time (in the sense of \eqref{RootType}) solves the Skorokhod problem.  We derive this using an argument similar to the one above:

Fix $t_0\in\R$ and pick a continuous function $\phi:\R_+\to [0,1]$ such that
\begin{itemize}
 \item $\phi(0)=0, \lim_{t\to\infty}\phi(t)=0, \phi(t_0)=1$
\item $\phi$ is strictly concave on $[0,t_0]$ 
\item $\phi$ is strictly convex on $[t_0,\infty)$.
\end{itemize}
It follows that $\phi$ is strictly increasing on $[0,t_0]$ and
strictly decreasing on $[t_0,\infty)$.  
\begin{theorem}[Cave embedding]\label{CaveEmb}
Suppose $\mu(\{0\}) = 0$. Let $\gamma(f,t)= \phi(t)$. Then a minimizer $\hat \tau$ of \eqref{IntPri} exists, and moreover for any minimizer $\hat \tau$, there exists a cave barrier $\mathcal R$ such that $\hat \tau=\inf\{ t \geq 0 : (t,B_t)\in\mathcal R\}$. In particular the Skorokhod embedding problem has a solution which is the hitting time of a cave barrier.
\end{theorem}
Since this construction does not already appear in the literature, we  emphasize that the result remains true for  integrable (centered) measures $\mu$ (see Section 7).
\begin{proof}[Proof of Theorem \ref{CaveEmb}]
Note that since $\phi$ is bounded, the problem \eqref{IntPri} is well posed. Following the steps of the proofs of Theorems~\ref{RootThm} and \ref{RostThm}, we find an optimizer $\hat \tau$ and a $\gamma$-monotone set $\Gamma\subseteq S$ such that $\P(((B_t)_{t\leq\hat\tau},\hat \tau)\in \Gamma)=1$. 
The set of stop-go pairs is given by 
\begin{align*}
 \BP=\{((f,s),(g,t))\in S\times S: f(s)=g(t); s < t\leq t_0 \text{ or } t_0\leq t < s\}.
\end{align*}
Indeed, for $s<t\leq t_0$ and any $(h,r)\in S$ we have
\begin{align*}
\gamma ((f\oplus h, s+r)) + \gamma((g,t))\ & >\  \gamma ((f, s)) + \gamma((g\oplus h,t + r))\\
\Leftrightarrow \quad \phi(s+r)-\phi(s)\ &>\ \phi(t+r)-\phi(t)
\end{align*}
which holds iff $t\mapsto\phi(t+r)-\phi(t)$ is strictly decreasing on $[0,t_0]$ for all $r>0.$ If $t+r,t\in[0,t_0]$ this follows from concavity of $\phi$. In the case that $t\leq t_0, t+r>t_0$ this follows since $\phi'$ is strictly positive on $[0,t_0)$ and strictly negative on $(t_0,\infty).$ The case $t_0\leq t<s$ can be established similarly.

Then, we define an `open' cave barrier by$$ \mathcal R_{\textsc{op}}^0:=\{(t,x): \exists (f,s)\in \Gamma, t<s\leq t_0\},\quad \mathcal R_{\textsc{op}}^1:=\{(t,x): \exists (f,s) \in \Gamma, t_0\leq s < t\} $$
and $\mathcal R_{\textsc{op}}:=\mathcal R_{\textsc{op}}^0 \cup \mathcal R_{\textsc{op}}^1$ (resp.\ a `closed' cave barrier where we allow $t\leq s$ and $s\leq t$ in $\mathcal R_{\textsc{cl}}^0$ and $\mathcal R_{\textsc{cl}}^1$ resp.). We denote the corresponding hitting time by $\tau_{\mathcal R_{\textsc{op}}}=\tau_{\mathcal R_{\textsc{op}}^0}\wedge \tau_{\mathcal R_{\textsc{op}}^1}$ (resp. $\tau_{\mathcal R_{\textsc{cl}}}$). 

By the same argument as for the Root and Rost embeddings it then follows that  $\tau_{\mathcal R_{\textsc{cl}}}\leq \hat\tau\leq \tau_{\mathcal R_{\textsc{op}}}$ a.s.\ and also that 
 $\tau_{\mathcal R_{\textsc{cl}}}=\tau_{\mathcal R_{\textsc{op}}}$  a.s., proving the claim.
\end{proof}

\subsection{Remarks}
In Section~\ref{sec:root-rost-embeddings} we will show that the arguments above can be adapted to prove the existence of Rost and Root embeddings in a more general setting. Specifically, in Sections \ref{sec:embeddings-abundance} and \ref{sec:feller} we will show that this approach generalizes to a multi-dimensional setup and (sufficiently regular) Markov processes. In the case of the Root embedding it does not matter for the argument whether the starting distribution is a Dirac in $0$ as in our setup or a more general distribution $\lambda$. For the Rost embedding a general starting distribution is slightly more difficult. In the case where $\lambda$ and $\mu$ have common mass, then it may be the case that $\proj_{\R_+}(\mathcal{R}_\cl \cap (A \times \R_+)) = \{0\}$ for some set $A$ --- that is, all paths which stop at $x \in A$ do so at time 
zero.  In this case it is possible that $\hat\tau < \tau_\op$ when the process starts in $A$, and in general, some proportion of the paths starting on $A$ must be stopped instantly. As a result, in the case of general starting measures, independent randomization is necessary. In the Rost case, it is also straightforward to compute the independent randomization which preserves the embedding property.
 
Other recent approaches to the Root and Rost embeddings can be found in \cite{GaMiOb14,ObRe13,CoPe12,CoWa12}. These papers largely exploit PDE techniques, and as a consequence, are able to produce more explicit descriptions of the barriers, however the methods tend to be highly specific to the problem under consideration.
 
\section{Preliminaries on stopping times and filtrations}
\label{sec:prel-stopp-times}

A key feature of this article is that we are taking a non-standard perspective on stopping times; 
the main purpose of this section is to provide a convenient framework.
To this end, we need to discuss connections between common notions defined on an arbitrary probability space and their related notions defined on the canonical path space $\CRo$ and the space $S$. We then see (by Lemma \ref{lem:equivOpt}, Theorem \ref{thm:equiv RST}) that in the context of our optimization problem, rather than studying the class of all possible stopping times, we can equivalently focus on \emph{randomized} stopping times on the canonical space. These can be characterized in various equivalent terms (cf.~Theorem \ref{thm:equiv RST}); e.g.\ viewing them as measures on $\CRo\times \R_+$ is useful to establish compactness results while the representation through `increasing' functions on $S$ is necessary for the manipulations of stopping times which we need to consider in the proof of the monotonicity principle, Theorem \ref{GlobalLocal}, in Section 5.
Finally, we shall consider the set of `joinings' which can be interpreted as a type of coupling between a randomized stopping time and an abstract probability measure. This is an important ingredient in the proofs of Theorem \ref{DualityIntro} and Theorem \ref{GlobalLocal}.

\subsection{Spaces and Filtrations}

We will primarily consider the space $\CRo$ of continuous functions on $\R_+$ starting at the value $0$, with the topology of uniform convergence on compact sets. The elements of $\CRo$ will be denoted by $\omega$. We denote the canonical process on $\CRo$ by $(B_t)_{t\geq 0}$, i.e.\ $B_t(\omega)=\omega_t.$ We denote the Wiener measure by $\W$. As explained above we consider the set $S$ of all continuous functions defined on some initial segment $[0,s]$ of $\R_+$ and starting with value $0$; we will denote the elements of $S$ by $(f,s)$ and $ (g,t)$. The set $S$ admits a natural partial ordering; we say that $(g, t)$ extends $(f,s)$ if $t\geq s $ and the restriction $ g_{\llcorner[0,s]}$ of $g$ to the interval $[0,s]$ equals $f$.  We consider $S$ with the topology induced by the metric
\begin{align}\textstyle\label{STop}
d_S((f,s),(g,t)):=\max\big( t-s, \sup_{0\leq
    u\leq s}|f(u)-g(u)|, \sup_{s\leq u\leq t} |g(u)-f(s)| \big),
\end{align}
for $(f,s),
(g, t)\in S, s\leq t$. Equipped with this topology, $S$ is a Polish space.

For our arguments it will be important to be precise about the relationship between the sets ${\CRo} \times \R_+$ and $S$. We therefore discuss the underlying filtrations in some detail.

We consider two different filtrations on the Wiener space ${\CRo}$, the canonical or natural filtration $\F^0=(\F_t^0)_{t\in\R_+}$  as  well as  its usual augmentation $\F^a=(\F^a_t)_{t\in\R_+}$.
As Brownian motion is a continuous Feller process, all right-continuous $\F^a$-martingales are continuous (\cite[Theorem~VI. 15.4]{RoWi00}) and hence all $\F^a$-stopping times are predictable and the $\F^a$-optional and $\F^a$-predictable $\sigma$-algebras coincide (\cite[Corollary IV 5.7]{ReYo99}). By \cite[Theorem\ IV.\ 97, Rem.\ IV.\ 98]{DeMeA} we also have that the $\F^0$-predictable, $\F^0$-optional and $\F^0$-progressive $\sigma$-algebras coincide because ${\CRo}$ is the set of \emph{continuous} paths. Moreover, we will use the following result.  
\begin{theorem}\label{indistinguishable predictable}
\comment{DC (3) We have change the wording of the Theorem. The first part is stated in \cite[Theorem IV. 78]{DeMeA}, the second assertion is derived below it.  }
Let $(\Omega,\G,(\G_t)_{t\in\R_+},\P)$ be a filtered probability space and let $\G^a $ be the usual augmentation of the filtration $\G$. 
\begin{enumerate}
 \item If $\tau$ is a predictable time wrt $\G^a$, then there exists a
  predictable time $\tau'$ wrt $\G$ such that $\tau=\tau'$ a.s. For every $\G^a$-predictable process $(X_t)_{t\in\R_+}$
  there is a $\G$-predictable process $(X_t')_{t\in\R_+}$ which is
  indistinguishable from $(X_t)_{t\in\R_+}.$ 
\item {\color{black} If $(A_t)_{t\in \R_+}$ is an increasing right-continuous $\G^a$-predictable process there is an increasing  right-continuous $\G$-predictable process $(A_t')_{t\in\R_+}$ (possibly assuming the value $+\infty$) which is
  indistinguishable from $(A_t)_{t\in\R_+}$.} \end{enumerate}
\end{theorem}
\begin{proof}
For   Statement (1) we refer to \cite[Theorem IV. 78]{DeMeA} and the comments directly afterwards.  
To prove  statement (2), let $(A_t)_{t\in \R_+}$ be an increasing right-continuous $\G^a$-predictable process. Arguing on $(\frac2\pi\arctan (A_t-A_0))_{t\in \R_+}$, we may assume that $A$ takes values in $[0,1]$. 

We use an extension of the filtered probability space denoted $(\ol{\Omega},\ol{\G},(\ol{\G}_t)_{t \ge 0},\ol{\P})$, where  we take $\ol{\Omega} = \Omega \times [0,1]$, $\ol{\G} = \G \otimes \mathcal{B}([0,1]), \ol{\P}(D_1\times D_2) = \P(D_1) \Leb(D_2)$, and set $\ol{\G}_t= \G_t \otimes \mathcal{B}([0,1])$ and let $\ol{\G}^a$ be its usual augmentation. Here, $\Leb$ denotes Lebesgue measure. Abusing notation we also write $A$ for the mapping $(\omega,x,t) \mapsto A_t(\omega)$ on $ \ol \Omega \times \R_+$. 

Set $Y(\omega, x):=x$. Then $A-Y$ is  $\ol{\G}^a$-predictable and right-continuous, hence
$$\rho(\omega, x):=\inf\{t\geq 0: A_t(\omega)\geq x\}=
 \inf\{ t\geq 0 : A_t(\omega)- Y(\omega, x)\geq 0\}$$ is a $\ol{\G}^a$-predictable stopping time by the (predictable) Debut theorem. Moreover $$A_t(\omega)= \inf\{ x\geq 0: \rho(\omega, x) > t\}= 1- \Leb\{ x: \rho(\omega,x)> t \}.$$
  Pick a $\ol{\G}$-predictable stopping time $\rho'$ such that $\rho'=\rho$, $\ol{\P}$-a.s.\ and set 
$$A_t'(\omega):=1-\Leb\{ x: \rho'(\omega,x)> t \}.$$ 
 Then $A'(\omega)$ is increasing and right-continuous for each $\omega$. For each $t$
 $$\Leb\{ x: \rho'(\omega,x)> t \} = \Leb\{ x: \rho(\omega,x)> t \}$$
 for $\P$-a.a.\ $\omega$, hence $A'$ is a version of $A$. By
 right-continuity, $A$ and $A'$ are indistinguishable. 
 Predictability of $\rho'$ asserts that (using obvious abbreviations)
$$\{(\omega,x,t): \rho'(\omega, x)> t\}\in \mathsf{pred}_{\ol{\G}}=\mathsf{pred}_{{\G}}\otimes \mathcal{B}_{[0,1]}. $$
Hence $(\omega,t) 
\mapsto A'_t(\omega)$ is  $\mathsf{pred}_{{\G}}$-measurable.
\end{proof}

The message of Theorem \ref{S2F} below is that a process $(X_t)_{t\in\R_+}$ is $\F^0$-optional (and hence also $\F^0$-predictable in our setup) iff $X_t(\omega)$ can be calculated from the restriction $\omega_{\llcorner [0,t]}$. We introduce the mapping
\begin{align}\label{TheRestr} 
  r:{\CRo} \times \R_+ \to S, \quad r(\omega,t)= (\omega_{\llcorner [0,t]},t). 
\end{align}
We note that the topology on $S$ introduced in \eqref{STop} coincides with the final topology induced by the mapping $r$; moreover $r$ is a continuous open mapping.

The following result is a particular case of \cite[Theorem IV. 97]{DeMeA} (in somewhat different notation).
\begin{theorem}\label{S2F} $\F^0$-optional sets and functions on ${\CRo} \times \R_+$ correspond to Borel measurable sets and functions on $S$. More precisely we have: 
 \begin{enumerate}
 \item A set $D\subseteq {\CRo}\times \R_+$ is $\F^0$-optional iff $D=r^{-1}(A)$ for some Borel set $A\subseteq S$.
\item A process $X=(X_t)_{t\in\R_+}$ is $\F^0$-optional iff $X=H\circ r$ for some Borel measurable $H:S\to\R$.  \end{enumerate}
\end{theorem}

The mapping $r$ is not a closed mapping: it is easy to see that there exist closed sets in ${\CRo} \times \R_+$ with a non-closed image under $r$. However this does not happen for closed optional sets: it is straightforward that an $\F^0$-optional set $A\subseteq {\CRo} \times \R_+$ is closed iff the corresponding set $r(A)$ is closed in $S$.

\begin{definition}\label{def:verycontinous} If $X$ is an $\F^0$-optional process we write $X^S$ for the unique function $S\to \R$ satisfying $X=X^S\circ r$.
  We say that an optional process $X$ is $S$\!-continuous (resp.\ $S$\!-lsc) if the corresponding function $X^S: S \to \R$ is continuous (resp.\ lsc).
\end{definition}
It is trivially true that an $S$\!-continuous process is continuous in the usual pathwise sense.  The converse is not generally true --- consider the case where $X_t(\omega)=\mbox{sign}(\omega(1))(t-2)_+$.  This is a continuous, optional process, however the corresponding function $X^S$ is not a continuous mapping from $S$ to $\R$. Other examples arise from functions connected to the local time of Brownian motion, cf.\ Section \ref{ValoisSection}.

\begin{definition}\label{EAverage} 
For a measurable $X:{\CRo}\to \R$  which is bounded or positive we set
  \comment{DC(8) Changed according to the suggestion.} 
\begin{align}\label{EAverageEQ}
\E[X|\F_t^0](\omega):= X^M_t(\omega):=\textstyle \int X((\omega_{\llcorner[0,t]})\oplus \omega')\, d\W(\omega').
\end{align}
\end{definition}
Clearly, \eqref{EAverageEQ} defines an $\F^0_t$-measurable function which is a version of the classical conditional expectation; subsequently, it will be useful to have this function defined for \emph{all} $\omega$. In accordance with Definition \ref{def:verycontinous} we write $X^{M,S}$ for the function satisfying $X^M =  X^{M,S}\circ r$.
\begin{proposition}\label{prop:very cont} \comment{DC(9) We replaced the first $X_t$ by $X^M_t$, got rid of the second sentence and added  line about $X^M_\infty$.} 
 Let $X\in C_b({\CRo})$. Then $X^M_t$ is an $S$\!-continuous martingale, 
   $X^M_\infty=\lim_{t\to \infty} X^M_t$ exists and equals $X$.
\end{proposition}

\begin{proof} 
Note that
$X^{M,S}(f,s)=\textstyle \int X^{(f,s)\oplus}(\omega)\, \W(d\omega)$ for $(f,s)\in S$.
Also, 
$(f_n,s_n)\to(f,s)$ implies $f_n\oplus \omega \to f\oplus \omega$  for  $\omega\in {\CRo}$  and,  by continuity of $X$, 
$X^{(f_n,s)\oplus}(\omega)\to X^{(f,s)\oplus}(\omega)$. 
Since $X$ is bounded,  dominated convergence implies 
$X^{M,S}(f_n,s_n)\to X^{M,S}(f,s).$
\end{proof} 
 For $X\in C_b({\CRo})$,  $X^M$ is a martingale with continuous paths and hence satisfies  the optional stopping theorem. Using the functional monotone class theorem, we see that the optional stopping theorem holds for $X^M$ for all bounded measurable $X: \CRo\to \R$.
Also one can prove that $X^M$ has almost surely continuous paths, even if $X$ itself was not continuous, but we will not use this fact.

\subsection{Randomized stopping times}\label{RSTTopology}

Working on the probability space $(\CRo, \W)$, a stopping time $\tau$ is a mapping which assigns to each path $\omega$ the time $\tau(\omega) $ at which the path is stopped. If the stopping time depends on external randomization, then we may consider a path $\omega$ which is not stopped at a single point $\tau(\omega)$, but rather that there is a sub-probability measure $\tau_\omega$ on $\R$ which represents the probability that the path $\omega$ is stopped at a given time, conditional on observing the path $\omega$. The aim of this section is to make this idea precise, and to establish connections with related properties in the literature. Specifically, the notion of a \emph{randomized stopping time} has previously appeared in e.g.\ \cite{BaCh77,meyer_convergence_1978,Ro71}.

Subsequently we will identify randomized stopping times as a subset of
 the well studied $\mathbf{P}$-measures:  
 A finite measure $\xi$ on $\CRo\times\R_+$ is a $\mathbf{P}$\emph{-measure} (wrt $\W$) if it does not charge any $\W$-evanescent set. 
A basic result of  Dol\'eans  \cite{Doleans68}  is the following 
\begin{theorem}[cf.\ {\cite[Theorem VI 65]{DeMeB}}]\label{thm:Doleans} 
  A finite measure $\xi$ on $\CRo\times\R_+$ is a $\mathbf{P}$-measure iff there exists a right-continuous increasing process $A$, $\E[A_\infty]<\infty$ such that for all bounded and measurable processes $X$
$$\xi(X)= \textstyle\E\big[\int
 X_s\,dA_s\big].  $$ 
Here the process $A$ is unique up to evanescence. 
\end{theorem}

We will be particularly interested in the following subset of $\mathbf{P}$-measures:
\begin{align*}\mathsf{M}:=&\,\{\xi\in\mathcal P^{\leq 1}( \CRo\times \R_+): \xi(d\omega,dt)=\xi_\omega(dt)\W(d\omega), \xi_\omega\in\mathcal P^{\leq 1}(\R_+) \mbox{ for $\W$-a.e. } \omega\}\\=&\,\{\xi\in\mathcal P^{\leq 1}( \CRo\times \R_+): \proj_{\CRo}(\xi)\leq \W\},
\end{align*}
where $(\xi_\omega)_{\omega\in\CRo} $ is a disintegration of $\xi$ in the first coordinate $\omega\in \CRo$. 
We equip $\mathsf{M}$ with the weak topology induced by the continuous bounded functions on $\CRo\times \R_+$.
Clearly any $\xi\in\mathsf{M}$ is a $\mathbf{P}$-measure with corresponding increasing process $A^\xi_\omega(t)=\xi_\omega([0,t])$ being the cumulative distribution function of $\xi_\omega.$

\begin{definition}[Randomized stopping times]\label{def:RST}
 A measure $\xi\in\M$ is called a randomized stopping time, written $\xi \in \RST$, iff the associated increasing process $A$ is optional. 
\end{definition}

Below, it will sometimes be convenient to represent randomized stopping times on  an extension of the space $(\CRo, \F^0, (\F^0_t)_{t\geq 0},\W)$: we will consider  $(\oCRo,\oF,(\oF_t)_{t \ge 0},\oW)$,  
where $\oCRo = \CRo \times [0,1]$, $ \oW(A_1\times A_2) = \W(A_1) \Leb(A_2)$ (where $\Leb$ denotes Lebesgue measure), $\oF$ is the completion of $\F^0 \otimes \mathcal{B}([0,1])$, and   $\oF_t$ the usual augmentation of $(\F_t^0 \otimes \mathcal{B}([0,1]))_{t \ge 0}$.  We will write $\bar B=(\bar B_t)_{t\ge 0}$ for the process given by $\bar B_t(\omega,u)=\omega_t.$ Observe that if $Y_t(\omega,u) = u$, then $(\bar B_t, Y_t)$ is (trivially) a continuous Feller process, and hence by the same arguments as above, the $\oF$-predictable and $\oF$-optional $\sigma$-algebras coincide. 

Randomized stopping times play a key role in this paper; depending on the respective context, the following different characterizations will be useful: 
\begin{theorem}\label{thm:equiv RST}
Let $\xi\in \mathsf{M}$. Then the following are equivalent:
\begin{enumerate}
\item There is a Borel function $A:S\to[0,1]$ such that the process $A\circ r$ is right-continuous increasing and
\begin{align}\label{AlternativeRep} \xi_\omega([0,s]):=A\circ r(\omega,s)\end{align}
defines a disintegration of $\xi$ wrt to $\W$. 
\item We have $\xi\in \RST$, i.e.\ given a  disintegration $(\xi_\omega)_{\omega\in\CRo}$ of $\xi$, the random variable   $ \tilde A_t(\omega)=\xi_\omega([0,t])$ is $\F^a_t$-measurable for all $t\in \R_+$. 
\item  For all $f\in C_b(\R_+)$ supported on some  $ [0,t]$, $t\geq 0$ and all $g\in C_b(\CRo)$ 
\begin{align}\label{CheckMeasurability}\textstyle\int f(s) (g-\E[g|\F_t^0])(\omega) \, \xi(d\omega, ds)=0\end{align}
\item On the probability space $(\oCRo,\oF,(\oF_t)_{t \ge 0},\oW)$, the random time
  \begin{equation}
    \label{eq:rhodefn}
     \rho(\omega,u) :=\inf \{ t \ge 0 : \xi_\omega([0,t]) \ge u\}
  \end{equation}
  defines an $\oF$-stopping time.
\end{enumerate}
\end{theorem}
\begin{proof}
The equivalence of (1) and (2) follows directly from Theorems \ref{indistinguishable predictable}, \ref{S2F} and \ref{thm:Doleans}.

It is straightforward to deduce (4) from (1). To see that (4) implies (2), consider for $t\geq 0, \omega \in \CRo$ 
\[\tilde A(\omega,t):=\textstyle\int_0^1 \1_{[0,t]}(\rho(\omega,u))  \, du.  \]

To show that (2) and (3) are equivalent, we first note that (2) is equivalent to requiring that $X_t(\omega) := \xi_\omega(f)$ is $\F^a_t$ measurable whenever $f \in C_b(\R_+)$ is supported on $[0,t]$. However we can express this measurability in a different fashion. Note that a bounded Borel function $h$ is $\F_t^a$-measurable iff for all bounded Borel functions $g$ 
\[\E [h (\E[g|\F_t^a]-g)]=\E [h (\E[g|\F_t^0]-g)]\]
vanishes;
of course this does not rely on our particular setup. 
By a functional monotone class argument, for $\F_t^a$-measurability of $X_t$ it is sufficient to check 
that 
\begin{align}\label{MeasurabilityCheck}
\E[X_t (g-\E[g|\F_t^0])]=0
\end{align}
for all $g\in C_b({\CRo})$. 
In terms of $\xi$, \eqref{MeasurabilityCheck} amounts to 
\begin{align*}
0=\E[X_t (g-\E[g|\F_t^0])] \ & =\textstyle\int \, \W (d\omega) \int \xi_\omega(ds) f(s) (g-\E[g|\F_t^0])(\omega)\\
& =\textstyle\int f(s) (g-\E[g|\F_t^0])(\omega) \, \xi(d\omega, ds). \qedhere
\end{align*}
\end{proof}

\begin{remark}\label{r:livelyhood}
\begin{enumerate}\item The function $A$ in \eqref{AlternativeRep} is unique up to indistinguishability (cf.\ Theorem \ref{thm:Doleans}). We will denote this function by $A^\xi$.
\item We will say $\xi\in\RST$ is a \emph{non-randomized stopping time} iff \comment{DC(11) We added that $\xi_\omega$ can also be null (corresponding to $\tau(\omega)=\infty$).} there is a disintegration $(\xi_\omega)_{\omega\in\CRo}$ of $\xi$ such that $\xi_\omega$ is either null (corresponding to a path which is not stopped) or a Dirac-measure (of mass 1) for every $\omega$. Clearly this means that $\xi_\omega= \delta_{\tau(\omega)}$ a.s.\ for some (non-randomized) stopping time $\tau$. $\xi$ is a non-randomized stopping time iff there is a version of $A^\xi$ which only attains the values $0$ and $1$. 
\item We will say $\xi \in \RST$ is a \emph{finite  randomized stopping time} iff $\xi(\CRo \times \R_+) = 1$.
\end{enumerate}
\end{remark}

An immediate  consequence of Theorem \ref{thm:equiv RST} (3) is the following
\begin{corollary} 
 The set $\RST$ is closed wrt the weak topology induced by the continuous bounded functions on ${\CRo}\times \R_+$.  \comment{DC(12) Changed as suggested.}
\end{corollary}

The next lemma implies that optimizing over usual stopping times on a rich enough probability space in \eqref{IntPri} is equivalent to optimizing over randomized stopping times on Wiener space.

\begin{lemma}\label{lem:equivOpt}   Let $B$ be a Brownian motion on some stochastic basis  $(\Omega, \G, (\G_t)_{t\geq 0}, \P)$, let $\tau$ be a $\G$-stopping time and consider
  \begin{align*}\Phi:\Omega\to \CRo\times \R_+, \bar\omega\mapsto
  ((B_t(\bar\omega))_{t\ge 0}, \tau(\bar\omega)).\end{align*}
  Then $\xi:= \Phi(\P)$ is a randomized stopping time
and for any measurable $\gamma:S \to \R$ we have \begin{align}\label{STRep2}\textstyle\int
\gamma((f,s)) \, r(\xi)(d(f,s)) = \E_\P[\gamma((B_{t})_{t \le \tau},\tau)].\end{align}
 If $\Omega$ is sufficiently rich that it supports a uniformly distributed random variable  which is $\G_0$-measurable then for any $\xi \in \RST$, we can  find a $\G$-stopping time $\tau$ such that $\xi= \Phi(\P)$ and  \eqref{STRep2}   holds.
\end{lemma}
\begin{proof}
  Clearly $\xi:=\Phi(\P)\in \M$. Write $(\xi_\omega)_{\omega\in \CRo}$ for a disintegration wrt Wiener measure. We need to show that $\xi_\omega([0,t])$ is $\F_t^a$-measurable. 
Let $g:\CRo\to \R$ be a measurable function. If $h = \E_\W[g|\F_t^a]$, writing $\G_t^a$ for the usual augmentation of $\G$, and noting that $(B_t)_{t\ge 0}$ is also a $(\G_t^a)_{t \ge 0}$-Brownian motion, we have
  \begin{equation*}
    \E_\P[g((B_r)_{r \ge 0})|\G_t^a] = h((B_r)_{r \ge 0}), \quad \mbox{$\P$ - a.s.}
  \end{equation*}
   It then follows that  \begin{align*}
 \textstyle   \int g(\omega) \xi_\omega([0,t]) \, \W(d\omega) 
    & =\textstyle \E_\P[g((B_r)_{r \ge 0}) \1_{\{\tau \le t\}}]\\
    & = \textstyle\E_\P[\E_\P[g((B_r)_{r \ge 0})|\G_t^a]\1_{\{\tau \le t\}}]\\
    & =\textstyle \E_\P[h((B_r)_{r \ge 0}) \1_{\{\tau \le t\}}]
    = \int h(\omega) \xi_\omega([0,t]) \, \W(d\omega).
  \end{align*}
  Hence $\xi_\omega([0,t])$ is $\F_t^a$-measurable as required.

To prove the second part, we observe that by Theorem~\ref{thm:equiv RST} (4), there exists an $\oF$-stopping time $\rho'$ representing $\xi$. Since $\rho'$ is $\oF$-predictable, it follows from Theorem~\ref{indistinguishable predictable} that there exists an almost surely equal $(\F_t^0 \times \mathcal{B}([0,1]))_{t \ge 0}$-stopping time $\rho$. Then we can define a random time on $\Omega$ by $\rho((B_s)_{s \ge 0},Y)$, where $B$ is the Brownian motion, and $Y$ the independent $\G_0$-measurable, uniform random variable. Consider the map $\bar\Phi:\Omega\to\oCRo, \bar\omega\mapsto ((B_t(\bar\omega))_{t\ge 0}, Y(\bar\omega)).$ Since $\rho$ is a $(\F_t^0 \times \mathcal{B}([0,1]))_{t \ge 0}$-stopping time and $\bar\Phi $ is measurable from $ (\Omega, \G_t)$ to $(\oCRo, \F_t^0 \times \mathcal{B} ([0,1]))$, $\rho\circ (B,Y)$ is a $\G$-stopping time.
\end{proof}

\subsection{Randomized stopping times solving the Skorokhod problem and compactness.}

For a finite randomized stopping time $\xi$ and optional $Y:\CRo\times \R_+ \to \R$ which is bounded or positive, define $Y_\xi$ as the push-forward of $\xi $ under the mapping $(t,\omega)\mapsto Y_t(\omega)$ and denote $Y^\xi_t:=Y_{\xi\wedge t}$ for $t\in \R_+$. Considering the representation $\rho$ of $\xi$ on the extended space $\oCRo$ as in $\eqref{eq:rhodefn}$ and writing $\bar Y_t(\omega,u)= Y_t(\omega)$, we then have
 \begin{align}
\bar Y_\rho\sim Y_\xi  \text{ and }  \bar Y^\rho_t \sim Y^\xi_t.
\end{align}
\comment{MB: I just deleted something on DC (18).}
Taking $Y_t = t$ we obtain ${\xi}(T) =  \oE[\rho]$, where $T$ denotes the projection
\begin{align}\label{TimeComponent} T:{\CRo}\times \R_+\to \R_+.\end{align}
Recall that  $\mu $ has mean $0$ and finite second moment
$ \int x^2\, \mu(dx)=:V$. Then the following result follows directly from classical properties of stopping times (e.g. \cite[Corollary~3.3]{Ho11}).

\begin{lemma}\label{MinimalLPT}
Let $\xi \in\RST^1$ with representation $\rho$  on $\oCRo$ as in \eqref{eq:rhodefn}. Assume that $B_\xi = \mu$, i.e.\ $\BB_\rho \sim \mu$.
Then the following are equivalent:
\begin{enumerate}
\item $\xi(T)=\oE [\rho] < \infty$,
\item $\xi(T)=\oE [\rho] = V$,
\item $(\BB_{\rho\wedge t})$ is uniformly integrable.
\end{enumerate} 
\end{lemma}

\begin{definition}
We denote by $\RST(\mu)$ the set of all finite randomized stopping times satisfying the conditions in Lemma \ref{MinimalLPT}. 
\end{definition}

For us it is crucial that \emph{randomized} stopping times have the following property:
\begin{theorem}\label{ComSol}
  The set $\RST(\mu)$ is non-empty and compact wrt the weak topology induced by the continuous and bounded functions on ${\CRo}\times \R_+$.  \comment{DC(12) Changed as suggested.}
\end{theorem}
\begin{proof}

  If $\mu$ is a centered probability then it is not hard to establish that the Skorokhod embedding problem has a solution, e.g.\ one can use the external randomization $u\in [0,1]$ to stop $(\bar B_t(\omega,u))_{t\geq 0}$ once it leaves $(a(u),b(u))$. Choosing $a,b$ carefully we obtain a solution of \eqref{SkoSol}, see e.g.\ \cite[p332]{Ob04} for a detailed account.

By Prokhorov's theorem we have to show that $\RST(\mu)$ is tight and closed. 

\emph{Tightness.}  Fix $\epsilon>0$ and take $R=2V/\eps$. \comment{DC (20) We have changed the definition of $R$ as indicated. It remains to introduce the complement notation.} Then, for any $\xi\in\RST(\mu)$ we have $\xi(T>R)\leq\epsilon/2.$ As ${\CRo}$ is Polish there is a compact set $\tilde K\subseteq {\CRo}$ such that $\W( {\tilde K}\complement)\leq \epsilon/2.$ Set $K:=\tilde K\times [0,R].$ Then $K$ is compact and we have for any $\xi\in\RST(\mu)$
$$ \xi( K\complement) \leq \W({\tilde K}\complement)  + \xi(T>R) \leq\epsilon.$$

\emph{Closedness.}
Take a sequence $(\xi_n)_{n\in\N}$ in $\RST(\mu)$ converging to some $\xi$. Putting $h:{\CRo}\times \R_+\to\R, (\omega,t)\mapsto \omega(t)$ we have to show that $h(\xi)=\mu$ and that $\xi(T)<\infty.$ Note that $h$ is a continuous map. Take any $g\in C_b(\R).$ Then $g \circ h\in C_b({\CRo}\times \R_+)$ and hence
\begin{align*}\textstyle
 \int g\ d\mu = \lim_n \int_{{\CRo}\times \R_+} g \circ h\ d\xi_n = \int_{{\CRo}\times \R_+} g \circ h\ d\xi = \int g\ dh(\xi), 
\end{align*}
thus $h(\xi)=\mu.$ 
Moreover,  $T\wedge N$ is continuous and bounded for each $N\in \N$, hence $\xi (T\wedge N)=\lim_n\xi_n(T\wedge N)\leq V$. As $N$ was arbitrary, it follows that also $\xi (T)\leq V<\infty$.  
\end{proof}

Our use of randomization to achieve compactness of a set of stopping
times has similarities to the work of Baxter and Chacon \cite{BaCh77}.
However their setup is different, and their intended applications are not
connected to Skorokhod embedding.

We close this section with a simple result that connects weak convergence of randomized stopping times with convergence in probability of their representatives on the stochastic basis $(\Omega, \G, (\G_t)_{t\ge 0}, \P)$. First, suppose that $\tau$ is a $\G$-stopping time. Then by the definition of $\xi$ in Lemma~\ref{lem:equivOpt}, it follows that for any measurable and bounded or non-negative $X:\CRo \times \R_+\to\R$ we have $\E_\P[X((B_{t})_{t},\tau)]=\textstyle\int X(\omega,t) \, d\xi(\omega, t)$. Now suppose in addition that the probability space is sufficiently rich to support a uniform $\G_0$-measurable random variable $Y$, independent of the Brownian motion $B$. Recall that if $\xi$ is a finite randomized stopping time and $\rho\circ(B,Y)$ its representative on $\Omega$ given by Lemma~\ref{lem:equivOpt}, then for measurable and bounded or non-negative $X:\CRo \times \R_+\to\R$
\begin{align}\label{DiffRep} \textstyle\int X_t(\omega) \, \xi(d\omega, dt) & =\textstyle \iint  X_{\rho(\omega,u)}(\omega)\,\Leb(du) \W(d\omega) = \E_\P[ X_{\rho((B_t)_{t \ge 0},Y)}((B_t)_{t\ge 0})].
\end{align}
\begin{lemma}\label{prop:LTConv} Let $\xi, \xi_n\in \RST(\mu), n\geq 1$ and denote their representatives on $\Omega$ by $\rho, \rho_n, n\geq 1$. 
 Then $\xi_n \to \xi$ weakly iff $\rho_n \to \rho $ in probability. 
\end{lemma}
\begin{proof} Let $X\in C_b(\CRo \times \R_+)$. By \eqref{DiffRep}
\begin{align*}\textstyle
\int X_t(\omega) \, d(\xi-\xi_n)(\omega, t) = \E_\P \big [ X_{\rho(B,Y)}(B)-X_{\rho_n(B,Y)}(B)\big].
\end{align*}
Considering processes which depend only on the time $t$ but not $B$, i.e.\ $X(z,t)= X(t)$, we obtain that $\xi_n \to \xi$ weakly implies that $\rho_n \to \rho$ in probability. Conversely, if $\rho_n \to \rho$ in probability under $\P$, then also $\rho_n \to \rho$ almost surely along some subsequence of every subsequence. By dominated convergence, $\xi_n\to \xi$ weakly. 
\end{proof}

\subsection{Joinings}\label{s:joinings}

We now add another dimension: we assume that $(\Ys, \nu)$ is some Polish probability space and consider  randomized stopping times where each death of a particle is tagged by an element of $\Ys$. More precisely,  the set of \emph{joinings} \comment{DC (21) Changed as indicated. We have also changed the title of the section accordingly.}
$\TRST(\nu)$
is given by
$$ \Big\{\pi \in\mathcal P^{\leq 1}({\CRo}\times \R_+ \times \Ys): \pr_{{\CRo}\times \R_+}(\pi_{\llcorner {\CRo}\times \R_+ \times D})\in \RST, D\in \mathcal{B}(\Ys), \pr_{\Ys}( \pi) \leq \nu \Big\}.$$
We shall also write $ \TRST^1(\nu)$ for the subset of $\pi\in\TRST(\nu)$ having mass 1.

\begin{remark}\label{ProdPred} 
  Write $\mathsf{pred}$ for the $\sigma$-algebra of $\F^0$-predictable sets in ${\CRo}\times \R_+$.  We call a set $A\subseteq {\CRo}\times \R_+\times \Ys$ predictable if it is an element of $\mathsf{pred}\otimes \mathcal{B}(\Ys)$.  We will say that a function defined on ${\CRo}\times \R_+\times \Ys$ is predictable if it is measurable wrt $\mathsf{pred}\otimes \mathcal{B}(\Ys)$.  As before, predictable subsets of $ {\CRo}\times \R_+\times \Ys$ correspond to measurable subsets of $S\times \Ys$, and similarly for functions.
\end{remark}

\section{The Optimization Problem and Duality}
\label{sec:optim-probl-dual}
\subsection{The Primal Problem}\label{sec:optim-probl-dual-start}
As defined in \eqref{IntPri} in the introduction, our primal problem is to minimize the value corresponding to a function $\gamma:S \to \R$, where the minimization is taken over stopping times of Brownian motion defined on a sufficiently rich probability space. By Lemma \ref{lem:equivOpt}, we obtain an equivalent problem if we take $B$ to be the canonical process on Wiener space $\CRo$ and minimize over all  randomized stopping times, i.e.\ we have
\begin{align}\label{RealPri}  \textstyle P_\gamma= \inf\left\{ \int \gamma\circ r(\omega,t) \, \xi(d\omega, dt): \xi\in \RST(\mu)\right\}.\end{align}
In the following we will mainly work with the technically convenient formulation given in \eqref{RealPri}. It immediately allows us to establish the existence of  optimal stopping times:

\begin{theorem}\label{MinimizerExistsP} Assume that
  $\gamma:S\to \R$ is lsc and bounded from below in the sense that for some constants $a,b,c\in \R_+$
  \begin{align}\label{MartinBound}
-\big( a + b s+ c \max_{r\leq s} B_r^2\Big) \leq   \gamma((B_r)_{r\leq s},s)
   \end{align}
   holds on $\CRo\times\R_+$. Then the functional
 \begin{align}\label{GamFun} \textstyle \xi \mapsto \int_{\CRo\times \R_+} \gamma\circ r(\omega,t)\, \xi(d\omega, dt)\end{align}
is lsc and  \eqref{RealPri} admits a minimizer. 
\end{theorem}
By Lemma \ref{lem:equivOpt}, Theorem \ref{MinimizerExists} is a consequence of this result.
\begin{proof}[Proof of Theorem \ref{MinimizerExistsP} / Theorem \ref{MinimizerExists}] By the Portmanteau theorem, the functional \eqref{GamFun} is lsc if
 $\gamma:S\to\R$ is lsc and bounded from below by a constant.

 For the general case we recall the pathwise version of Doob's inequality (see \cite{AcBePeScTe12})
\begin{align}\label{PathDoob} {\textstyle\max_{r\leq s} B_r^2}\leq \underbrace{\textstyle\int_0^s 4 \max_{t\leq r}|B_t|\, dB_r}_{=:M_s} + {\textstyle4 B_s^2}.\end{align}
We emphasize that we can understand the integral defining $M$ in a \emph{pathwise} fashion. This is possible since $r\mapsto \max_{t\leq r}|B_t|$ is increasing; we refer to \cite{AcBePeScTe12} for details. In fact it is straightforward to show that $M$ is an $S$\!-continuous martingale satisfying $|M_{t}| <2 \max_{r\leq t} B_r^2$.  It follows that $\tilde \gamma(f,s):= \gamma(f,s) + b s+ c (M^S(f,s) + 4 f(s)^2)$ is bounded from below and hence $\xi\mapsto \int \tilde \gamma\, d\xi$ is lsc. As the value of $\int b s+ c (M_s(\omega) + 4 B_s^2(\omega))\, d\xi(\omega, s)$ is the same for any $\xi\in \RST(\mu)$ the functional \eqref{GamFun} is lsc as well.
 \end{proof}
 In Section \ref{sec:feller} below we establish existence of a minimizing stopping time in the case where the measure $\mu$ does not necessarily admit a finite second moment. However we will then replace Assumption \eqref{MartinBound} by the requirement that $\gamma$ is bounded from below. \comment{DC (22) Changed as indicated.}

\subsection{The dual problem}\label{s:robust super-hedging}

The following result implies Theorem \ref{DualityIntro}.

\begin{theorem}\label{HedgingDual} Let $\gamma: S \to \R$ be lsc and  bounded from below in the sense  of \eqref{MartinBound}. Set
$$D_\gamma=\sup\left\{\int \psi(y)\, d\mu(y): \psi \in C(\R),\begin{array}{l}
 \exists \phi, 
\phi\mbox{ is an $S$\!-continuous martingale}, \phi_0 = 0 \\
 \phi_t(\omega) + \psi(\omega(t)) \leq \gamma\circ r(\omega,t),\  (\omega, t)\in \CRo\times \R_+
 \end{array}\!
 \right\}
 $$ where $\phi, \psi$ 
satisfy $|\phi_t| \leq a + bt+c B_t^2 $, $|\psi(y)| \leq a+ b y^2$ for some $a,b,c>0$. 
     Then we have 
\begin{align}\label{HedgingDualEq2} P_\gamma = D_\gamma.\end{align}
\end{theorem}

Using the same argument as in the proof of Theorem \ref{MinimizerExistsP}, we see that it suffices to establish Theorem \ref{HedgingDual} in the case where $\gamma$ is bounded from below.
As usual, one part of the duality relation is straightforward to verify: for $(\phi,\psi)$ satisfying the dual constraint and $\xi\in\RST(\mu)$  we have 
\begin{align*}\hspace{1cm}\textstyle
 \int \psi(y)\ \mu(dy) = \int \psi(\omega(t))\ \xi(d\omega,dt) + \int \phi_t(\omega)\ \xi(d\omega,dt) \leq \int \gamma\circ r(\omega,t) \ \xi(d\omega,dt),\end{align*} hence $D_\gamma\leq P_\gamma$.

We will establish  Theorem \ref{HedgingDual} as a consequence of the following auxiliary   duality result, where
 we write  
 $T$ for the projection map $\CRo\times\R_+\times\R \to \R_+$, $T(\omega,t,y)=t$.

\begin{proposition}\label{super hedging}
  Let $c:\CRo\times \R_+\times \R\to \R \cup \{\infty\}$ be lsc, predictable (cf.\ Remark \ref{ProdPred}) and
  bounded from below. Write $V = \int x^2 \, \mu(dx)$. Then
\begin{align}\label{NewDuality}\tag{$\star$}
 \inf_{\pi  } {\textstyle \int c(\omega,t,y)} \, d\pi(\omega, t,y)=\sup_{(\phi,\psi)} {\textstyle \int \phi \, d\W+\int \psi\, d\mu} .
\end{align}
where the infimum is taken over the set 
\begin{align*}\TRST^{1,V} (\mu):=\{\pi\in \TRST^1(\mu): \pi(T)\leq V\}
\end{align*}
 and the supremum is taken over  $\phi\in C_b(\CRo)$, $\psi\in C_b(\R)$ such that
\begin{align}\label{d^M}\tag{$d^M[c,V]$}
\exists{\alpha\geq 0} \mbox{ s.t. } \phi^M_t(\omega)\!+\! \psi (y)\!-\!\alpha(t\!-\!V) \leq c(\omega, t, y) \mbox{ for } \omega \in \CRo, t\in\R_+, y\in \R.
\end{align}
\end{proposition}
Proposition \ref{super hedging} should be compared to the (formally) very similar classical 
 duality theorem of optimal transport, see e.g.\ \cite[Section~5]{Vi09} for a proof as well as for a discussion of its origin and related literature. 

\begin{theorem}[Monge-Kantorovich Duality]\label{OTDuality}
 Let $(\Xs_i,\mu_i),$ $ i=1,2$ be Polish probability spaces and $c:\Xs_1\times \Xs_2\to \R \cup \{\infty\}$ a  lsc  and bounded from below \emph{cost function}. Then
\begin{align}
 \inf_{\pi  } {\textstyle \int c(x_1,x_2)} \, d\pi(x_1,x_2)=\sup_{(\phi,\psi)} {\textstyle\left( \int \phi \, d\mu_1+\int \psi\, d\mu_2\right)}, 
\end{align}
where the $\inf$ is taken over probabilities $\pi$ on $\Xs_1\times \Xs_2$ satisfying $\proj_{\Xs_1}(\pi)=\mu_1, \proj_{\Xs_2}(\pi)=\mu_2$ 
 and the $\sup$ is taken over  $\phi\in C_b(\Xs_1)$, $\psi\in C_b({\Xs_2})$ satisfying for $x_1\in {\Xs_1}, x_2\in {\Xs_2}$
\begin{align*}
\phi(x_1)+\psi(x_2) \leq c(x_1,x_2).\end{align*}
\end{theorem}

The strategy of the proof of Proposition \ref{super hedging} is to
establish the duality relation \eqref{NewDuality} for $\pi$, resp.\
$(\phi, \psi) $ taken from certain larger candidate sets, in which
case the duality relation follows from Theorem \ref{OTDuality}. Then
we introduce additional constraints via a variational approach to obtain an improved duality through the following min-max theorem.

\begin{theorem}[{see e.g.\ \cite[Thm.\ 45.8]{St85}  or \cite[Thm.\ 2.4.1]{AdHe96}}]\label{minmax} 
 Let $K, L$ be convex subsets of vector spaces $H_1$ resp. $H_2$, where $H_1$ is locally convex and let $F:K\times L\to\R$ be given. If
\begin{enumerate}
 \item K is compact,
\item $F(\cdot, y)$ is continuous and convex on $K$ for every $y\in L$,
\item $F(x,\cdot)$ is concave on $L$ for every $x\in K$
\end{enumerate}
then
$$ \sup_{y\in L}\inf_{x\in K} F(x,y)=\inf_{x\in K}\sup_{y\in L} F(x,y).$$
\end{theorem}
\begin{proof}[Proof of Proposition \ref{super hedging}.]
Fix $t_0> 0$ and consider for a probability $\pi$ on $\CRo\times \R_+\times\R$ and $(\phi, \psi)\in C_b(\CRo)\times C_b(\R)$ the conditions 
\begin{align}\label{p,t_0}\tag{$p[t_0]$}
\supp\pi \subseteq \CRo \times [0, t_0] \times \R, \proj_\CRo(\pi)=\W, \proj_\R(\pi)=\mu
\\
\label{d,t_0}\tag{$d[c,t_0]$}
\phi(\omega)+ \psi (y) \leq c(\omega, t, y), \mbox{ for } \omega \in \CRo, t\leq t_0, y\in \R.
\end{align}
Using compactness of $[0,t_0]$ it is not hard to see that $\tilde c(\omega,y)=\inf_{t\leq t_0} c(\omega, t, y)$ is continuous. We may thus apply the Monge-Kantorovich duality (Theorem \ref{OTDuality})  to the cost  $\tilde c$ and obtain:

{\bf Claim 1.} Taking the $\inf$ over $\pi$ satisfying \eqref{p,t_0} and the $\sup$ over $(\phi, \psi)$ satisfying \eqref{d,t_0}, the duality relation \eqref{NewDuality} holds for continuous bounded $c: \CRo \times \R_+\times \R\to \R$.

 \medskip
 
 Next consider the constraints 
\begin{align}\label{p,t_0, V}\tag{$p[t_0, V]$}
\supp\pi \subseteq \CRo \times [0, t_0] \times \R, \proj_\CRo(\pi)=\W, \proj_\R(\pi)=\mu, \pi(T)\leq V \\
\label{d,t_0, V}\tag{$d[c,t_0, V]$}
\exists{\alpha\geq 0} \mbox{ s.t. } \phi(\omega)+ \psi (y)-\alpha(t-V) \leq c(\omega, t, y) \mbox{ for } \omega \in \CRo, t\leq t_0, y\in \R.
\end{align}
Using the min-max theorem (Theorem~\ref{minmax}) with the function 
$ F(\pi,\alpha)= \textstyle \int  c +\alpha  (T-V) \, d\pi$,
 the set of $\pi$ satisfying \eqref{p,t_0}, and $\alpha\geq 0$  we thus obtain 
\begin{align}
\inf_{\pi \text{ sat. } \eqref{p,t_0, V}}\mbox{$\int$} c\, d\pi & =
 \inf_{\pi \text{ sat. } \eqref{p,t_0}} \Big(\mbox{$\int$} c\, d\pi + \sup_{\alpha \geq 0} \alpha \mbox{$\int$} T-V \, d\pi \Big)   \nonumber \\
\label{withoutQMK} & =\sup_{\alpha \geq 0} \,
 \inf_{\pi\text{ sat. } \eqref{p,t_0}}\mbox{$\int$}  c +\alpha  (T-V) \, d\pi\\
\label{withQMK} & =\sup_{\alpha \geq 0} \,
    \sup_{(\phi,\psi) \text{ sat. } (d[c  +\alpha  (T-V),t_0]) } \W(\phi) + \mu(\psi)\\
&=\sup_{(\phi,\psi)\text{ sat. } (d[c  ,t_0, V])} \W(\phi) + \mu(\psi), \nonumber
\end{align}
where we  applied Claim 1 to the function $\tilde c=c  +\alpha  (T-V)$ to establish the equality between \eqref{withoutQMK} and \eqref{withQMK}. 
Hence we obtain:

{\bf Claim 2.} Taking the $\inf$ over $\pi$ satisfying \eqref{p,t_0, V} and the $\sup$ over $(\phi, \psi)$ satisfying \eqref{d,t_0, V}, the duality relation \eqref{NewDuality} holds for continuous bounded $c: \CRo \times \R_+\times \R\to \R$.

\medskip

In the next step we will drop $t_0$ and consider the constraints 
\begin{align}\label{p, V}\tag{$p[V]$}
 \proj_\CRo(\pi)=\W, \proj_\R(\pi)=\mu, \pi(T)\leq V \\
\label{d, V}\tag{$d[c, V]$}
\exists{\alpha\geq 0} \mbox{ s.t. } \phi(\omega)+ \psi (y)-\alpha(t-V) \leq c(\omega, t, y), \mbox{ for } \omega \in \CRo, t\in \R_+, y\in \R.
\end{align}

{\bf Claim 3.} Taking the $\inf$ over $\pi$ satisfying \eqref{p, V} and the $\sup$ over $(\phi, \psi)$ satisfying \eqref{d, V}, the duality relation \eqref{NewDuality} holds for $c: \CRo \times \R_+\times \R\to \R$ lsc and bounded from below.
\medskip

Given $c\geq 0 $ lsc,  $\supp c \subseteq \CRo\times [0,t_0]\times \R$ for some $t_0$ it is straightforward to verify  \begin{align*}
\inf_{\pi \text{ sat. } \eqref{p,t_0, V}} {\textstyle\int} c\ d\pi&\ =\inf_{\pi \text{ sat. } \eqref{p, V}} {\textstyle\int} c\ d\pi,\\
 \sup_{(\phi,\psi)\text{ sat. } \eqref{d,t_0, V}} \W(\phi) + \mu(\psi)&\ =
 \sup_{(\phi,\psi)\text{ sat. } \eqref{d, V}} \W(\phi) + \mu(\psi).
\end{align*}
Such functions can be used to approximate any   non-negative lsc function on $ \CRo\times \R_+\times \R$ from below. Using that the set of $\pi$ satisfying \eqref{p, V} is compact, a straightforward approximation argument (see e.g.\ \cite[Proof of Theorem 5.10, Step 5]{Vi09} for details) yields Claim 3. 

\medskip

Recalling \eqref{CheckMeasurability}, 
$\pi \in \TRST^{1,V}(\mu)$ if and only if 
\comment{DC (26) The function $g$ is necessary here. We have included some integration variables, hopefully this makes this part more transparent. }
\begin{align}\label{p^M}\tag{$p^M[V]$}\begin{split}  \proj_\CRo(\pi)=\W, \proj_\R(\pi)=\mu, \pi(T)\leq V, \quad \mbox{and } 
\\
\textstyle\int f(s) (g-\E[g|\F_t^0])(\omega)k(y) \, \pi(d\omega, ds, dy)=0\quad \quad\quad  \\
\mbox{for $f\in C_b(\R_+), \supp f\subseteq [0,t]$, $t\geq 0$,
  $g\in C_b(\CRo)$, $k\in C_b(\R)$}; \end{split} \end{align} here, 
$k$ enforces the condition that
$\pr_{ {\CRo}\times \R_+}(\pi_{\llcorner {\CRo}\times \R_+ \times D})
\in \RST$ for all Borel
sets $D$. We will apply the min-max theorem to
$ F(\pi, h )= \textstyle \int \cost + h\, d\pi, $ where $\pi$
satisfies \eqref{p, V} and
\begin{align}\label{TestH}   h(\omega,s,y)= \textstyle \sum_ {i=1}^n f_i(s) (g_i-\E[g_i|\F_{t_i}^0])(\omega)k_i(y),\end{align} 
 $n\in \N$, $f_i\in C_b(\R_+), \supp f_i\subseteq [0,t_i]$, $t_i\geq 0$,  $g_i\in C_b(\CRo)$, $k_i\in C_b(\R)$.

The set of $\pi$ satisfying   \eqref{p, V} is convex and compact by Prokhorov's theorem and the set of all $h$ of the form \eqref{TestH} is a vector space as well. 
Hence we obtain for $c$ continuous and bounded 
\begin{align}\label{QuiteDuality}\begin{split}
 \inf_{\pi\text{ sat. } \eqref{p^M}} \mbox{$\int$} \cost\, d\pi
& = \inf_{\pi\text{ sat. } \eqref{p, V}} \sup_{ h} \mbox{$\int$} \cost  +  h\ d\pi \\
&\hspace{-2.56mm}\stackrel{\text{Thm.~\ref{minmax}}}{=} \sup_{h}\inf_{\pi\text{ sat. } \eqref{p, V}}  \int \cost +  h \ d\pi \\
&= \sup_{ h} \sup_{(\phi,\psi)\text{ sat. } (d[c + h,V])} \W(\phi) + \mu(\psi),\end{split}
\end{align}
where the last equality holds by Claim 3. 
Assume now that $\cost$ is also predictable.
For $(\phi,\psi)$ satisfying $(d[\cost+ h,V])$ there is some $\alpha\geq 0$ such that 
\begin{align}\label{Manu}  \phi(\omega) + \psi(y) - \alpha(t-V)\leq(\cost+ h)(\omega,t,y).\end{align}
Fixing $t$ and $y$, \eqref{Manu} can be read as an inequality between functions in $\omega$. 
Taking conditional expectations wrt $\F^0_t$ in the sense of Definition \ref{EAverage} we  obtain 
$$   \phi^M_t(\omega) + \psi(y) - \alpha(t-V) \leq \cost(\omega,t,y) $$
for all $\omega\in\CRo,t\in\R_+, y\in\R$, where we have used that   $\cost$ is predictable and that $\E [f(t) (g-\E[g|\F_{u}^0])|\F^0_t]=0$ whenever $\supp f \subseteq [0,u]$.

It follows that  $(\phi,\psi)$ satisfy \eqref{d^M}. Thus \eqref{QuiteDuality} yields the non-trivial part of \eqref{NewDuality} for the constraints \eqref{p^M}, \eqref{d^M} in the case of continuous bounded $c$.
As above, the extension to lsc $c$ is straightforward.
\end{proof}

\begin{proof}[Proof of Theorem \ref{HedgingDual}]

Consider the space $\CRo \times \R_+ \times \R$ and the cost function
\begin{align}\label{eq:c}
 c(\omega,t,y):=\begin{cases}
                  \gamma\circ r(\omega,t) & \text{ if } \omega(t)=y\\
		  \infty & \text{ otherwise}
                 \end{cases}. 
\end{align}
It is straightforward to see that $\cost $ is lsc since $\gamma$ was assumed to be lsc. Hence \eqref{NewDuality} holds by Proposition  \ref{super hedging}. 
It remains to show that 
\begin{align}\label{ExplicitIneq}\sup_{(\phi, \psi)\text{ sat. }
  \eqref{d^M}} \W(\phi)+\mu(\psi) \leq D_\gamma \text{ and }  P_\gamma
  \leq \inf_{\pi \in  \TRST^{1,V}( \mu)} \pi(c).\end{align} 
To prove the first inequality, consider a bounded pair $(\phi,\psi)$ satisfying \eqref{d^M}, i.e.\  there is $\alpha\geq 0$ such that $ \phi^M_t(\omega)  +\psi(y) - \alpha(t-V)\leq c(\omega,t,y) $ for all $\omega\in\CRo, y\in \R, t\in \R_+$. But then
$$ \phi^M_t(\omega)  +\psi(\omega(t)) - \alpha(t-V)\leq \gamma\circ r(\omega,t) ,$$
which we rewrite as
\begin{align}\label{ToMart} \big[\underbrace{\phi^M_t(\omega)-\W(\phi)+ \alpha(\omega(t)^2-t)}_{=:\bar \phi_t(\omega)}\big] + \big[\underbrace{\psi(\omega(t))+\W(\phi)-\alpha \omega(t)^2 +\alpha V}_{=:\bar \psi(\omega(t))} \big]\leq \gamma\circ r(\omega,t).\end{align}
Noting that $\alpha(\omega(t)^2-t)$ is an $S$\!-continuous martingale starting in $0$, we find that $(\bar\phi, \bar\psi)$
satisfies the constraint of the dual problem considered in Theorem \ref{HedgingDual}. 
 Since $V=\int y^2\, \mu(dy)$ we have $\int \bar\psi(y)\ \mu(dy)=\int \psi(y)\ \mu(dy)+\W(\phi)$, establishing the first part of \eqref{ExplicitIneq}.

 To prove the latter inequality, note that each $\pi\in \TRST^{1,V}( \mu)$ satisfying $\int c\, d\pi<\infty$ is concentrated on $\{(\omega,t, y): \omega(t)=y\}$ and writing $p(\omega, t, y):=(\omega,t)$ we find $\xi:= p(\pi)\in \RST(\mu)$, $\int c\, d\pi= \int \gamma\, d\xi$.
\end{proof}

\subsection{General starting distribution}\label{S:general start}
In this section we consider $\CRR$, the set of all continuous functions on $\R_+$, and
$$ S_\R =\ \{(f,s): f:[0,s]\to \R \text{ is continuous, } f(0)\in\R\}.$$
Let $\lambda$ be a probability measure on $\R$ prior to $\mu$ in convex order --- i.e., $\int F(x) \ \lambda(dx) \le \int F(x) \ \mu(dx)$ for any convex function $F$. In particular $\lambda$ is centered  and $V_\lambda=\int x^2\ \lambda(dx)\leq V < \infty$. This ensures the existence of solutions to the Skorokhod embedding problem with general starting distribution $\lambda$ with finite first moment. Denote by $\W_x$ the law of Brownian motion starting in $x$ and put $\W_\lambda(d\omega)=\int \W_x(d\omega)\lambda(dx)$ for $\omega\in\CRR$, the law of Brownian motion starting at a random point according to the distribution $\lambda$. Given a function $\gamma: S_\R\to\R$ we are interested in the minimization problem
\begin{equation}\textstyle
 P_{\gamma} =\ \inf\big\{ \int \gamma\circ r(\omega,t)\, \xi(d\omega,dt): \xi\in\RST(\lambda,\mu)\big\},
\end{equation}
where $\RST(\lambda,\mu)$ is the set of all randomized stopping times $\xi$ on $(\CRR,\W_\lambda)$ embedding $\mu$ and satisfying ${\xi}(T)=V-V_\lambda$; in particular $\proj_{\CRR}(\xi)=\W_\lambda$ and $h(\xi)=\mu$ for the map $h:\CRR\times \R_+ \to \R, (\omega,t)\mapsto \omega(t).$ We then have the following result:

\begin{theorem}\label{HedgingDual_general start} Let $\gamma:  S_\R \to \R$ be  lsc and  bounded from below as in \eqref{MartinBound}. Put
$$D_\gamma=\sup\left\{\int \psi(y)\, d\mu(y): \psi \in C(\R), 
\begin{array}{l} \exists \phi, 
\phi\mbox{ is a $S_\R$-continuous martingale},  \\
{\W_\lambda}[\phi_0] = 0,  \phi_t(\omega) + \psi(\omega(t)) \leq \gamma\circ r(\omega,t)
 \end{array}
 \right\}
 $$ where $\phi, \psi$ 
satisfy $|\phi_t| \leq a + bt+c B_t^2 $, $|\psi(y)| \leq a+ b y^2$ for some $a,b,c>0$. 
     Then we have the duality relation 
     $P_\gamma = D_\gamma.$
\end{theorem} 
The proof goes along the same lines as the proof of Theorem \ref{HedgingDual}. The inequality $D_\gamma(\lambda,\mu)\leq P_\gamma(\lambda,\mu)$ is straightforward. For the other direction we can use the same argument as before, replacing $\W$ by $\W_\lambda$ and $V$ by $\tilde V:=V-V_\lambda$. Up to equation \eqref{ToMart} everything can be copied verbatim. Then we rewrite $ \phi^M_t(\omega)+\psi(\omega(t)) - \alpha(t-V+V_\lambda)$ as
\begin{align*}
 [\phi^M_t(\omega) -\W_\lambda(\phi) + \alpha(\omega(t)^2 - t - V_\lambda)] + [\psi(\omega(t))+\W_\lambda(\phi) - \alpha(\omega(t)^2 - V)]
\end{align*}
and note that ${\W_\lambda}(\omega(t)^2)=t+V_\lambda.$ The proof  concludes as before.

\section{The monotonicity principle}\label{s:mp}

In this section we will establish the monotonicity principle: suppose $\xi \in \RST(\mu)$ is an optimal stopping rule for some function $\gamma$, then we will find a set $\Gamma$ supporting $\xi$ such that $\BP\cap(\Gamma^<\times \Gamma)=\emptyset$. The argument can be divided into two major steps:
\begin{enumerate}
\item Consider an optimal stopping rule $\xi$ and a stop-go pair $((f,s),(g,t))\in\SG$ where $(f,s)$ is still going according to the stopping rule $\xi$ while $(g,t)$ is stopped by $\xi$. Intuitively speaking, we can find an (infinitesimal) improvement of $\xi$ by switching the roles of $f$ and $g$.  As $\xi$ is optimal, there should only exist a few such pairs.  We formalize this in Proposition \ref{Invisible2} by showing that if $\pi(\SG)>0$ for some $\pi\in\TRT(r(\xi))$ we can explicitly construct a stopping rule with strictly lower `cost'.  

\item Knowing that $\SG$ is negligible in the sense that it is not seen by the `couplings' $\pi$ just described, it remains to find a support $\Gamma$ of $\xi$ such that $\SG\cap \left(\Gamma^< \times \Gamma\right)=\emptyset.$ The crucial step is the characterization of a set which is null wrt all $\pi\in\TRT(r(\xi))$ which we establish in Proposition~\ref{KeLe} based on Choquet's capacitability theorem and an auxiliary duality result.
\end{enumerate}

Armed with Propositions~\ref{Invisible2} and \ref{KeLe}, we will establish Theorem \ref{GlobalLocal}: If $\xi$ is an optimal stopping time, then Proposition \ref{Invisible2} implies that a certain \emph{set of pairs of paths}, i.e.\ the set of stop-go pairs is negligible in a quasi-sure sense, i.e.\ almost surely null with respect to all $\pi\in\TRT(r(\xi))$. Proposition \ref{KeLe} will then allow us to exclude a \emph{$r(\xi)$-null set of paths} to obtain a support $\Gamma$ of $r(\xi)$ such that $\Gamma^<\times \Gamma$ avoids all stop-go pairs.

\medskip

In the first part of this section we will give a number of definitions and results that are needed to establish Theorem \ref{GlobalLocal} (including the statements of Propositions~\ref{Invisible2} and \ref{KeLe}); the respective proofs will be given subsequently.

The notion of stop-go pairs introduced in Definition \ref{def:SG} requires that \emph{all} possible extensions $\sigma$ are considered. However, to establish the monotonicity principle, it is actually more natural to prove a stronger result that appeals to a relaxed notion of stop-go pairs which are sensitive to the stopping measure $\xi$, or -- more precisely -- to a representation of $\xi$ through a function
$A^\xi$ as in  Theorem \ref{thm:equiv RST} (1). 
\begin{importantconvention}
Throughout this section we will fix $\xi \in \RST(\mu)$, as well as the particular representation $A^\xi$.
\end{importantconvention}

\begin{definition}\label{def:CRST} For $(f,s)\in S$, the conditional randomized stopping time $\xi^{(f,s)}$ is given as \begin{align}\label{CondStop} \xi^{(f,s)}_\omega([0,t]):=
\begin{cases}
\frac{1}{1-A^\xi(f,s)}\left(A^\xi(f\oplus \omega_{\llcorner [0,t]},s+t)-A^\xi(f,s)\right) & \text{if }A^\xi(f,s)<1\\
1&\text{otherwise}
\end{cases}.\end{align}
\end{definition}
The measure $\xi^{(f,s)}$ is the normalized stopping measure given that we followed the path $f$ up to time $s$. In other words this is the normalized stopping measure of the `bush' which follows the `stub' $(f,s)$. We note that $\xi^{(f,s)}$ depends measurably on $(f,s)\in S$.
    
Informally, the following lemma asserts that if $\xi$ is a well-behaved stopping time, then the same holds for $\xi^{(f,s)}$ for typical $(f,s)\in S$.  More precisely, we say that $V\subseteq S$ is \emph{evanescent} if $r^{-1}(V)$ is an evanescent subset of $\CRo\times \R_+$. Equivalently, $V$ is evanescent if there is a Borel set $A\subseteq \CRo, \W(A)=1$ such that $r(A\times \R_+)\cap V=\emptyset$.  Recall that $T$ denotes the projection from ${\CRo}\times \R_+$ onto $ \R_+.$

\begin{lemma}\label{lem:xifs} 
The set $\{(f,s)\in S : \xi^{(f,s)}\notin\RST^1\}$ is evanescent. Moreover, if $F:\CRo\times \R_+\to \R_+$ is predictable and satisfies $\xi(F)<\infty$ then the set $\{(f,s)\in S : \xi^{(f,s)}(F^{(f,s)\oplus}) =\infty\}$ is evanescent.  In particular, $\{(f,s)\in S:\xi^{(f,s)}(T) =\infty\}$ is evanescent, since $\xi\in\RST(\mu)$.  
\end{lemma}

\begin{definition}\label{WeakBP}  
  The set $\BP^\xi$ of \emph{stop-go pairs relative to $\xi$} consists of all
 $\big( (f,s), (g,t)\big)\in S\times S$,   $f(s)= g(t)$ \comment{DC (31) Changed as
   indicated.}  such that 
\begin{align}\label{E:WeakBP}
 \int \gamma^{(f,s)\oplus}( r(\omega,u))\, d\xi^{(f,s)}(\omega,u)+ \gamma(g,t) 
> \gamma(f,s) + \int \gamma^{(g,t)\oplus}( r(\omega,u))\,  d\xi^{(f,s)}(\omega,u) .
\end{align}  
We define stop-go pairs \emph{in the wide sense} by  $\SGW^\xi=\SG^\xi \cup \{(f,s)\in S: A^\xi(f,s)=1\}\times S$.
\end{definition}
In analogy to Definition \ref{def:SG} we agree that \eqref{E:WeakBP} holds in any of the following cases: 
\begin{enumerate} 
\item $\int T\, d\xi^{(f,s)}=\infty$ or $\xi^{(f,s)}(\CRo\times \R_+)<1$; 
\item the integral on the left side equals $\infty$; 
\item either of the integrals is not defined.  
\end{enumerate}

We now discuss the relation between the set $\SG$ given in Definition \ref{def:SG} and the set $\SGW^\xi$. Note that if $A^\xi(f,s)<1$ and $((f,s),(g,t))\in \SG$ for some $(g,t) \in S$ then we shall show below that $((f,s),(g,t))\in \SG^\xi$. In contrast to this, whenever $A^\xi(f,s)=1$, the left and right hand sides of \eqref{E:WeakBP} are identical and $((f,s),(g,t))$ cannot be a stop-go pair relative to $\xi$. However, in general $ \SG \cap \{(f,s)\in S: A^\xi(f,s)=1\}\times S$ may be non-empty. For this reason we are also interested in the set of stop-go pairs in the wide sense which satisfy:
\begin{lemma} \label{lem:SGsubSGW}
   Every stop-go pair is  a stop-go pair in the wide sense, i.e.
   \begin{align}\label{CrucialInclusion}
     \SG\subseteq \SGW^\xi.
   \end{align}   
 \end{lemma}
 
 \begin{remark}\label{rem:SGprop}
 Note that $\BP^\xi$ and $\SGW^\xi$ are Borel subsets of $S\times S$ (corresponding to  predictable subsets of $({\CRo}\times\R_+)\times({\CRo}\times\R_+)=({\CRo}\times\R_+)\times \Ys$ in the sense of Remark \ref{ProdPred}). In contrast, $\BP$ is in general just co-analytic.
 \end{remark} 

\begin{definition}\label{def:gamma monotone}
A set $\Gamma\subset S$ is called  \emph{$(\gamma,\xi)$-monotone} iff 
$$ \SGW^\xi\cap (\Gamma^<\times \Gamma)=\emptyset.$$
\end{definition}

Recall that we say that our optimization problem \eqref{IntPri} is well posed if $\textstyle{\int \gamma ~ d\xi}$ exists with values in $(-\infty,\infty]$ for all $\xi\in\RST(\mu)$ and it is finite for one such $\xi$.
Together with Lemma~\ref{lem:SGsubSGW}, the following result implies Theorem~\ref{GlobalLocal} stated in the introduction, and is itself a slightly stronger result.

\begin{theorem}\label{GlobalLocal2}
Assume that $\gamma:S\to \R$ is Borel measurable, the optimization problem \eqref{RealPri} is well posed and that $\xi\in\RST(\mu)$ is an optimizer. 
Then there exists a $(\gamma,\xi)$-monotone Borel set $\Gamma\subseteq S$ which supports $\xii$ in the sense that $r(\xii)(\Gamma)=1$. 
\end{theorem}

The proof of Theorem \ref{GlobalLocal2} relies on Proposition \ref{Invisible2} and Proposition \ref{KeLe} below.  The first result formalizes the heuristic idea that an optimizer cannot be improved on a large set of paths but at most on a small set of exceptional paths.  The second result allows us to entirely exclude such an exceptional set of paths.

Given functions $F:\Xs\to\Xs', G:\Ys\to\Ys'$ we denote the product map by $F\otimes G:\Xs\times \Ys\to\Xs'\times\Ys'.$ Given a probability $\nu$ on a Polish space $\Ys$, we defined the set $\TRST(\nu)$ in Section \ref{s:joinings}.  An element $\pi \in \TRST(\nu)$ is a measure on $({\CRo}\times \R_+) \times \Ys$, and we will commonly consider the push-forward measure $(F \otimes G)(\pi)$. Typically $F$ will be the map $r: {\CRo} \times \R_+ \to S$, and $G$ will be $r$ or the identity.

\begin{proposition}\label{Invisible2} 
  Assume that $\gamma:S\to \R$ is Borel measurable, the optimization problem \eqref{RealPri} is well posed and that $\xii\in\RST(\mu)$ is an optimizer. Then $(r \otimes \id)(\pi)(\BP^\xii)=0$ for any $\pi\in \TRST^1( r(\xii))$. 
\end{proposition}

Below we  apply Proposition \ref{KeLe} to $(\Ys, \nu)=(S, r(\xii))$, but this choice is not relevant for the proof of Proposition \ref{KeLe} and so we state it for an abstract Polish probability space $(\Ys, \nu)$. 

\begin{proposition}\label{KeLe}
  Let $(\Ys, \nu)$ be a Polish probability space and  $E\subseteq S\times \Ys$ a Borel set. Then the following are equivalent:
  \begin{enumerate}
  \item $(r \otimes \id)(\pi)(E)=0$ for all $\pi\in \TRST^1( \nu )$.
  \item $E \subseteq (F \times \Ys)\ \cup\ (S\times N)$ for some evanescent set $F\subset S$ and a $\nu$-null set $N\subseteq \Ys$.
  \end{enumerate}
\end{proposition}
Intuitively speaking, Proposition \ref{KeLe} characterizes when a predictable set $E\subseteq S\times \Ys$ is `negligible'. In this sense it relates to the classical (cross) section theorem, which implies the following characterization of negligible subsets of $S$.

\begin{proposition}\label{OurSection}
Let $E\subset S$ be Borel. Then the following are equivalent:
\begin{enumerate}
\item$r(\alpha)(E)=0$ for all $\alpha\in \RST$. 
\item $E$ is evanescent.
\item[(1')] $\W(((B_s)_{s\leq \tau},\tau)\in E)=0$ for every $\F^0$-stopping time $\tau $.
\end{enumerate}
\end{proposition}
Note that the equivalence of (1) and (2) in Proposition \ref{OurSection} corresponds precisely to Proposition \ref{KeLe} in the case where $\Ys$ consists of a single element.

\begin{proof}[Proof of Theorem \ref{GlobalLocal2}]
  By Proposition \ref{Invisible2}, $(r \otimes \id)(\pi)(\BP^\xii)=0$ for any $\pi\in \TRST^1(r(\xii))$. Applying Proposition \ref{KeLe} with $(\Ys, \nu)=(S, r(\xii))$ we deduce that there exists an evanescent set $\tilde F\subset S$ and a set $N\subseteq S$ such that $r(\xii)(N)=0$, and $$\BP^\xii \subseteq (\tilde F \times S)\ \cup\ (S\times N).$$ Put $F:=\{(g,t)\in S\!:\!\exists(f,s)\in\tilde F,t\geq s, g\equiv f \text{ on }[0,s]\}.$ Then $F$ is evanescent and satisfies $$\BP^\xii \subseteq (F \times S)\ \cup\ (S\times N).$$ Setting $\Gamma_0= S\setminus (F\cup N) $ we have $r(\xii)(\Gamma_0)=1$ as well as $\BP^\xii\cap ( \Gamma_0^<\times \Gamma_0)=\emptyset.$ Next we define $$\Gamma_1:=\Gamma_0 \cap \{(g,t)\in S: A^\xi(g_{\llcorner [0,s]},s) <1 \mbox{ for all } s<t\}.$$ Then $r(\xi)( \Gamma_1)=1$ and $ \Gamma_1^< \cap \{(f,s):A^\xi(f,s)=1\}=\emptyset$ so that $ \SGW^\xi \cap (\Gamma_1^<\times \Gamma_1)=\emptyset$. Finally we take $\Gamma$ to be a Borel subset of $\Gamma_1$ which has full measure.
\end{proof}

It remains to establish the auxiliary results stated above.
\begin{proof}[Proof of Lemma \ref{lem:xifs}]
  Consider
  \begin{align*}U_1&\ =\textstyle\left\{(f,s)\in S: A^\xi(f,s)<1, \int d\xi^{(f,s)}(\omega,t)<1\right\},\\
    U_2&\ =\textstyle \left\{(f,s)\in S: A^\xi(f,s)<1, \xi^{(f,s)}(F^{(f,s)\oplus})=\infty\right\}.
  \end{align*}
  Set $A^\xi(\omega):= \lim_{t\to \infty} A^\xi \circ r (\omega,t)$.
  Then 
  $(f,s)\in U_1 $ is equivalent to $\int A^\xi(f\oplus \omega)\, d\W(\omega)<1$. Given an $\F^0$-stopping time $\tau$, the strong Markov property implies 
  \begin{align*}
    1&\textstyle =\int d\W(\omega)\, A^\xi(\omega)
       \textstyle = \int d\W(\omega)\, \left[\1_{\tau(\omega)=\infty} A^\xi(\omega) + \1_{\tau(\omega)<\infty} \int d\W(\omega')\, A^\xi(\omega_{\llcorner [0, \tau(\omega)]}\oplus \omega')
       \right], 
  \end{align*}
  hence $\W(((B_s)_{s\leq \tau}, \tau)\in U_1)=0$.

  Additionally, setting $\alpha(d\omega,dt)=\delta_{\tau(\omega)}(dt)\W(d\omega)$ we have 
  \begin{align*}
    \textstyle
    \infty> \xi(F) \geq \int_{U_2} dr(\alpha)(f,s)\, (1-A^\xi(f,s)) \int F^{(f,s)\oplus}(\omega,t)\ d\xi^{(f,s)}(\omega,t),
  \end{align*}
  which implies $r(\alpha)(U_2)=0$. Summing up, we get $\W(((B_s)_{s\leq \tau}, \tau)\in U_1\cup U_2)=0$ proving the claim in view of Proposition \ref{OurSection}.
\end{proof}
\begin{proof}[Proof of Lemma \ref{lem:SGsubSGW}]
  Suppose that $A^\xi(f,s)<1$ and $((f,s),(g,t))\not\in \SGW^{\xi}$ for some $(g,t) \in S$ with $g(t) = f(s)$. In particular, \eqref{E:WeakBP} fails for $\xi^{(f,s)}$, and conditions (1)--(3) above all fail. By Theorem~\ref{thm:equiv RST} (4), and using the same argument as seen in the proof of Lemma~\ref{lem:equivOpt}, we can find a $(\F^0_t \otimes \mathcal{B}([0,1]))_{t\ge 0}$-stopping time $\rho$ such that $\oW(\rho>0) > 0$ and for any measurable and bounded or non-negative $X: \CRo \times \R_+ \to \R$, we have $\int X_t(\omega) \, d \xi^{(f,s)}(\omega,t) = \int \Leb(du) \int \W(d\omega) X_{\rho(\omega,u)}(\omega)$. By the conditions below Definition~\ref{WeakBP}, it follows that there exists $u_0 \in [0,1]$ such that 
  \begin{align}\label{E:WeakBP2a}\textstyle
    \int \gamma^{(f,s)\oplus}(B_{\rho(\omega,u_0)})\, d\W(\omega)+ \gamma(g,t) 
    \le \gamma(f,s) + \int \gamma^{(g,t)\oplus}(B_{\rho(\omega,u_0)})\,  d\W(\omega),
  \end{align}  
  and such that $\rho_0: \omega \mapsto \rho(\omega,u_0)$ is an $(\F^0_t)_{t \ge 0}$-stopping time with $0 < \W(\rho_0) < \infty$, both sides of \eqref{E:WeakBP2a} are well defined, and the left hand side is finite. 
  In particular, writing $B^\Omega$ for Brownian motion on the abstract probability space $\Omega$, $\sigma = \rho_0\circ B^\Omega$ defines an $\F^B$-stopping time, and \eqref{BPIneqProb-intro} fails for this stopping time. Hence $((f,s),(g,t))\not\in \SG$. 
\end{proof}

\subsection{Proof of Proposition \ref{Invisible2}} 
Working towards a contradiction we assume that there is $\pi\in \TRST( r(\xii))$ such that $(r \otimes \id)(\pi)(\BP^\xii)>0$. Observe that $\pi\in\TRST(r(\xi))$ implies that $\pi_{\llcorner (r\otimes \id)^{-1}(E)}\in\TRST(r(\xi))$ for any $E\subset S\times S$. Hence, considering $(r \otimes \id)(\pi)_{\llcorner  \BP^\xii}$, we can also assume that $(r \otimes \id)(\pi)$ is concentrated on $\BP^\xii$ and then  $r(\proj_\Xs(\pi))(\{(f,s):A^\xii(f,s)=1\})=0$, where $\Xs:= \CRo\times \R_+$. Finally we also consider the representation of $\pi$ on $(\CRo\times \R_+)\times (\CRo\times \R_+)$  defined through
\begin{align}
  \label{eq:barpi}
  \textstyle \bar\pi (C\times D):= \int d\pi((\omega,s),(g,t)) \int d\W(\eta)\,  \1_{C\times D } ((\omega,s), ((g,t)\oplus \eta))
\end{align}
and note that $\pi=(\id\otimes r)(\bar \pi). $

We will use $\pi$ and $\bar\pi$ to define modifications $\xii_0^\pi\in\RST$ and $\xii_1^\pi\in\RST$ of $\xi$ such that the following hold true: 
\begin{enumerate}
\item The terminal distributions $\mu_0, \mu_1$ corresponding to $\xii_0^\pi$ and $\xii_1^\pi$ satisfy $(\mu_0+\mu_1)/2= \mu.$ \item $\xii_0^\pi$ stops paths earlier than $\xii$ while $\xii_1^\pi$ stops later than $\xii$. \item The cost of $\xii_0^\pi$ plus the cost of $\xii_1^\pi$ is less than twice the cost of $\xii$, i.e.\
$$ \textstyle \int \gamma\circ r(\omega,t) \, d\xii^\pi_0(\omega, t)+ \int \gamma\circ r(\omega,t) \, d\xii^\pi_1(\omega, t) < 2 \int \gamma\circ r(\omega,t) \, d\xii(\omega, t).$$
\end{enumerate}
More formally, (2) asserts that for almost all $\omega$, and every $s\geq0$
\begin{align*}
(\xii_0^\pi)_\omega([0,s])\geq \xii_\omega([0,s]) \quad
\mbox{ and } \quad  (\xii_1^\pi)_\omega([0,s])\leq \xii_\omega([0,s]), 
\end{align*}
where $(\xii_\omega)_{\omega\in{\CRo}}$ is the disintegration of $\xii$ wrt $\W$ induced by $A^{\xi}$ and
$((\xii_0^\pi)_\omega)_{\omega\in{\CRo}},$ $ ((\xii_1^\pi)_\omega)_{\omega\in{\CRo}}$ are disintegrations of $\xii_0^\pi, \xii_1^\pi$  wrt $\W$. 

\medskip

If we are able to construct such a pair $\xii_0^\pi, \xii_1^\pi$, then $\xii^\pi:=(\xii_0^\pi+\xii_1^\pi)/2 \in\RST(\mu)$  is strictly better than $\xii$ and therefore yields the desired contradiction.

In the proof we will often use the following `strong Markov property' of randomized stopping times: for  $\alpha\in\RST$ and  bounded measurable $F:\CRo\times\R_+\to\R$ we have
\begin{align}\label{eq:strong Markov}
 \textstyle \iint F^{(g,t)\oplus}( \omega)~ d\W(\omega) ~dr(\alpha)(g,t) = \int F(\omega,t)~d\alpha(\omega,t)~.
\end{align}
To define $\xii_0^\pi$, let $\alpha_0= \pr_\Xs(\pi)\in\RST$ and consider $A^{\alpha_0}:S\to [0,1]$ as in Theorem \ref{thm:equiv RST} (1).
We define the randomized stopping time $\xii_0^\pi$ via the product
$$ (1-A^{\xii_0^\pi}) (f,s):= (1- A^{\alpha_0})(f,s) \cdot (1-A^{\xii})(f,s).$$
The probabilistic interpretation of this definition is that a particle is stopped by $\xii_0^\pi$ if it is stopped by $\alpha_0$ or stopped by $\xii$, where these events are taken to be conditionally independent given the particle followed the path $f$ until time $s$. Comparing $\xii$ and $\xii_0^\pi$ the latter will stop some particles earlier than the first one. Also, $\xii^\pi_0\in\RST$ by Theorem \ref{thm:equiv RST} (1). By partial integration, if $D\subset {\CRo}\times \R_+$ then $\xi^\pi_0$ satisfies 
\begin{align*}\textstyle
\xii^\pi_0(D)\ =\int_D (1- A^{\xii}\circ r(\omega,t))\ d\alpha_0(\omega,t) + \int_D (1- (A^{\alpha_0}\circ r)_-(\omega,t))\ d\xii(\omega,t),
\end{align*} 
where $(A^{\alpha_0}\circ r)_-$ denotes the left continuous version of $A^{\alpha_0}\circ r$.

Our next goal is to derive (in \eqref{eq:nu 0} below) a representation for the difference between $\xii_0^\pi$ and $\xii$. For Borel $D\subset {\CRo}\times \R_+$ we have 
\begin{align}\label{Diff01}\textstyle
 \xii^\pi_0(D)-\xii(D) = \int_D (1- A^\xii\circ r(\omega,t))\ d\alpha_0(\omega,t) - \int_D (\alpha_0)_\omega([0,t))\ d\xii(\omega,t).
\end{align}
Furthermore, writing $D_\omega= \{ t\in \R_+: (\omega,t)\in D\}$ and $\theta_s(\omega)= (\omega_{t+s}-\omega_s)_{t\geq 0}$, we have
 \begin{align}\nonumber
\textstyle\int_D (\alpha_0)_\omega([0,t))\, d\xii(\omega,t) 
=&\textstyle \int_{\CRo} d\W(\omega)\int_{\R_+} d\xii_\omega(t) \int_{\R_+}\, d(\alpha_0)_{\omega}(s)\, \1_D(\omega,t) \1_{[0,t)}(s)
 \\
 =&\textstyle \int_{\CRo} d\W(\omega) \int_{\R_+}\, d(\alpha_0)_\omega(s) \int_{\R_+} d\xii_\omega(t) \, \1_{D_\omega}(t) \1_{(s,\infty)}(t) \nonumber
 \\
 =&\textstyle \int_{\CRo\times \R_+}  d\alpha_0(\omega,s) \,   (1\!-\!A^\xii\!\circ\! r(\omega,s)) \xii^{r(\omega,s)}_{\theta_s(\omega)}(D_\omega-s). \label{Diff02}
 \end{align}
Combining \eqref{Diff01} and \eqref{Diff02} we obtain for bounded  measurable $F:S\to \R$ using \eqref{eq:strong Markov} 
\begin{align}\label{eq:nu 0}  
\begin{split}
&\textstyle \int F \circ r\, d(\xii^\pi_0 -\xii) =\\
& \textstyle\int d\alpha_0(\omega,s)\,   (1\!-\!A^\xii\!\circ\! r(\omega,s)) \left[F\circ r(\omega,s) - \int F^{r(\omega,s)\oplus}(r(\tilde\omega,u))\, d\xii^{r(\omega,s)}(\tilde\omega,u)\right]=\\
& \textstyle\int d\bar\pi((\omega,s),(\eta,t))\,   (1\!-\!A^\xii\!\circ\! r(\omega,s)) \left[F\circ r(\omega,s) - \int F^{r(\omega,s)\oplus}(r(\tilde\omega,u))\, d\xii^{r(\omega,s)}(\tilde\omega,u)\right].
\end{split}
\end{align}
 Let us now turn to the definition of $\xii_1^\pi$. 
For $D\subset {\CRo}\times \R_+$ we define 
\begin{align}\label{Diff10}
\textstyle\alpha_1(D)\ =\ \int_{\Xs\times D} (1-A^\xii \circ r(\omega,s)) \ d\bar\pi((\omega,s),(\eta,t))
\end{align}
and observe that $\alpha_1\in\RST$ by Theorem \ref{thm:equiv RST} (2) since $\eta\mapsto (\alpha_1)_\eta([0,t])$ is $\F_t^a$-measurable by \eqref{eq:barpi}.
Then we define the probability measure $\xii_1^\pi$ on ${\CRo}\times \R_+$ by
\begin{align}\label{Diff11}\textstyle
\xii_1^\pi(D):=\xii(D)\!-\!\alpha_1(D)+
\int  (1\!-\!A^\xii\!\circ\! r(\omega,s))\ \xii^{r(\omega,s)}_{\theta_t(\eta)} (D_{\eta}-t) \ d\bar\pi\big((\omega,s), (\eta,t)\big).
\end{align}
 To motivate this definition, we note that  the support of the randomized stopping time $\xii$ can be viewed (informally) as a sub-tree of $S$. The joining $\pi$ defines a plan how to trim this tree, i.e.\ to cut a bush at position $r(\omega,s)$ and to plant it on top of $r(\eta,t)$. Hence, we take the tree, $\xii$, prepare the position where something will be newly planted, subtract $\alpha_1$ which takes away some mass, and plant as much as possible (accounting for the factor $(1-A^\xii\circ r(\omega,s))$ in \eqref{Diff10} and \eqref{Diff11}) on these stubs to end up with a tree of mass one again. 

 As a consequence of Definition \ref{def:CRST} for each $u$ the map $\eta \mapsto \xi^{(f,s)}_{\theta_t(\eta)}([0,(u-t)\vee0])$ is $\F_u^0$-measurable. Moreover, $(\xi^\pi_1)_\eta\in\mathcal P^{\leq 1}(\R_+)$ and it follows that $\xii^\pi_1\in\RST.$ From \eqref{Diff10} and \eqref{Diff11} it follows that for bounded measurable $F:S\to \R$ using \eqref{eq:barpi} 
 \begin{align}\label{eq:nu 1} 
   \begin{split}
     & \textstyle\int F \circ r\, d(\xii^\pi_1-\xii) =\\
     &\textstyle \int d\bar\pi((\omega,s),(\eta,t))\, (1-A^\xii\circ r(\omega,s))\left[\int F^{r(\eta,t)\oplus} ( r(\tilde\omega,u))\, d\xii^{r(\omega,s)}(\tilde\omega,u) - F\circ r(\eta,t)\right].
   \end{split} 
 \end{align}
 Adding \eqref{eq:nu 0} and \eqref{eq:nu 1} and recalling $2\xi^\pi=\xi^\pi_0+\xi^\pi_1$, we obtain for bounded measurable $F:S\to \R$ 
 \begin{align}\label{eq:nu 2}
   & \textstyle 2  \int F\circ r\, d(\xii^\pi-\xii) = \int d\bar\pi((\omega,s),(\eta,t))\,   (1-A^\xii\!\circ\! r(\omega,s)) \\ 
   \textstyle \Big[ & \textstyle F\circ r(\omega,s)
                      + \int F^{r(\eta,t)\oplus} (r(\tilde\omega,u))\, d\xii^{r(\omega,s)}(\tilde\omega,u)   - \int F^{r(\omega,s)\oplus}(r(\tilde\omega,u))\, d\xii^{r(\omega,s)}(\tilde\omega,u) - F\!\circ\!  r(\eta,t)
                      \Big].
                      \nonumber
 \end{align}
 Next we show that \eqref{eq:nu 2} extends to non-negative functions $F:S\to\R_+$ satisfying $\xi(F)<\infty$. Put $X(\omega):=\int F(r(\omega,t))~\xi_\omega(dt)$. Then $\E[X]=\xi(F)<\infty.$ Moreover, recalling Definition \ref{EAverage} we have
 $$\textstyle X^M_s(\omega)=\int_0^s F(r(\omega,t))~d\xi_\omega(t) +  \int F^{r(\omega,s)\oplus}(r(\tilde\omega,u))\cdot(1-A^\xi(r(\omega,s))~d\xi^{r(\omega,s)}(\tilde\omega,u)~.$$
 It then follows that
 \begin{align*}
   & \textstyle \iint F^{r(\omega,s)\oplus}(r(\tilde\omega,u))\cdot(1-A^\xi(r(\omega,s))~d\xi^{r(\omega,s)}(\tilde\omega,u)~d\bar\pi((\omega,s),(\eta,t))\\
   \leq & \textstyle \int X^M_s(\omega)~d\bar\pi((\omega,s),(\eta,t))  = \bar \E[\bar X^M_{\rho_0}\1_{\rho_0<\infty}]\leq\E[X^M_\infty]=\E[X]<\infty~, 
 \end{align*}
 where $\rho_0$ denotes the representation of $\alpha_0$ as in \eqref{eq:rhodefn} and $\bar X^M_t(\omega,u)=X^M_t(\omega)$. This implies
 $$ \textstyle \int d\bar\pi((\omega,s),(\eta,t))\, (1-A^\xii\!\circ\! r(\omega,s)) \Big[ \int F^{r(\omega,s)\oplus}(r(\tilde\omega,u))\, d\xii^{r(\omega,s)}(\tilde\omega,u) + F\circ r(\eta,t) \Big] < \infty,$$ hence \eqref{eq:nu 2} holds also for such $F$.  Applying this to $F(f,s)=s$ we find $\xii^\pi(T)=\xii(T)<\infty$. 
Taking $F(f,s)=G(f(s))$ for bounded measurable $G:\R\to \R$, the right hand side of \eqref{eq:nu 2} vanishes since $\bar\pi$ is concentrated on pairs $((\omega,s),(\eta,t))$ satisfying $\omega(s)=\eta(t)$. This implies that $\xii$ and $\xii^\pi$ embed the same distribution, i.e.\ $\xii^\pi\in \RST(\mu)$. 

Arguing on the negative and positive part of $\gamma$ and using that $\xi^\pi(\gamma^-),\xi(\gamma^-)<\infty$ we see that \eqref{eq:nu 2} applies to $F=\gamma$. By definition of $\SG^\xii$,
$$\textstyle (1-A^\xii\!\circ\! r(\omega,s)) 
\textstyle \Big[  \textstyle \gamma\circ r(\omega,s)
+ \int \gamma^{r(\eta,t)\oplus} \circ  r\, d\xi^{r(\omega,s)}   - \int \gamma^{r(\omega,s)\oplus}\circ r\, d\xii^{r(\omega,s)} - \gamma\circ r(\eta,t)
\Big]$$ is $\bar\pi$-a.s.\ strictly negative since $r(\proj_\Xs(\pi))(\{(f,s):\xi^{(f,s)}(T)=\infty \text{ or } \xi^{(f,s)}\notin\RST^1\})=0$ by Lemma \ref{lem:xifs}. Hence $\xii^\pi(\gamma) <\xii(\gamma)$, contradicting optimality of $\xii$.\qed

\subsection{Proof of Proposition \ref{KeLe}} 
Only the implication (1) $\Rightarrow$ (2) of Proposition \ref{KeLe} is non-trivial. The proof is based on Choquet's capacitability theorem and the following auxiliary duality result which is closely related to Proposition \ref{super hedging}. We fix $t_0\in \R_+$ and set $S\!_{t_0}:=\{(f,s)\in S: s\leq t_0\}$. 
\begin{proposition}\label{duality} 
  Consider a Polish probability space $(\Ys,\nu)$ and let $c:\CRo\times \R_+\times \Ys\to \R \cup \{\infty\}$ be lsc, predictable (cf.\ Remark \ref{ProdPred}) and bounded from below. Then 
  \begin{align}\label{New Duality2}\tag{$\star\star$}
    \inf_{\pi  } {\textstyle \int c(\omega,t,y)} \, \pi(d\omega,
    dt,dy)=\sup_{(\phi,\psi)} \left({\textstyle \W(\phi) +\nu(\psi)}\right)
  \end{align}
  where the infimum is taken over the set 
  \begin{align*}
    \TRST_{t_0}^{1} (\nu)=\{\pi\in \TRST^1(\nu): \supp\pi \subseteq \CRo \times [0, t_0] \times \Ys\}
  \end{align*}
  and the supremum is taken over $\phi\in C_b(\CRo)$, $\psi\in C_b(\Ys)$ such that 
  \begin{align*}
    \phi^M_t(\omega)+ \psi (y) \leq c(\omega, t, y), \mbox{ for } \omega \in \CRo, t\leq t_0, y\in \Ys.
  \end{align*}
\end{proposition}
\begin{proof}
As the arguments are almost identical to the ones from Proposition \ref{super hedging}
 we will only sketch the proof. By approximation it is sufficient to establish the result for continuous bounded $c$. 
As before the Monge-Kantorovich duality yields that \eqref{New Duality2} holds provided that $\pi$ and $(\phi, \psi)$, resp., satisfy
\begin{align*}
  \supp\pi \subseteq \CRo \times [0, t_0] \times \Ys, \proj_\CRo(\pi)=\W, \proj_\Ys(\pi)=\nu
  \\
\phi(\omega)+ \psi (y) \leq c(\omega, t, y), \mbox{ for } \omega \in \CRo, t\leq t_0, y\in \Ys.
\end{align*}
If $c$ is predictable, we can then argue as in the last step of Proposition \ref{super hedging} to obtain the assertion of Proposition \ref{duality} 
\end{proof}

We now state several consequences of Proposition \ref{duality} in which we switch the roles of $\inf$ and $\sup$ to provide a more natural formulation.  

Denote  for Borel $K\subseteq S\!_{t_0}\times \Ys $
\begin{align}\label{eq:preddualS}
D_{t_0}(K):= \inf_{(\phi,\psi)\in \DC_{t_0}(K)}\left( \W(\phi) + \nu(\psi) \right)
\end{align}
where $\DC_{t_0}(K)$ consists of all pairs of lsc $\phi,\psi$ on 
$\CRo$ resp.\
$\Ys$ satisfying 
\begin{align}\label{KCover}  0\leq \phi, \psi\leq 1, \1_K((f,s),y) \leq
  \phi^{M,S}(f,s)+\psi(y), (f,s)\in S\!_{t_0} , y\in \Ys,
\end{align}
where we recall the notation $\phi^{M,S}$ from Definition \ref{EAverage}.

\begin{corollary}\label{cor:bounded duality on S}
 Let $K\subset S\!_{t_0}\times \Ys$ be closed. Then 
 $$\sup_{\pi \in\TRT^1_{t_0}( \nu)} (r \otimes \id)(\pi)(K) = D_{t_0}(K).$$
\end{corollary}
\begin{proof}
  Fix $\eps>0$. Applying Proposition \ref{duality} to  $c=-\1_{(r\otimes \id)^{-1}(K)}$, which is lsc due to the continuity of $r$, we obtain that there exist functions $\phi\in C_b(\CRo)$, $\psi\in C_b(\Ys)$ such that
\begin{align}\label{RevCover}
 \1_K((f,s), y)\leq \phi^{M,S}(f,s)+ \psi (y) , \mbox{ for } (f,s)\in S\!_{t_0}, y\in \Ys, \quad \mbox{and}
\end{align}
 $$\textstyle
 \sup_{\pi \in\TRT^1_{t_0}( \nu)} \int \1_{(r\otimes \id)^{-1}(K)}\, d\pi = \sup_{\pi \in\TRT^1_{t_0}( \nu)} (r \otimes \id)(\pi)(K) >  {\textstyle \W(\phi)+ \nu(\psi)} -\eps.
$$ 
It follows from \eqref{RevCover} that $\phi^{M,S}$ is bounded from below on $S\!_{t_0}$ and wlog we may assume that $\phi(\omega)=\phi_{t_0}^M(\omega)$. Subtracting a constant from $\phi$ and adding it to $\psi$, we may assume that $\inf\phi= 0$ (which implies $\psi\geq0$). It follows that we can replace $\psi$ with $\bar\psi=\psi\wedge 1$. 

It suffices to consider the case $\W(\phi)\leq 1$.
Put $\rho=\inf\{t\geq0:\phi^M_t > 1\}.$ Due to $S$\!-continuity of $\phi^M$ (by Proposition~\ref{prop:very cont}) the set $O:=\{(\omega,t) :\phi^{M,S}\circ r(\omega,t)>1\}$ is open. Hence also $\{\rho<\infty\}=\proj_{\CRo} O$ is open as projections are open mappings and the map $\omega\mapsto \phi^M_{\rho(\omega)}(\omega)=:\bar \phi(\omega)\leq 1$ is lsc.  Clearly, $(\bar\phi,\bar\psi)$ satisfies \eqref{KCover} and $\W(\bar\phi)+\nu(\bar\psi)\leq \W(\phi)+\nu(\psi).$
\end{proof}

\begin{lemma}\label{lem:choquet}
$D_{t_0}$
is a Choquet capacity on $S\times \Ys.$
\end{lemma}
\begin{proof}
  We need to verify the defining properties of a capacity $\Psi$ (cf.\ \cite[Definition 30.1]{Ke95}):
\begin{enumerate}
 \item monotonicity: $A\subset B \Rightarrow \Psi(A)\leq \Psi(B)$
\item continuity from below: $A_1\subset A_2\subset... \Rightarrow \Psi(A_n) \to \Psi(\bigcup_j A_j)$
\item boundedness: $\Psi(K)<\infty$ for all compact $K$; if $\Psi(K)<u$ there exists open $U\supset K$ with $\Psi(U)<u.$
\end{enumerate}
Moreover, it is sufficient to test these properties for Borel sets (see \cite[Section 30B]{Ke95}).  The monotonicity is immediate. Let us turn to the continuity from below.

Take an increasing sequence $A_1 \subseteq A_2\subseteq \ldots \subseteq S \times \Ys$ of Borel sets and put $A=\bigcup_n A_n.$ For all $n$ there are lsc functions $\phi_n:{\CRo}\to [0,1]$ (which give rise to  $S$\!-lsc martingales) and  $\psi_n: \Ys \to [0,1]$ such that $\1_{A_n}((f,s),y)\leq \phi_n^{M,S}(f,s) + \psi_n(y)$ for all $(f,s)\in S\!_{t_0},y\in \Ys$ and
$$ \nu(\psi_n) + \W(\phi_n) \leq D_{t_0}(A_n) + 1/n .$$
Using a Mazur/Komlos-type lemma (e.g. Lemma A1.1\ in \cite{DeSc94}) we can assume that some appropriate convex combinations of $\psi_n$ and $\phi_n$ converge a.s.\ to functions $\psi$ and $\phi$. More precisely: there exist convex coefficients $\alpha_n^{n}, \ldots,\alpha_{k_n}^{n}, n\geq 1, k_n<\infty,$ and full measure subsets $\Omega_1\subset \CRo$, $\Ys_1\subseteq \Ys$ such that with $\tilde \phi_n:= \sum_{i=n}^{k_n} \alpha_i^{n} \phi_i$, $\tilde \psi_n:= \sum_{i=n}^{k_n} \alpha_i^{n} \psi_i$ we have that for all  $\omega\in \Omega_1$ and all $y\in \Ys_1$
\begin{align}\lim_{n\to \infty} \tilde \phi_n(\omega)=: \phi(\omega) \mbox{ and }
  \lim_{n\to \infty} \tilde\psi_n(y)=:\psi(y)
\end{align}
exist.  Extend these functions to ${\CRo}$ and $\Ys$, resp., through 
\begin{align}\limsup_{n\to \infty} \tilde \phi_n(\omega)=:\phi(\omega) \mbox{ and }
\limsup_{n\to \infty} \tilde\psi_n(y)=:\psi(y).
\end{align}
This implies for $(f,s)\in S$
\begin{align*}\textstyle
 \limsup_{n\to\infty} \tilde\phi_n^{M,S}(f,s)\leq \int \limsup_{n\to\infty} \tilde\phi_n(f\oplus \omega)~\W(d\omega) = \int \phi(f\oplus\omega)~ \W(d\omega)=\phi^{M,S}(f,s).
\end{align*}
Given $m\leq n$ we have for $(f,s)\in S\!_{t_0},y\in \Ys$
$$\I_{A_m}((f,s), y) \leq \tilde \phi^{M,S}_n(f,s)+\tilde\psi_n(y),$$
hence $\I_{A_m}((f,s), y) \leq \phi^{M,S}(f,s)+\psi(y)$ and thus also 
$$\I_{A}((f,s), y) \leq  \phi^{M,S}(f,s)+\psi(y).$$
Given $\eps>0$, we can find lsc functions $\phi^\eps \geq \phi$ and $\psi^\eps\geq \psi$ such that $\W(\phi^\eps)-\eps/2< \W(\phi)=\lim \W({\tilde \phi_n})$ and $\nu(\psi^\eps)-\eps/2 < \nu(\psi)=\lim \nu(\tilde \psi_n).$
It follows that 
\begin{align*}\textstyle
 D_{t_0}(A) \leq  \limsup_n  D_{t_0}(A_n) +1/n +\eps .
\end{align*}

Let us turn to the third property. Trivially, $D_{t_0}(K) \le 1$, so take a compact set $K\subset S\times \Ys$ and fix $\eps >0.$ By Corollary \ref{cor:bounded duality on S} 
there is $(\phi,\psi)\in \DC_{t_0}(K)$ such that
$$ \nu(\psi) + \W(\phi)   \leq D_{t_0}(K) + \epsilon.$$
As $(\phi,\psi)\in \DC_{t_0}(K)$ we have $K\subset \{((f,s),y):\phi^{M,S}(f,s)+\psi(y) \geq 1\}.$ At the additional cost of  2 $\eps$ we can find two lsc functions $\phi^\eps :=(\phi+\eps) \wedge 1 \geq \phi$ and $\psi^\eps:= (\psi+\eps) \wedge 1 \geq \psi$ such that $\W(\phi^\eps) + \nu(\psi^\eps)\leq \W(\phi) + \nu(\psi)+2\eps $ and $K\subset\{((f,s),y):(\phi^\eps)^{M,S}(f,s)+\psi^\eps(y)>1\}$.  By lower semi-continuity, $U:=\{((f,s),y):(\phi^\eps)^{M,S}(f,s)+\psi^\eps(y)>1\}$ is open. Hence, for every $\eps>0$ we have found an open $U\supset K$ such that $D_{t_0}(U)\leq D_{t_0}(K) + 3\eps$, proving the last claim.
\end{proof}

The next step is to show that up to a factor of $2$ we can restrict ourselves to dual functions $\phi$ and $\psi$ which are indicator functions. The simple reason is that if $1\leq a + b$  then $a> 1/2$ or $b\geq 1/2$.
 In the formulation of the next lemma and subsequently we use the notation
\begin{align*}
\deb_{t_0}(F) \,& :=\{ \omega: \exists t  < t_0, r(\omega, t)\in F\}.
\end{align*}
\begin{lemma}\label{lem:cover}
  Let $K\subseteq S\!_{t_0} \times \Ys$ be Borel. Then
\begin{align}\label{SetCover}
 \inf_{(F,A)\in \mathsf{Cov}(K)} \Big(\W(\deb_{t_0}(F))) + \nu(A)\Big) \leq 2 D_{t_0}(K),
\end{align}
$\text{where }\mathsf{Cov}(K)=\{F \subseteq S \mbox{ open }, A\subseteq \Ys: K\subseteq (F \times \Ys) \cup (S\times A)\}.$
\end{lemma}
\begin{proof}
We may assume $D_{t_0}(K)<1/2$, otherwise simply take $A=\Ys, F=\emptyset$.

 Take $(\phi,\psi)\in\DC_{t_0}(K)$. As the cost function is $\{0,1\}$-valued, the dual constraint
$$ \I_K((f,s),y)\leq \phi^{M,S}(f,s) + \psi(y)$$
implies that 
$$K\subseteq (\{(f,s): \phi^{M,S}(f,s)> 1/2\}\times \Ys) \cup (S \times \left\{y: \psi(y)\geq 1/2\right\}).$$
Recalling  that $0\leq\psi\leq 1$ we set $A=\{\psi\geq1/2\}$ and note that $\nu(A)/2\leq\nu(\psi) .$

Let us turn our attention to the set $F=\{(f,s):\phi^{M,S}(f,s)> 1/2 \}$. As $D_{t_0}(K)<1/2$, we may assume that $\phi^{M,S}(0,0)< {1}/{2}$. 
Given $\eps>0$ we apply the optional section theorem to $r^{-1}(F)\cap (\CRo\times [0,t_0))$ to obtain a stopping time $\tau$ such that $\W(\tau<t_0)>\W(\deb_{t_0}(F))-\eps$ and $\phi^M_\tau>1/2$ on $\{\tau< t_0\}$.  By optional stopping
$$  \E[\phi_0^M]= \E [\phi^M_\tau] \geq \W(\tau<t_0)/2. $$ 
As $\eps>0$ was arbitrary,   $ \W(\deb_{t_0}(F))+ \nu(A) \leq 2 (\E[\phi_0^M]+ \nu(\psi))$, establishing \eqref{SetCover}.
\end{proof}
\begin{proof}[Proof of Proposition \ref{KeLe}.]
Assume first that $E\subseteq S\!_{t_0}\times \Ys$. 
We have $\sup_{\pi\in \TRST^1( \nu )} \pi (K)=0$ for all compact $K\subset E$. By Corollary \ref{cor:bounded duality on S}, this implies that $D_{t_0}(K)=0$ for all compact $K\subset E$. By Choquet's capacitability theorem \cite[Theorem 30.13]{Ke95} and Lemma \ref{lem:choquet} this in turn implies $D_{t_0}(E)=0$.

Hence, by Lemma~\ref{lem:cover}, for each $\eps>0$ there exist  $F \subset S$ and a set $N\subset \Ys$ such that $E \subseteq (F \times \Ys)\cup (S \times N)$ and $\W(\deb_{t_0}(F))+\nu(N)\leq 2\eps.$

For each $k$, pick some  set $F_k\subset S$ and a set $N_k\subset \Ys$ such that $E \subseteq \left(F_k\times \Ys\right)\cup \left(S \times N_k\right)$ and $\W(\deb_{t_0}(F_k)) + \nu(N_k) \leq 2^{-k}$. 
 Setting $ F=\limsup_k F_k$ and $ N= \limsup_k N_k$ we get $\W(\deb_{t_0}( F))=0$, $\nu(N)=0$ and 
\begin{eqnarray*}\textstyle 
  E  \subseteq  \left( F\times \Ys\right)\cup \left(S \times  N\right).
\end{eqnarray*}
To establish the result in the case of general $E\subseteq S\times \Ys$, for each $n\in \N$ pick sets $N_n\subseteq \Ys, \nu (N_n)=0,$ $F_n\subseteq S$, $\W(\deb_{n}(F_n))=0$ such that $E\cap (S_n\times \Ys)\subseteq (F_n\times \Ys) \cup (S\times N_n).$ Then $N:= \bigcup_{n\geq 1} N_n$ and $F:=\bigcup_{n\geq 1} F_n$ are as required.
\end{proof}

\subsection{A secondary minimization result}\label{sec:scd max}

In certain cases, in order to resolve possible non-uniqueness of a minimizer, it will be useful to identify particular solutions as the solution not only to a primary optimization result, but also as the unique optimizer within this class of a second minimization problem. To this end, we begin by making the following definition: Supposing that $\gamma, \tilde \gamma:S \to \R$ are Borel measurable, we write $\Opt_\gamma$ for the set of optimizers of \eqref{RealPri}.  If $\Opt_\gamma\neq\emptyset$, we consider the secondary optimization problem
  \begin{equation}\label{eq:secmax}\textstyle
    P_{\tilde\gamma|\gamma} = \inf_{\xii \in \Opt_\gamma} \int \tilde\gamma\, d\xii.
  \end{equation}
We will say that \eqref{eq:secmax} is \emph{well posed} if the primary optimization problem \eqref{RealPri} is well posed and 
$\int \tilde\gamma\, d\xii$ exists with values in $(-\infty,\infty]$ for all $\xii\in  \Opt_\gamma$ and is finite for one such $\xii$.
Observe that, when $P_\gamma$ is finite and the map $\pi\mapsto \int \gamma\ d\pi$ is lsc the set $\Opt_\gamma$ is a closed subset of $\RST(\mu)$, and hence also compact.

We need an extended version of the stop-go pairs introduced in Definition~\ref{WeakBP}.
\begin{definition} \label{def:SG2} 
    Let $\gamma,\tilde\gamma:S\to\R$ be Borel measurable. The set of secondary stop-go pairs $\SG_2^\xi$ (relative to $\xi$) consists of all $\big( (f,s), (g,t)\big)\in S\times S$, $f(s)= g(t)$ \comment{DC (31) Changed as indicated.}  such that either $((f,s), (g,t)) \in \SG^\xi$, or 
\begin{equation}\label{E:WeakBP2}
\begin{split}
& \textstyle \int \gamma^{(f,s)\oplus}\circ r\, d\xi^{(f,s)}+ \gamma(g,t) = \gamma(f,s) + \int \gamma^{(g,t)\oplus}\circ r \, d\xi^{(f,s)}\,\\
\text{and } & \textstyle \int \tilde{\gamma}^{(f,s)\oplus}\circ r\, d\xi^{(f,s)}+ \tilde{\gamma}(g,t) > \tilde{\gamma}(f,s) + \int \tilde{\gamma}^{(g,t)\oplus}\circ r \, d\xi^{(f,s)}.
\end{split}
\end{equation}
As before, we also say that \eqref{E:WeakBP2} holds if any of the integrals in the second equation are not defined, or the integral on the left-hand side equals $\infty$.

We also define secondary stop-go pairs \emph{in the wide sense} by  ${\SGW_2}^\xi=\SG_2^\xi \cup \{(f,s)\in S: A^\xi(f,s)=1\}\times S$.
\end{definition}

Then we have the following generalization of Theorem~\ref{GlobalLocal2}.
\begin{theorem}
  \label{thm:second-maxim-result}
  Let $\gamma, \tilde\gamma$ be Borel measurable functions on $S$. Suppose that $\Opt_\gamma\neq\emptyset$, and that the optimization problem \eqref{eq:secmax}
  is well posed with optimizer $\xii\in\Opt_\gamma$.
  Then there exists a Borel set $\Gamma\subseteq S$ such that $r(\xii)(\Gamma) = 1$ and
  \begin{equation}\label{eq:SBP}
    {\SGW_2}^{\xii} \cap \big(\Gamma^< \times \Gamma\big) = \emptyset.
  \end{equation}
\end{theorem}

The proof given for Theorem \ref{GlobalLocal2} also applies in the present situation. Hence, the result follows immediately from the following straightforward variant of Proposition \ref{Invisible2}.

\begin{proposition}
  Assume that $\gamma, \tilde{\gamma}: S \to \R$ are measurable, the optimization problem \eqref{eq:secmax} is well posed, and that $\xii \in \RST(\mu)$ is an optimizer. Then $(r \otimes \id)(\pi)(\SG_2^\xii)=0$ for any $\pi\in \TRST^1(r(\xii))$.
\end{proposition}
\begin{proof} 
 As $\xii\in\Opt_\gamma$ we have to show that $(r \otimes \id)(\pi)(\SG_2^\xii\setminus\BP^\xii)=0$, however this follows by considering the same construction as in the proof of Proposition \ref{Invisible2}. 
\end{proof}

\section{Embeddings in abundance}
\label{sec:embeddings-abundance}

In the following we suppose that $(\Omega,\G,(\G_t)_{t\geq0},\P)$ is a stochastic basis which is sufficiently rich to support  a Brownian motion $B$ and a uniformly distributed $\G_0$-random variable. We suppose that $\gamma: S\to \R$ is a Borel measurable function. In a slight abuse of notation we will also write $(\gamma_t)_{t\in \R_+}$ for the process given by 
$$t\mapsto \gamma((B_s)_{s\leq t}, t).$$
In the previous section we have considered a secondary optimization problem and a version of the monotonicity principle (Theorem \ref{thm:second-maxim-result}) accounting for this extension. We now give a brief summary in probabilistic terms. 

Write $\Opt_\gamma$ for the set of $\G$-stopping times on $\Omega$ which are  optimizers of  \eqref{IntPri}
and consider  another Borel function $\tilde \gamma:S\to\R$. We call ${\hat \tau}\in \Opt_\gamma$  a secondary minimizer if it solves
\begin{align}\label{OptSep2}\tag{OptSEP$_2$}
 P_{\tilde\gamma|\gamma}=\inf\{\E\left[ \tilde\gamma_\tau\right]  : \tau\in\Opt_\gamma\}.
\end{align}
As in \eqref{eq:secmax}  we say that \eqref{OptSep2} is well posed if the primary optimization problem \eqref{IntPri} is well posed and 
$\E\left[ \tilde\gamma_\tau\right]$ exists with values in $(-\infty,\infty]$ for all $\tau\in  \Opt_\gamma$ and is finite for one such $\tau$.
Then we have the following version of Theorems~\ref{MinimizerExists} and \ref{MinimizerExistsP}:
\begin{theorem}\label{MinimizerExists2}
  Let $\gamma, \tilde{\gamma}:S\to \R$ be lsc and bounded from below in the sense of \eqref{MartinBound}.  Then \eqref{OptSep2} admits a minimizer $\hat \tau$.
\end{theorem}

 We now provide the appropriate generalizations of Definitions~\ref{def:SG} and \ref{def:Gamma1} and Theorem~\ref{GlobalLocal} for this case.

\begin{definition} \label{def:SG2B} The pair $\big((f,s), (g,t)\big)\in S\times S$ constitutes a \emph{secondary stop-go pair}, written $\big((f,s), (g,t)\big)\in\SG_2$, iff $f(s)=g(t)$, and 
 for \emph{every} $(\F^B_t)_{t \ge 0}$-stopping time $\sigma$ which satisfies $0 < \E[\sigma] < \infty$, 
  \begin{align} \label{BPIneq1}
    \E\big[\big(\gamma^{(f,s)\oplus }\big)_{\sigma}\big]\ +\ \gamma(g,t)\quad \geq \quad \gamma(f,s) \ +\ \E\big[\big(\gamma^{(g,t)\oplus }\big)_{\sigma}\big],
  \end{align}
  whenever both sides are well defined, and the left-hand side  is finite; and if
  \begin{align}\label{BPIneq2}
    \E\big[\big(\gamma^{(f,s)\oplus }\big)_{\sigma}\big]\ +\ \gamma(g,t)\quad = \quad \gamma(f,s) \ +\ \E\big[\big(\gamma^{(g,t)\oplus }\big)_{\sigma}\big]
  \end{align}
  then
  \begin{align}
    \E\big[\big(\tilde\gamma^{(f,s)\oplus }\big)_{\sigma}\big]\ +\ \tilde\gamma(g,t)\quad&>\quad \tilde\gamma(f,s) \ +\ \E\big[\big(\tilde\gamma^{(g,t)\oplus }\big)_{\sigma}\big], \label{BPIneq3}
  \end{align}
  whenever both sides are well-defined and the left-hand side (of \eqref{BPIneq3}) is finite.
\end{definition}

\begin{definition} \label{def:ggtildeMon}
We say that $\Gamma$ is $ \tilde\gamma| \gamma$-monotone if \begin{align}\label{GM3}\SSG \cap \big( \Gamma^< \times \Gamma\big)=\emptyset.\end{align} 
\end{definition}
From Theorem~\ref{thm:second-maxim-result} together with a trivial modification of Lemma~\ref{lem:SGsubSGW} we then obtain:
\begin{theorem}[Monotonicity Principle II]\label{GlobalLocal3}
  Let $\gamma, \tilde \gamma:S\to\R$ be Borel measurable, suppose that \eqref{OptSep2} is well posed and that $\hat \tau$ is an optimizer.
  Then there exists a \emph{$\tilde\gamma| \gamma$-monotone} Borel set $\Gamma\subset S$ such that $\P$-a.s.
\begin{align}\label{GammaSupport2} ((B_t)_{t\leq\hat\tau},\hat\tau)\in\Gamma\;.\end{align}
\end{theorem}

\subsection{Recovering classical embeddings}
\label{sec:recov-class-embedd}

In this section we derive a number of classical embeddings as well as establish new embeddings. Figure~\ref{fig:AssortedConstructions} shows graphical representations of some of these constructions. We highlight the common feature of all these pictures: when plotted in an appropriate phase space, the stopping time is the hitting time of a barrier type set. Identifying the appropriate phase space, and determining the exact structure of the barrier will be the key step in deriving the solutions to \eqref{SkoSol} in this section.

\begin{figure}[th]
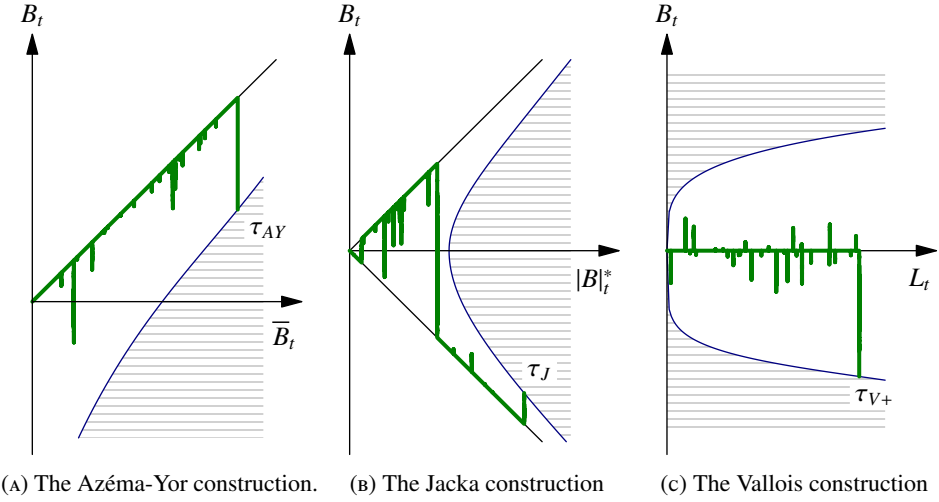

\begin{subfigure}{.33\textwidth}
\centering
\begin{asy}[width=0.9\textwidth]
  import graph;
  import stats;
  import patterns;

  // Construct a Brownian motion of time length T, with N time-steps
  int N = 3000;
  real T = 1.25;
  real dt = T/N;
  real B0 = 0;

  real[] B; // Brownian motion
  real[] t; // Time
  real[] M; // Maximum

  path BM;
  path BMM;

  real xmax = 1;
  real xmin = -0.6;
  real ymax = (xmax-xmin)/1.6;

  // Seed the random number generator. Delete for a "random" path:
  srand(98);

  B[0] = B0;
  t[0] = 0;
  M[0] = B[0];
  
  BM = (t[0],B[0]);
  BMM = (M[0],B[0]);

  //real psiinv(real y) {return y-((y-2)**2)*exp(-y/4)/4-1/2;}
  real psiinv(real y) {return max((-exp(-(y-0.7)**3)+y*1.2+0.4),xmin);}

  real x0 = 0.18;
  real x1 = 0.9;

  real eps3 = 0.275;
  
  int H = N+1;
  int BMMstop;
  
  for (int i=1; i<N+1; ++i)
  {
    B[i] = B[i-1] + Gaussrand()*sqrt(dt);
    t[i] = i*dt;
    M[i] = (B[i] > M[i-1])? B[i] : M[i-1];
    BM = BM--(t[i],B[i]);
    if (M[i-1] < M[i])
    BMM = BMM--(M[i-1],M[i-1])--(M[i],B[i]);
    else
    {
      if ((H==N+1)&&(B[i]<psiinv(M[i])||(M[i] >= ymax-eps3)))
      {
        H = i;
        BMMstop = length(BMM)+1;
        BMM = BMM--(M[i],psiinv(M[i]));
      }
      else
      {
        BMM = BMM--(M[i],B[i]);
      }
    }
  }
  
  if (H==N+1)
  BMMstop = length(BMM);

  pen p = deepgreen + 1.5;
  pen p2 = lightgray + 0.25;
  
  //if (H<N+1)
  //draw(subpath(BMM,BMMstop,length(BMM)),p2);

  pair tau = point(BMM,BMMstop);

  //draw((B0,psiinv(tau.y))--(tau.y,psiinv(tau.y)),p2+dashed);
  
  pen q = black + 0.5;

  real eps1 = 0.05;

  real eps2 = 0.15;

  path barrier = (graph(identity,psiinv,x0,x1)--(max(ymax-eps2,x1),psiinv(x1))--(max(ymax-eps2,x1),psiinv(x0))--cycle);

  add("hatch",hatch(1mm,W,mediumgray));
  
  fill(barrier,pattern("hatch"));

  //Label L = Label("$\Psi_{\mu}(B_t^0)$",UnFill);

  // draw(shift(-1mm,0)*Label("$\Psi_{\mu}(B_t^0)$",UnFill(0.5mm)),graph(psiinv,identity,B0,xmax),NW,deepblue+0.5);
  draw(graph(identity,psiinv,x0,x1),NW,deepblue+0.5);

  draw((B0,B0)--(xmax-eps1,xmax-eps1));

  draw(subpath(BMM,1,BMMstop),p);
  draw((0,xmin)--(0,xmax+eps1),q,Arrow);
  draw((0,0)--((ymax+eps1),0),q,Arrow);
  label("$B_t$",(0,xmax+eps1),(0,1));
  //label("$\sup_{s \le t} B_t$",(0,(xmax+eps1)),(0,1));
  label("$\ol{B}_t$",((ymax+eps1),-eps1),SW);
  
  draw(Label("$\tau_{AY}$",UnFill(0.25mm)),(tau.x+0.12,tau.y-0.08));   
\end{asy}
\caption{The Az\'ema-Yor construction.}
\end{subfigure}%
\begin{subfigure}{.33\textwidth}
\centering
\begin{asy}[width=0.9\textwidth]
  import graph;
  import stats;
  import patterns;

  // import settings;

  // gsOptions="-P"; 

  // Construct a Browniam motion of time length T, with N time-steps
  int N = 3000;
  //int N = 300;
  real T = 1.6*1.6;
  real dt = T/N;
  real B0 = 0;

  real sig = 0.7;

  real xmax = 0.8;
  real xmin = -0.8;

  real tmax = (xmax-xmin)/1.6;

  real[] B; // Brownian motion
  real[] t; // Time
  real[] M; // Abs max

  path BM;

  // Seed the random number generator. Delete for a "random" path:
  srand(108);

  B[0] = B0;
  t[0] = 0;
  M[0] = 0;

  BM = (M[0],B[0]);

  // Define a barrier

  // Assumes hits below zero.
  real R(real y) {return sqrt(y**2+0.15)+0.5*(y-1)**2*max(y,0)**2;}
  real Rinv(real y) {return sqrt(y**2-0.15);}
  
  int H = N+1;
  int H2 = N+1;
  int BMstop;
  int BMstop2 = N+2;

  for (int i=1; i<N+1; ++i)
  {
    B[i] = B[i-1] + sig*Gaussrand()*sqrt(dt);
    t[i] = i*dt;
    M[i] = max(M[i-1],abs(B[i]));

    if ((H==N+1)&&(M[i]>=R(B[i])))
    {
      H = i;
      BMstop = length(BM);
      B[i] = Rinv(M[i])*sgn(B[i]);
    }
    if (M[i-1]==M[i])
    {
      BM = BM--(M[i],B[i]);
    }
    else
    {
      BM = BM--(M[i-1],sgn(B[i-1])*M[i-1])--(M[i],sgn(B[i])*M[i])--(M[i],B[i]);
    }

  }

  if (H==N+1)
  BMstop = length(BM);

  pen p = deepgreen + 1.5;
  pen p2 = lightgray + 0.25;
  //pen p2 = mediumgray + 1;

  //if (H<N+1)
  //draw(subpath(BM,BMstop,BMstop2),p2);

  pair tau = point(BM,BMstop+1);

  pen q = black + 0.5;

  real eps1 = 0.05;
  real eps2 = 0.15;

  path barrier = (graph(R,identity,xmin+eps1,xmax-eps1)--(max(tmax-eps2,R(xmax-eps1),R(xmin+eps1)),xmax-eps1)--(max(tmax-eps2,R(xmax-eps1),R(xmin+eps1)),xmin+eps1)--cycle);

  add("hatch",hatch(1mm,W,mediumgray));

  fill(barrier,pattern("hatch"));

  draw((0,0)--(xmax-eps1,xmax-eps1),q);
  draw((0,0)--(xmax-eps1,xmin+eps1),q);

  draw(graph(R,identity,xmin+eps1,xmax-eps1),NW,deepblue+0.5);

  draw((0,xmin)--(0,xmax+eps1),q,Arrow);
  draw((0,0)--((tmax+eps1),0),q,Arrow);

  //draw((tau.x,0)--tau,mediumgray+dashed);
  //label("$\hat\tau$",(tau.x,tau.y),S);
  draw(Label("$\tau_J$",UnFill(0.25mm)),(tau.x+0.06,tau.y+0.08));   

  draw(subpath(BM,0,BMstop+1),p);

  label("$|B|^*_t$",(tmax+eps1,-eps1),SW);
  label("$B_t$",(0,(xmax+eps1)),(0,1));

  //label("$B_{\Rt}$",(3.5,0.5),UnFill(0.5mm));

  //label("$D_{\Rt}$",(2.2,1.35),UnFill(0.5mm));
\end{asy}
\caption{The Jacka construction}
\end{subfigure}%
\begin{subfigure}{.33\textwidth}
\centering
\begin{asy}[width=0.9\textwidth]
  import graph;
  import stats;
  import patterns;

  // import settings;

  // gsOptions="-P"; 

  // Construct a Browniam motion of time length T, with N time-steps
  int N = 3000;
  //int N = 300;
  real T = 1.6*1.6;
  real dt = T/N;
  real B0 = 0;

  real sig = 0.7;

  real xmax = 0.8;
  real xmin = -0.8;

  real tmax = (xmax-xmin)/1.6;

  real[] B; // Brownian motion
  real[] M; // Max of Brownian motion
  real[] t; // Time
  real[] sign; // Sign of excursion

  real[] B2; // Brownian motion

  path BM;

  // Seed the random number generator. Delete for a "random" path:
  srand(94);

  B[0] = B0;
  t[0] = 0;
  B2[0] = B0;
  M[0] = B[0];
  sign[0] = sgn(Gaussrand());

  BM = (M[0],B2[0]);

  // Define a barrier

  real R1(real y) {return 0.8*xmax*exp(2.5*(y-4))+y**(1/4)/2;}
  real R2(real y) {return 0.55*xmin*exp(3.3*(y-2))+0.3*xmin*exp(1.4*(y-3))-y**(1/6)/2;}

  int H = N+1;
  int H2 = N+1;
  int BMstop;
  int BMstop2 = N+2;

  for (int i=1; i<N+1; ++i)
  {
    B[i] = B[i-1] + sig*Gaussrand()*sqrt(dt);
    t[i] = i*dt;
    M[i] = max(M[i-1],B[i]);
    
    B2[i] = sign[i-1]*(M[i]-B[i]);

    if (M[i]==B[i])
    {
      sign[i] = sgn(Gaussrand());
    }
    else
    {
      sign[i] = sign[i-1];
    }

    if ((H==N+1)&&(B2[i]>=R1(M[i])))
    {
      H = i;
      BMstop = length(BM);
      B2[i] = R1(M[i]);
    }

    if ((H==N+1)&&(B2[i]<=R2(M[i])))
    {
      H = i;
      BMstop = length(BM);
      B2[i] = R2(M[i]);
    }
    
    if (M[i] > M[i-1])
    {
      BM = BM--(M[i-1],0)--(M[i],B2[i]);   
    }
    else
    {      
      BM = BM--(M[i],B2[i]);   
    }
  }

  if (H==N+1)
  BMstop = length(BM);

  pen p = deepgreen + 1.5;
  pen p2 = lightgray + 0.25;
  //pen p2 = mediumgray + 1;

  //if (H<N+1)
  //draw(subpath(BM,BMstop,BMstop2),p2);

  pair tau = point(BM,BMstop+1);

  pen q = black + 0.5;

  real eps1 = 0.05;
  real eps2 = 0.15;
  real eps3 = 0.1;

  path barrier1 = (graph(R1,0,tmax-eps2)--(tmax-eps2,xmax-eps3)--(0,xmax-eps3)--cycle);
  path barrier2 = (graph(R2,0,tmax-eps2)--(tmax-eps2,xmin+eps3)--(0,xmin+eps3)--cycle);

  add("hatch",hatch(1mm,W,mediumgray));

  fill(barrier1,pattern("hatch"));
  fill(barrier2,pattern("hatch"));

  draw(graph(R1,0,tmax-eps2),NW,deepblue+0.5);
  draw(graph(R2,0,tmax-eps2),NW,deepblue+0.5);

  draw((0,xmin)--(0,xmax+eps1),q,Arrow);
  draw((0,0)--((tmax+eps1),0),q,Arrow);

  //draw((tau.x,0)--tau,mediumgray+dashed);
  //label("$\hat\tau$",(tau.x,0),S);
  draw(Label("$\tau_{V+}$",UnFill(0.25mm)),(tau.x+0.06,tau.y-0.08));   

  draw(subpath(BM,0,BMstop+1),p);

  label("$L_t$",(tmax+eps1,-eps1),SW);
  label("$B_t$",(0,(xmax+eps1)),(0,1));

  //label("$B_{\Rt}$",(3.5,0.5),UnFill(0.5mm));

  //label("$D_{\Rt}$",(2.2,1.35),UnFill(0.5mm));
\end{asy}
\caption{The Vallois construction}
\end{subfigure}%
\caption{Representations of the Az\'ema-Yor, Vallois and Jacka constructions.}\label{fig:AssortedConstructions}
\end{figure}

 For subsequent use, it will be helpful to write, for $(f,s) \in S$, $\ol{f} = \sup_{r \le s} f(r)$, $\ul{f} = \inf_{r \le s} f(r)$ and $|f|^* = \sup_{r \le s} |f(r)|$. 
\begin{theorem}[The Az\'ema-Yor embedding, cf.\ \cite{AzYo79}]
  \label{thm:AY}
  There exists a stopping time $\tau_{AY}$ which \emph{maximizes}
  \begin{equation*}
    \E\Big[\sup_{t \le \tau} B_t\Big]
  \end{equation*}
  over all solutions to \eqref{SkoSol} and which is of the form $\tau_{AY} = \inf \big\{ t > 0: B_t \le \psi\big(\sup_{s \le t} B_s\big)\big\}$ a.s., for some increasing function $\psi$.
\end{theorem}

\begin{proof}
  Fix a bounded and strictly increasing continuous function $\phi:\R_+\to\R$ and consider the  continuous functions $\gamma((f,s)) = -\ol{f}$ and $\tilde\gamma((f,s)) = \phi(\ol{f})(f(s))^2$. Then
 \eqref{OptSep2} is well posed and by Theorem~\ref{MinimizerExists2} there exists a minimizer $\tau_{AY}$.  By Theorem \ref{GlobalLocal3}, pick a $\tilde\gamma|\gamma$-monotone set $\Gamma \subseteq S$ supporting $\tau_{AY}$. We claim that 
\begin{align}\label{AYBadPairs}
\SSG \supseteq\{((f,s),(g,t))\in S\times S: g(t)=f(s), \ol g <\ol f \}.
\end{align} 
This is represented graphically in Figure~\ref{fig:AYSG}.

\begin{figure}[t]
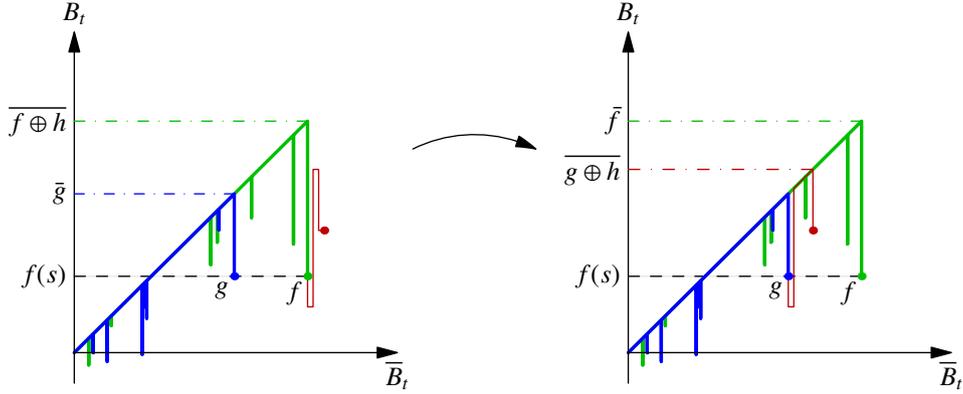

\centering
\begin{asy}[width=\textwidth]
  import graph;
  import stats;
  import patterns;

  // Construct a Brownian motion of time length T, with N time-steps
  //int N = 3000;
  //real T = 5;
  //real dt = T/N;
  real B0 = 0;

  real[] B; // Brownian motion
  //real[] t; // Time
  real[] M; // Maximum

  real[] B2; // Brownian motion
  //real[] t2; // Time
  real[] M2; // Maximum

  //path BM;
  path BMM;
  //path BM2;
  path BMM2;

  real xmax = 1;
  real xmin = -0.1;
  real ymax = xmax;

  // Seed the random number generator. Delete for a "random" path:
  srand(98);

  //B[0] = B0;
  //t[0] = 0;
  //M[0] = B[0];
  //B2[0] = B0;
  //t2[0] = 0;
  //M2[0] = B[0];
  
  //BM = (t[0],B[0]);
  BMM = (B0,B0);
  //BM2 = (t2[0],B2[0]);
  BMM2 = (B0,B0);

  int N1 = 6;
  real lambda = 8;
  real sig = 0.2;
  real temp=0;
  real temp2;

  real xfinal = 0.25;

  for (int i = 1; i<N1+1;++i)
  {
    temp = temp-log(unitrand())/lambda;
    temp2 = max(temp-abs(Gaussrand()*sig),(xmin*0.5+temp*0.2));
    BMM = BMM--(temp,temp)--(temp,temp2)--(temp,temp);
  }

  temp = temp-log(unitrand())/lambda;
  BMM = BMM--(temp,temp)--(temp,xfinal);

  real M1 = temp;
  
  temp = 0;
  for (int i = 1; i<N1+1;++i)
  {
    temp = temp-log(unitrand())/lambda;
    temp2 = max(temp-abs(Gaussrand()*sig),(xmin*0.5+temp*0.2));
    BMM2 = BMM2--(temp,temp)--(temp,temp2)--(temp,temp);
  }

  temp = temp-log(unitrand())/lambda;
  BMM2 = BMM2--(temp,temp)--(temp,xfinal);

  real M2 = temp;

  path BM1Ext, BM2Ext;
  
  pair final1 = point(BMM,length(BMM));
  pair final2 = point(BMM2,length(BMM2));
  
  real eps= 0.018;

  BM1Ext =
  final1--(final1.x,0.15)--(final1.x+eps,0.15)--(final1.x+eps,0.6)--(final1.x+2*eps,0.6)--(final1.x+2*eps,0.4)--(final1.x+3*eps,0.4);
  
  BM2Ext = final2--(final2.x,0.15)--(final2.x+eps,0.15)--(final2.x+eps,final2.x+eps)--(0.6,0.6)--(0.6,0.4);
   
  int BMMstop = length(BMM);

  pen col1 = heavyred;
  pen col2 = heavygreen;
  pen col3 = blue;

  pen p = col2 + 1.25;
  pen p2 = col3 + 1.25;
  pen p3 = col1 + 0.5;
  
  //if (H<N+1)
  //draw(subpath(BMM,BMMstop,length(BMM)),p2);

  pair tau = point(BMM,BMMstop);
  pair tau2 = point(BMM2,length(BMM2));

  //draw((B0,psiinv(tau.y))--(tau.y,psiinv(tau.y)),p2+dashed);
  
  pen q = black + 0.5;

  //Label L = Label("$\Psi_{\mu}(B_t^0)$",UnFill);

  real eps1 = 0.05;

  real eps2 = 0.15;

  //draw((B0,B0)--(xmax-eps1,xmax-eps1));

  draw((0,tau.y)--tau,lightgray);
  draw((0,tau.y)--tau,dashed);

  draw(BMM,p);
  draw(BMM2,p2);
  draw(BM1Ext,p3);
  dot((final1.x+3*eps,0.4),col1);
  dot(tau,col2);
  dot(tau2,col3);
  draw((0,xmin)--(0,xmax+eps1),q,Arrow);
  draw((0,0)--((ymax+eps1),0),q,Arrow);
  label("$B_t$",(0,xmax+eps1),(0,1));
  //label("$\sup_{s \le t} B_t$",(0,(xmax+eps1)),(0,1));
  label("$\ol{B}_t$",((ymax+eps1),0),S);

  draw((0,M1)--(M1,M1),dashdotted+col2+0.5);
  draw((0,M2)--(M2,M2),dashdotted+col3+0.5);

  label("$f(s)$",(0,xfinal),W);
  label("$\bar{g}$",(0,M2),W);
  label("$\overline{f\oplus h}$",(0,M1),W);
  label("$g$",final2,SW);
  label("$f$",final1,SW);

  transform xshift = shift(xmax*1.8,0);

  //draw(xshift*((B0,B0)--(xmax-eps1,xmax-eps1)));

  draw(xshift*( (0,tau.y)--tau),lightgray);
  draw(xshift*( (0,tau.y)--tau),dashed);

  draw(xshift*BMM,p);
  draw(xshift*BMM2,p2);
  draw(xshift*BM2Ext,p3);
  dot(xshift*(0.6,0.4),col1);
  dot(xshift*tau,col2);
  dot(xshift*tau2,col3);
  draw(xshift*( (0,xmin)--(0,xmax+eps1)),q,Arrow);
  draw(xshift*( (0,0)--((ymax+eps1),0)),q,Arrow);
  label("$B_t$",xshift*( (0,xmax+eps1)),(0,1));
  //label("$\sup_{s \le t} B_t$",(0,(xmax+eps1)),(0,1));
  label("$\ol{B}_t$",xshift*( ((ymax+eps1),0)),S);

  draw(xshift*((0,M1)--(M1,M1)),dashdotted+col2+0.5);
  draw(xshift*((0,0.6)--(0.6,0.6)),dashdotted+col1+0.5);

  label("$f(s)$",xshift*(0,xfinal),W);
  label("$\bar{f}$",xshift*(0,M1),W);
  label("$\overline{g\oplus h}$",xshift*(0,0.6),W);
  label("$g$",xshift*final2,SW);
  label("$f$",xshift*final1,SW);

  draw((xmax+2*eps1,2*xmax/3){(1,0.5)}..shift(-2*eps1,0)*shift(xmax*1.6)*(0,2*xmax/3),q,Arrow);
  
\end{asy}
\caption{The stop-go pairs for the Az\'ema-Yor embedding. On the left,
the blue path $(g,t)$ is stopped, and the green path, $(f,s)$, is allowed to
continue; a possible continuation, $h$, being shown in red. On the right
hand side we see the effect of allowing $g$ to go and stopping $f$: the maximum of $g$ is increased, but the maximum of $f$
stays the same.\label{fig:AYSG}}
\end{figure}

Indeed, pick $((f,s),(g,t))\in S\times S$ with $f(s)=g(t)$ and $ \ol g< \ol f$ and a stopping time $\sigma$ with positive and finite expectation. Then  \eqref{BPIneq1} amounts to
$$\textstyle\E\big[\bar f \vee (f(s)+\bar B_\sigma)\big] + \ol g \leq \ol f + \E\big[\bar g \vee (g(t) + \bar B_\sigma)\big]$$
with a strict inequality unless $\bar g  \geq g(t) + \bar B_\sigma$ a.s. However in that case \eqref{BPIneq2} is trivially satisfied and \eqref{BPIneq3} amounts to 
\begin{align*}\textstyle
\E\Big[\phi(\bar f) (f(s)+B_\sigma)^2 \Big] + \phi(\bar g ) g(t)^2  >\phi(\bar f )f(s)^2 +  \E\Big[\phi(\bar g) (g(t)+B_\sigma)^2 \Big]   
\end{align*}
which holds since $g(t)=f(s)$.
Summing up, $((f,s),(g,t))\in \BP \subseteq \SSG$ in the former case and $((f,s),(g,t))\in \SSG$ in the latter case, proving \eqref{AYBadPairs}.

In complete analogy with the derivation of the Root embedding (Theorem~\ref{RootThm}) we define
\begin{align*}
  \mathcal{R}_\cl  := \left\{(m,x): \exists (g,t) \in \Gamma, \ol{g} \le m, g(t) = x\right\},\ 
  \mathcal{R}_\op  := \left\{(m,x): \exists (g,t) \in \Gamma, \ol{g} < m, g(t) = x \right\},
\end{align*}
and write $\tau_\cl, \tau_\op$ for the first times the process $(\ol
B_t(\omega),{B}_t(\omega))$ hits the sets $\mathcal{R}_\cl$ and
$\mathcal{R}_\op$ respectively.  Then we claim $\tau_\cl \le \tau_{AY}
\le \tau_\op$ a.s.  Note that $\tau_\cl \le \tau_{AY}$ holds by
definition of $\tau_\cl.$ To show $\tau_{AY} \le \tau_\op$, consider  $\omega$ satisfying $((B_s(\omega))_{s\leq\tau_{AY}(\omega)},\tau_{AY}(\omega))\in \Gamma$ and assume for contradiction that $\tau_\op(\omega)<\tau_{AY}(\omega).$ Then there exists $s\in \big[\tau_\op(\omega),\tau_{AY}(\omega)\big)$ such that $f:=(B_r(\omega))_{r\leq s}$ satisfies $(\bar f, f(s))\in \mathcal{R}_\op$. Since $s< \tau_{AY}(\omega)$ we have $(f,s)\in \Gamma^<$. By definition of $\mathcal{R}_\op$, there exists $(g,t)\in \Gamma$ such that $f(s)= g(t)$ and $\bar g < \bar f$, yielding a contradiction.

Finally, we define
\begin{equation*}
  \psi_0(m) = \sup\{x: (m,x) \in \mathcal{R}_\cl\}.
\end{equation*}
It follows from the definition of $\mathcal{R}_\cl$ that $\psi_0(m)$ is increasing, and we define the right-continuous function $\psi_+(m) = \psi_0(m+)$, and the left-continuous function $\psi_-(m) = \psi_0(m-)$. It follows from the definitions of $\tau_{\op}$ and $\tau_{\cl}$ that:
\begin{equation*}
  \tau_+ := \inf\{t \ge 0: B_t \le \psi_+(\ol{B}_t)\} \le \tau_{\cl} \le \tau_{\op} \le  \inf\{t \ge 0: B_t < \psi_-(\ol{B}_t)\} =: \tau_-.
\end{equation*}
It is then easily checked that $\tau_- = \tau_+$ a.s., and the result follows on taking $\psi = \psi_+$.
\end{proof}

\begin{theorem}[The Jacka Embedding, cf.\ \cite{Ja88}]
  \label{thm:Jacka}
    Let $\phi:\R_+\to\R$ be a bounded, strictly increasing right-continuous function.\comment{DC(45) change as suggested throughout} There exists a stopping time $\tau_{J}$ which \emph{maximizes}
  \begin{align*}\textstyle
    \E\Big[\phi\Big(\sup_{t \le \tau} |B_t|\Big)\Big]
  \end{align*}
    over all solutions to \eqref{SkoSol}, and which is of the form \[  \textstyle\tau_{J} = \inf \Big\{ t > 0: B_t \ge \alpha_-\Big(\sup_{s \le t} |B_s|\Big) \text{ and } B_t \le \alpha_+\Big(\sup_{s \le t} |B_s|\Big)\Big\}\] a.s., for some functions $\alpha_+, \alpha_-$, where $\alpha_+$ is increasing, $\alpha_-$ is decreasing, and $\alpha_+(y) \ge \alpha_-(y)$ for all $y > y_0$, $\alpha_-(y) = -\alpha_+(y) = \infty$ for $y < y_0$, some $y_0 \ge 0$.
\end{theorem}

\begin{proof}
 The proof runs along similar lines to the proof of Theorem~\ref{thm:AY}, when we take $\gamma((f,s)) = -\phi(|f|^*)$ and set $\tilde{\gamma}((f,s)) = \tilde{\phi}(|f|^*)(f(s))^2$ for some bounded and strictly increasing, continuous function $\tilde{\phi}$. Then the statement follows once we see\comment{MH: To be able to write down $\SSG$ we need a $\tilde\gamma$...}
\begin{align*}
\SSG \supseteq \left\{((f,s),(g,t))\in S\times S:f(s)=g(t), |f|^*>|g|^*\right\},
\end{align*}
 define
 \begin{align*}
    \mathcal{R}_\cl & := \left\{(m,x): \exists (g,t) \in \Gamma, |g|^* \le m, g(t) = x\right\}\\
    \mathcal{R}_\op & := \left\{(m,x): \exists (g,t) \in \Gamma, |g|^* < m, g(t) = x\right\}, 
  \end{align*}
    and then take
  $ \alpha_-(m) = \inf \{ x : (m,x) \in \mathcal{R}_\cl\} \text{ and }
    \alpha_+(m)  = \sup\{x:  (m,x) \in \mathcal{R}_\cl\}. $
\end{proof}

\begin{remark}
  We observe that both the results hold for one-dimensional Brownian motion with an arbitrary starting distribution $\lambda$ satisfying the usual convex ordering condition.
\end{remark}

\begin{theorem}[The Perkins Embedding, cf.\ \cite{Pe86}]
  \label{thm:Perkins}
  Suppose $\mu(\{0\}) = 0$. Let $\varphi:\R_+^2\to\R$ be a bounded function which is continuous and strictly increasing in both arguments. There exists a stopping time $\tau_{P}$ which \emph{minimizes}
  \begin{equation*}
    \E\Big[\varphi\Big(\sup_{t \le \tau} B_t,-\inf_{t \le \tau} B_t\Big)\Big]
  \end{equation*}
  over all solutions to \eqref{SkoSol} and which is of the form $\tau_{P} = \inf \big\{ t > 0: B_t \not\in \big(\alpha_+(\ol{B}_t), \alpha_-(\ul{B}_t)\big)\big\}$, for some decreasing functions $\alpha_+$ and $\alpha_-$ which are left- and right-continuous respectively.
\end{theorem}

\begin{proof} Fix a bounded and strictly increasing continuous function $\tilde\varphi:\R^2_+\to\R$ and consider the continuous functions $\gamma((f,s)) = \varphi(\ol f,-\ul f)$ and $\tilde\gamma((f,s)) = -(f(s))^2 \tilde\varphi(\ol{f},-\ul{f})$. Then \eqref{OptSep2} is well posed and by Theorem~\ref{MinimizerExists2} there exists a minimizer $\tau_{P}$. By Theorem \ref{GlobalLocal3}, pick a $\tilde\gamma|\gamma$-monotone set $\Gamma \subseteq S$ supporting $\tau_{P}$. Note that we may assume that $\Gamma$ only contains points such that $\ul{g} < 0 < \ol{g}$, since $\mu(\{0\}) = 0$.

  By a similar argument to that given in the proof of Theorem \ref{thm:AY} we can show
  $$\SSG \supset \{((f,s),(g,t))\in S\times S: f(s)=g(t), (\ol f, - \ul f)< (\ol g, -\ul g)\},$$
  where $(\ol f, - \ul f)< (\ol g, -\ul g)$ iff $(\ol f, - \ul f)\leq (\ol g, -\ul g)$ but not $(\ol f, - \ul f)= (\ol g, -\ul g)$ and$(\ol f, - \ul f)\leq (\ol g, -\ul g)$ refers to the partial order of $\R^2$.

  In addition, consider a path $(g,t) \in S$ such that $\ul{g} < g(t) < \ol{g}$. Then there exists $(f,s) \in S$ such that $f(r) = g(r)$ for $r \le s$, and such that $f(s) = g(t)$, and exactly one of $\ol{f} = \ol{g}$, or $\ul{f} = \ul{g}$. This is true since there must exist a last time that $g(r) = x$ before setting the most recent extremum. In particular, $((f,s),(g,t)) \in \SSG$. It follows that $\Gamma \cap \{(g,t): \ul{g} < g(t) < \ol{g}\} = \emptyset$, that is, any stopped path must stop at a minimum or a maximum.

Now consider the sets:
\begin{align*}
  \mathcal{R}_\cl & = \textstyle\left\{(m,x): \exists (g,t) \in \Gamma, g(t) = x = \ul{g}, \ol{g} \ge m\right\} \cup \left\{(x,i): \exists (g,t) \in \Gamma, g(t) = x = \ol{g}, \ul{g} \le i\right\}\\
  & = \ul{\mathcal{R}}_\cl \cup \ol{\mathcal{R}}_\cl\\
  \mathcal{R}_\op & =\textstyle \left\{(m,x): \exists (g,t) \in \Gamma, g(t) = x = \ul{g}, \ol{g} > m\right\} \cup \left\{(x,i): \exists (g,t) \in \Gamma, g(t) = x = \ol{g}, \ul{g} < i\right\}\\
  & = \ul{\mathcal{R}}_\op \cup \ol{\mathcal{R}}_\op,
\end{align*}
and their respective hitting times by $(\ol{B}_t,\ul{B}_t)_{t\geq 0}$,
denoted $\tau_\cl, \tau_\op$. Since $\Gamma \cap \{(g,t): \ul{g} <
g(t) < \ol{g}\} = \emptyset$, it follows that $\tau_\cl \le \tau_P$
a.s. In addition, an essentially identical argument to that used in the proof of Theorem~\ref{thm:AY} gives $\tau_{P} \le \tau_\op$ a.s.

We now set
  $\alpha_+(m)  = \sup\{x < 0 : (m,x) \in \ul{\mathcal{R}}_\cl\}$, 
 $ \alpha_-(i)  = \inf\{x > 0 : (x,i) \in \ol{\mathcal{R}}_\cl\}.$ 
Then these functions are both clearly decreasing and left- and right-continuous respectively, by definition of the respective sets $\ul{\mathcal{R}}_\cl, \ol{\mathcal{R}}_\cl$. Moreover, it is immediate that 
\[\textstyle \tau_\cl = \inf \left\{ t > 0: B_t \not\in \left(\alpha_+(\ol{B}_t), \alpha_-(\ul{B}_t)\right)\right\},\] 
and we deduce that $\tau_\cl = \tau_\op$ a.s.~ by standard properties of Brownian motion. The conclusion follows.
\end{proof}

\begin{theorem}[Maximizing the range]
  \label{thm:range}
 Let $\varphi:\R_+^2\to\R$ be a bounded function which is continuous and strictly increasing  in both arguments. There exists a stopping time $\tau_{xr}$ which \emph{maximizes}
  \begin{equation*}
    \E\Big[\phi\Big(\sup_{t \le \tau} B_t, -\inf_{t \le \tau} B_t\Big)\Big]
  \end{equation*}
  over all solutions to \eqref{SkoSol}, and which is of the form $\tau_{xr} = \inf \big\{ t > 0: B_t \ge \alpha_-(\ol B_t,-\ul{B}_t)\ \text{or}$ $\ B_t \le \alpha_+(\ol{B}_t,-\ul B_t)\big\}$ for some right-continuous functions $\alpha_-(m,i)$ decreasing in both coordinates and $\alpha_+(m,i)$ increasing in both coordinates.
\end{theorem}
\begin{proof} 
 Our primary objective will be to minimize $\gamma((f,s)) = -\varphi(\ol{f},-\ul{f})$, which is a lsc function on $S$.  We again introduce a secondary minimization problem: specifically, we consider the function $\tilde\gamma((f,s)) = (f(s))^2 \tilde\varphi(\ol{f},-\ul{f})$ for some bounded, continuous and strictly increasing function $\tilde\varphi:\R_+^2\to\R$.  Then \eqref{OptSep2} is well posed and by Theorem~\ref{MinimizerExists2} there exists a minimizer $\tau_{xr}$.  By Theorem \ref{GlobalLocal3}, pick a $\tilde\gamma|\gamma$-monotone set $\Gamma\subset S$ supporting $\tau_{xr}.$

   By a similar argument to that given in the proof of Theorem \ref{thm:AY} we can show $\SSG \supset \{((f,s),(g,t))\in S\times S: f(s)=g(t), (\ol f, - \ul f)> (\ol g, -\ul g)\}$.

    Let $\conv$ denote the convex hull, and write
    \begin{align*}
      I_{\cl}(\ol{b},-\ul{b}) & := \conv\left\{x: \exists (g,t) \in
                                \Gamma, g(t) = x, (\ol{g},-\ul{g})
                                \leq (\ol{b},-\ul{b})\right\}, \\
      I_{\op}(\ol{b},-\ul{b}) & := \conv\left\{x: \exists (g,t) \in
                                \Gamma, g(t) = x, (\ol{g},-\ul{g})
                                <(\ol{b},-\ul{b})\right\}.
    \end{align*}
Then $I_{\cl}, I_{\op}$ are both increasing in both coordinates, and $I_{\cl} \supset I_{\op}$. Write $\tau_{\op} := \inf\{t \ge 0: B_t \in I_{\op}(\ol{B}_t,-\ul{B}_t)\}$, and $\tau_{\cl} := \inf\{t \ge 0: B_t \in I_{\cl}(\ol{B}_t,-\ul{B}_t)\}$. As previously, we deduce that $\tau_{\cl} \le \tau_{xr} \le \tau_{\op}$. If, in addition, we define
    \begin{align*}
      \alpha_+(m,i) & := \sup I_{\op}(m,i)\quad & \alpha_-(m,i) & := \inf I_{\op}(m,i)\\
      \alpha_{+,\cl}(m,i) & := \sup I_{\cl}(m,i)\quad & \alpha_{-,\cl}(m,i) & := \inf I_{\cl}(m,i)
    \end{align*}
    then $\alpha_+, \alpha_-$ satisfy the conditions of the theorem, and 
    \begin{align*}
      \tau_{\op} & = \inf \left\{ t \ge 0 : B_t \ge \alpha_-(\ol B_t,-\ul{B}_t) \mbox{ or } B_t \le \alpha_+(\ol{B}_t,-\ul B_t)\right\}\\
      \tau_{\cl} & = \inf \left\{ t \ge 0 : B_t \ge \alpha_{-,\cl}(\ol B_t,-\ul{B}_t) \mbox{ or } B_t \le \alpha_{+,\cl}(\ol{B}_t,-\ul B_t)\right\}.
    \end{align*}
    To conclude, we need to show that $\tau_{\op} = \tau_{\cl}$. However, we first observe that $\tau_{\op} \ge \sigma$, and $\tau_{\cl} \ge \sigma_{\cl}$, where
    \begin{align*}
      \sigma & := \inf \left\{ t \ge 0 : \alpha_-(\ol B_t,-\ul{B}_t)<\infty \mbox{ or } \alpha_+(\ol{B}_t,-\ul B_t)>-\infty\right\}\\
      \sigma_{\cl} & := \inf \left\{ t \ge 0 : \alpha_{-,\cl}(\ol B_t,-\ul{B}_t)<\infty \mbox{ or } \alpha_{+,\cl}(\ol{B}_t,-\ul B_t)>-\infty\right\}, 
    \end{align*}
    and in fact, $\sigma = \sigma_{\cl}$ a.s. In addition, on $\{\sigma >0\}$ we have $B_{\sigma} \in \{\ol{B}_\sigma, \ul{B}_\sigma\}$. On the set $\{B_\sigma = \ol{B}_\sigma\}$ say, then
    \begin{align*}
      \tau_{\op}  = \inf \{ t \ge \sigma: B_t \le \alpha_{+}(\ol{B}_t, -\ul{B}_\sigma)\}
       = \inf \{ t \ge \sigma: B_t \le \alpha_{+,\cl}(\ol{B}_t,
      -\ul{B}_\sigma)\} \quad a.s.
    \end{align*}
    by the same argument as used at the end of the proof of Theorem~\ref{thm:AY}, and the fact that $\alpha_{+}(m+,i) = \alpha_{+,\cl}(m,i)$, by the definition of the sets $I_{\cl}, I_{\op}$
.\end{proof}

\begin{remark}
  We observe that, in the case of Theorem~\ref{thm:range}, the characterization provided would not appear to be sufficient to identify the functions $\alpha_+, \alpha_-$ given the measure $\mu$. This is in contrast to the constructions of Az\'ema-Yor, Perkins and Jacka, where knowledge of the form of the embedding is sufficient to identify the corresponding stopping rule.

On a more abstract level, uniqueness of barrier type embeddings in a two dimensional phase space can be seen as a consequence of Loynes' argument \cite{Lo70}. More precisely, let $A_t$ be some continuous process and suppose that $\tau_1$ and $ \tau_2$ denote the times when $(A_t, B_t)$ hits a closed barrier type set $R_1$ resp.\ $R_2$.  If $\E[ \tau_1], \E [\tau_2] < \infty$ and both stopping times embed the same measure,  the argument presented in Remark \ref{rem:Loynes} shows that $\tau_1=\tau_2$.
\end{remark}

\begin{remark} 
  In Cox and \OB{} \cite{CoOb11b}, embeddings are constructed which maximize certain double-exit probabilities: for example, to maximize the probability that both $\ol{B}_\tau \ge \ol{b}$ and $\ul{B}_\tau \le \ul{b}$, for given levels $\ol{b}$ and $\ul{b}$. In this case, the embedding is no longer naturally viewed as a barrier type construction; instead, it is natural to characterize the embedding in terms of where the paths with different crossing behaviour for the barriers finish (for example, the paths which only hit the upper level may end up above a certain value, or between two other values). However, it is possible, again using a suitable secondary maximization problem, to show that there exists an optimizer demonstrating the behaviour characterizing the Cox-\OB{} embeddings. (Specifically, if we write $H_{b}((f,s)) = \inf\{t\le s: f(t) = b\}$, $\ul{H} = H_{\ul{b}} \wedge H_{\ol{b}}$ and $\ol{H} = H_{\ul{b}}\vee H_{\ol{b}}$ then the secondary maximization problem \comment{DC(47) changed as suggested.}
$$\tilde\gamma((f,s)) =  1/2 ((f(s)-\ul{H}((f,s)))^2 \1_ {\ul{H} \le s} - ((f(s)-\ol{H}((f,s)))^2 \1_{\ol{H} \le s}$$
 is sufficient to rederive the form of these embeddings.)
\end{remark}

\subsection{The Vallois-embedding and optimizing functions of local time}\label{ValoisSection}
In this section we shall determine the stopping rule which solves 
\begin{align}\label{ValloisOpt}
\inf\{\E[h(\LT_\tau)]: \tau \mbox{ solves } \eqref{SkoSol}\}, 
\end{align}
where $\LT$ denotes the local time of Brownian motion at $0$ and $h$ is a convex or concave function. In many ways, the proof of this result will follow the arguments used in the previous section, however in contrast to the functions considered there, $h(\LT)$ is not defined on $S$ in a straightforward way and hence we need to apply some care in fixing our notions. Moreover local time does not have an $S$\!-continuous modification and hence some additional argument is needed to establish that \eqref{ValloisOpt} admits a minimizer.

We say that a $\G$-adapted process $\LT^x$ is a \emph{local time in $x$} if it is a (right-continuous, increasing) compensator of $|B-x|$ and we suppress $x$ in the case of local time at $0$. This determines $\LT^x$ up to indistinguishability (and clearly the choice of $\LT^x$ is irrelevant for \eqref{ValloisOpt}).

For us it is convenient to allow local time to assume the value $+\infty$ on an evanescent set. Using this convention, Theorem 4.1 implies that there exists a Borel function $L^x:S\to [0,\infty]$ such that $L^x\circ r $ is a (right-continuous, increasing) $\F^0$-predictable local time on Wiener space. We will call such a process $L^x$ a \emph{raw local time} in $x$. We note that the value $+\infty$ cannot be avoided here, see \cite{najnudel_2011}.

\begin{lemma}\label{prop:LTequiv}
Let $L$ be a raw local time in $0$. Then there exists a Borel set $A\subseteq \CRo$, $\W(A)=1$ such that for all  
 \begin{equation*}
 (f,s)\in  U = \{(f,s) \in S : \exists \omega \in A, f = (\omega_r)_{r \le s}\}
 \end{equation*}
we have $L(f,s)<\infty$ and 
 \begin{equation} \label{eq:LTattach}
 (g,t)\mapsto L^{(f,s)}(g,t) := L(f\oplus g, s+t )-L(f,s) \end{equation}
  is a raw local time in $-f(s)$. 
\end{lemma}
\begin{proof}
  Write $V$ for the set of all $(f,s)$ such that $L^{(f,s)}$ is not a raw local time. To understand whether $(f,s)\in V$ we need to check whether or not $ (\omega, t)\mapsto |B_{s+t}(f\oplus \omega)|-L^{(f,s)}(r(\omega,t)) $ defines a martingale. Since this is a Borel property, $V\subseteq S$ is Borel. Hence 
  $$\deb(V):=\{\omega: \exists t, r(\omega, t)\in V\}$$ is analytic and
  thus universally measurable. To prove that $\W(\deb(V))=0$ it is sufficient to show this for any given Borel subset of $\deb(V)$. Suppose for contradiction that $\W(E)>0$ for some Borel set $E\subseteq \deb(V)$. By the optional section theorem this implies that there exists an $\F^a$-stopping time $\tau$ such that $\W(\tau<\infty)>0$ and $(\omega, \tau(\omega))\in r^{-1}(V)$ whenever $\tau(\omega)<\infty$. Upon requiring this only a.s.\ we may of course assume that $\tau $ is an $\F^0$-stopping time.

Given $H=G\1_{\ll\tau,\infty\ll}$ for bounded $\F^0_\tau$-measurable $G$ it follows from usual properties of local time that 
$$ t\mapsto (H\cdot (|B|-L\circ r))_t =G\big[(|B|-L\circ r)_{t}-(|B|-L\circ r)_{\tau\wedge t} \big]$$ is a martingale. As $G$ was arbitrary,  $$ (\omega, t)\mapsto |B_{\tau(\omega')+t}(\omega'\lsub{[0,\tau(\omega')]}\oplus \omega)|-L^{(\omega'\lsub{[0,\tau(\omega')]},\tau(\omega'))}(r(\omega,t)) $$ defines a martingale for almost all $\omega'$, $\tau(\omega')<\infty$, contradicting $\W(\deb(V))>0$.

It follows that $\W(\deb(V))=0$, hence we may pick a Borel set $A\subseteq  \deb(V)\complement$ with $\W(A)=1$ such that \eqref{eq:LTattach} holds. 
\end{proof}

Our next goal is to verify that \eqref{ValloisOpt} admits an optimizer.

\begin{lemma}\label{lem:LocalTimes}  Let $L$ be a raw local time, and define local time on $\Omega$ by $\LT_t(\omega):=L\circ r(B(\omega), t)$. Let $ \xi_n, \xi \in \RST(\mu)$ and let $\rho_n, \rho$ be their representatives on $\Omega$ as in Lemma~\ref{lem:equivOpt}.
  If $\xi_n \to \xi$ weakly then $\LT_{\rho_n} \to \LT_{\rho}$ in $L^1(\Omega, \P)$. 
\end{lemma}

\begin{proof}[Proof of Lemma~\ref{lem:LocalTimes}] 
  As a consequence of Proposition~\ref{prop:LTConv} we have that $\rho_n \wedge \rho \to \rho, \rho_n \vee \rho \to \rho$ in probability. By \eqref{DiffRep}, $\int (L\circ r)(\omega,t) \, \xi(d\omega, dt)= \E[\LT_\rho]$.

  For every embedding $\xi'\in \RST (\mu')$, $\xi'(L\circ r) = \E[\LT_{\rho'}]= \int |x| \, d\mu'(x)$ by Lemma~\ref{MinimalLPT}.  Write $ \mu_n$ for the law embedded by $\rho_n \wedge \rho$. Then $\mu_n \to \mu$ weakly, and $\LT_{\rho_n\wedge \rho} \leq \LT_{\rho}$, so (again using Lemma~\ref{MinimalLPT}) $\E[\LT_{\rho_n\wedge \rho}] = \int |x|\, d\mu_n \to \int |x|\, d\mu$ and hence $\E [\LT_{\rho_n\wedge \rho}] \to \E [\LT_{\rho}]$. This implies that $\LT_{\rho_n\wedge \rho} \to \LT_{\rho}$ in $L^1(\Omega, \P)$. Since $\LT_{\rho_n\vee \rho} + \LT_{\rho_n \wedge \rho} = \LT_\rho+ \LT_{\rho_n}$ a.s.\ we also find that $\E [\LT_{\rho_n \vee \rho}] = \E \left[ \LT_{\rho_n}+ ( \LT_{\rho}-\LT_{\rho_n})_+\right] = \E [ \LT_{\rho}]+ \E \left[ ( \LT_{\rho}- \LT_{\rho_n})_+\right] \to  \E [ \LT_{\rho}]$, where we used that $\xi_n,\xi\in\RST(\mu)$.  Thus $\LT_{\rho_n\vee \rho} \to \LT_{\rho}$ in $L^1( \Omega, \P)$.  Combining these results, we see that $\LT_{\rho_n} \to \LT_{\rho}$ in $L^1(\Omega, \P)$.
\end{proof}

\begin{corollary}\label{cor:ValloisOpt}
  Let $h:[0,\infty) \to \R$ be continuous bounded. Then there exists an optimizer for \eqref{ValloisOpt}. Moreover, if $\tilde\gamma(f,s) = e^{-L(f,s)}f^2(s)$ or $\tilde\gamma(f,s) =- e^{-L(f,s)}f^2(s)$ also the secondary minimization problem \eqref{OptSep2} admits a solution.
\end{corollary}
\begin{proof}
  Let $L$ be a raw local time. We first observe that $(\LT_t)_{t \ge 0}:= (L\circ r((B_t)_{t \ge 0},t))_{t \ge 0}$ is (indistinguishable from) the local time of $(B_t)_{t \ge 0}$ on $(\Omega, \G, (\G_t)_{t\geq 0}, \P)$. By Lemma~\ref{lem:equivOpt} there exists a sequence $\xi_1, \xi_2, \xi_3, \ldots \in \RST(\mu)$ such that \[\textstyle V^* = \lim \int h(\LT_t(\omega)) \,\xi_n(d\omega,dt) = \inf\{\E[h(\LT_\tau)]: \tau \mbox{ solves } \eqref{SkoSol}\}.\] Possibly passing to a subsequence $\xi = \lim_n \xi_n$ satisfies $\int h(\LT_t(\omega)) \,\xi(d\omega,dt) = V^*$ by Lemma~\ref{lem:LocalTimes}. Moreover (again by Lemma~\ref{lem:equivOpt}) there exists a $\G$-stopping time $\tau^*$ such that $\E[h(\LT_{\tau^*})] = \int h(\LT_t(\omega)) \,\xi(d\omega,dt)$. Hence, $\Opt_\gamma$ is non-empty and closed.  The second assertion follows by the same argument.
\end{proof}

We are now able to show:
\begin{theorem}
  \label{thm:Vallois} 
  Let $h:[0,\infty] \to \R$ be a bounded, strictly concave function and $\LT$ the local time of $B$ at $0$. 
  \begin{enumerate}
  \item
    There exists a stopping time $\tau_{V-}$ which \emph{maximizes}
    \begin{equation*}
      \E\left[h\left(\LT_\tau\right)\right]
    \end{equation*}
    over the set of all solutions to \eqref{SkoSol}, and which is of the form
    \[\tau_{V-} = \inf \left\{ t > 0: B_t \notin
      \left(\alpha_-\left(\LT_t\right),\alpha_+\left(\LT_t\right)\right)\right\}
    \text{ a.s.,}\]
    for some decreasing function $\alpha_+\geq 0$ and
    increasing  function $\alpha_-\leq 0$.
  \item
    There exists a stopping time $\tau_{V+}$ which \emph{minimizes}
    \begin{equation*}
       \E\left[h\left(\LT_\tau\right)\right]
    \end{equation*}
    over the set of all solutions to \eqref{SkoSol}, and which is of the form
    \begin{align*} \tau_{V+} = Z \wedge \inf \left\{ t > 0: B_t \notin \left(\alpha_-\left(\LT_t\right),\alpha_+\left(\LT_t\right)\right)\right\}, \text{ a.s.}
    \end{align*}
    for some increasing function $\alpha_+\geq 0$, and some decreasing  function $\alpha_-\leq 0$, and a $\{0,\infty\}$-valued  $\G_0$-measurable random variable $Z$.
  \end{enumerate}
\end{theorem}

\begin{proof}
  We consider the second case, under the additional assumption that $0<\mu(\{0\}) <1$, the other cases being slightly simpler.  As above, we let $L$ be a raw local time and observe that $(\LT_t)_{t \ge 0}:= (L\circ r((B_t)_{t \ge 0},t))_{t \ge 0}$ is (indistinguishable from) the local time of $(B_t)_{t \ge 0}$ on $(\Omega, \G, (\G_t)_{t\geq 0}, \P)$.

  Applying Corollary \ref{cor:ValloisOpt} and Theorem~\ref{GlobalLocal3} to the optimizations corresponding to $\gamma(\omega,t) = h(L\circ  r(\omega,t))$ and $\tilde\gamma(\omega,t) = e^{-L\circ r(\omega,t)}\omega^2_t$ we obtain a minimizer $\tau_{V+}$ and a
 $\tilde\gamma|{\gamma}$-monotone set $\Gamma\subset S$ supporting $\tau_{V+}$.

 Recall the set $A \subset \CRo$ given by
 Lemma~\ref{prop:LTequiv}.  By projection the set
 \begin{equation*}
   U = \{(f,s) \in S : \exists \omega \in A, f = (\omega_r)_{r \le s}\}
 \end{equation*}
 is universally measurable and since $\tau_{V+}$ is a finite stopping time, $\P( ((B_t)_{t\leq \tau_{V+}},\tau_{V+}) \in U)=1$. 
 Passing to an appropriate subset if necessary, we may also assume that $U$ is Borel.  We may therefore assume $\Gamma \subseteq U$, and it then also follows that $\Gamma^{<} \subseteq U$.

 By a similar argument to the previous cases we can show that
 \begin{equation}\label{eq:LTSBP}
   \SSG \supset\{((f,s),(g,t))\in U\times U:f(s)=g(t), L(f,s)<
   L(g,t)\},
 \end{equation}
 where Lemma~\ref{prop:LTequiv}
 guarantees that local time of paths is well-behaved following a path-swapping
 operation. In particular, since both $f$ and $g$ belong to $U$, it
 follows that \eqref{eq:LTattach} holds, and \eqref{eq:LTSBP} is a
 direct consequence of this.

 Define the sets
\begin{align*}
  \mathcal{R}_\op & := \left\{(l,x): \exists (g,t) \in \Gamma, g(t) = x, L(g,t) >l \right\},\\
  \mathcal{R}_\cl & := \left\{(l,x): \exists (g,t) \in \Gamma, g(t) = x, L(g,t) \ge l \right\},
\end{align*}
and  the corresponding stopping times 
\begin{align*}
  \tau_\op^*  := \inf\left\{t \ge 0: \LT_t > 0, (\LT_t,B_t) \in \mathcal{R}_{\op}\right\},\ 
  \tau_\cl^*  := \inf\left\{t \ge 0: \LT_t >0,  (\LT_t,B_t) \in \mathcal{R}_{\cl}\right\}.
\end{align*} 
Strictly speaking,  the random times on the right-hand side only define stopping times in the augmented filtration (by the D\'ebut Theorem), however by Theorem~\ref{indistinguishable predictable}, this is sufficient to find almost surely equal $\G$-stopping times.

Since $(\Gamma^{<} \times \Gamma) \cap \SG_2 = \emptyset$ and $(0,0) \in \Gamma^{<}$ ($\Gamma$ contains a non-trivial element since $\mu(\{0\}) < 1$) then $(l,0) \not\in \Gamma$ for any $l \ge 0$. It follows that $\P(\tau_{V+} = 0) = \mu(\{0\})$.

We now consider $\tau_{V+}$ on $\{\tau_{V+} > 0 \}$. Note that $\{\tau>0\}= \{ \LT_\tau >0\}$ a.s., for any stopping time $\tau$ and hence in particular $\{\tau_{V+} > 0 \}= \{ \LT_{\tau_{V+}} >0\}$ a.s. 
Then on $\{\tau_{V+} >0\}$, $\tau_\cl^* \le \tau_{V+} \le \tau_\op^*$ a.s., and hence $\P(\tau_{V+} \le \tau_{\op}^*) = 1$.  Define $\alpha_+(l) = \inf\{ x>0: (l,x) \in \mathcal{R}_\op\}$ and $\alpha_-(l) = \sup\{ x<0: (l,x) \in \mathcal{R}_\op\}$. 

If either of $\alpha_-(\eta) = 0$ or $\alpha_+(\eta) = 0$ for some $\eta>0$, then $\tau_{\op}^* = 0$ a.s. Since $\tau_{V+} \le \tau_\op^*$ and $\P(\tau_{V+} >0) >0$ we must therefore have $\alpha_+(\eta)>0, \alpha_-(\eta)<0$ for $\eta>0$. In addition, $\alpha_+(l)$ is clearly right-continuous and increasing, so it must have at most countably many discontinuities, and similarly for $\alpha_-(l)$. We can write
\begin{equation*}
  \inf\left\{ t: \LT_t>0, B_t \not\in  \left(\alpha_-\left(\LT_t-\right),\alpha_+\left(\LT_t-\right)\right)\right\} \le \tau_{\cl}^* \le \tau_{\op}^* \le \inf\left\{ t: \LT_t>0, B_t \not\in  \left[\alpha_-\left(\LT_t\right),\alpha_+\left(\LT_t\right)\right]\right\}
\end{equation*}
and observe that (by standard properties of Brownian motion) the
stopping times on the left and right are almost surely equal (since
there are at most countably many discontinuities, and $\alpha_+(l)$
and $\alpha_-(l)$ are bounded away from zero on $[\eta,\infty)$  for $\eta>0$). It
follows that $\tau_{V+} = \inf\left\{t: \LT_t >0,  B_t \not\in
  \left(\alpha_-\left(\LT_t\right),\alpha_+\left(\LT_t\right)\right)\right\}
$ on $\{\tau_{V+} > 0\}$, and 
we deduce that $\tau_{V+}$ is zero with probability $\mu(\{0\})$, and,
conditional on being greater than zero,
$ \tau_{V+} =\inf \left\{ t > 0: B_t \notin
  \left(\alpha_-\left(\LT_t\right),\alpha_+\left(\LT_t\right)\right)\right\}$ a.s.
\end{proof}

\begin{remark}
  The arguments above extend from local time at $0$ to a general continuous additive functional $A$. Recalling that $\LT^x$ denotes local time in $x$, $A$ can be represented in the form $A_t:=\int_0^t \LT_s^x\, dm_A(x)$.  Let $f$ be a convex function such that $f'' = m_A$ in the sense of distributions. If $\int f\, d\mu < \infty$, then Lemma \ref{prop:LTConv} still holds with $A$ in place of $\LT$; the above proof is easily adapted to the more general situation.

  In this manner, we deduce the existence of optimal solutions to \eqref{SkoSol} for functions depending on $A$.  By analogy with Theorem \ref{thm:Vallois} this can be used to generate (inverse-/cave-) barrier type embeddings of various kinds. Other generalizations and variants may be considered in a similar manner. We leave specific examples as an exercise for the reader.
\end{remark}

\subsection{Root and Rost Embeddings in Higher Dimensions} \label{sec:root-rost-embeddings}

In this section we consider the Root and Rost constructions of Sections~\ref{sec:root-embedding} and \ref{sec:rost-embedding} in the case of $d$-dimensional Brownian motion with general initial distribution, for $d\ge 2$. In $\R^d$, since Brownian motion is transient, it is no longer straightforward to assert the existence of an embedding. In general, \cite{Ro71} gives necessary and sufficient conditions for the existence of an embedding, and without the additional condition that $\E[\tau] < \infty$. In the Brownian case, Rost's conditions for $d \ge 3$ can be written as follows.
 There exists a stopping time $\tau$ such that $B_0 \sim \lambda$ and $B_\tau \sim \mu$ if and only if for all $y \in \R^d$
\begin{equation}\label{d-potential}\textstyle
  \int u(x,y)\, \lambda(dx) \le \int u(x,y)\, \mu(dx), \text{ where } u(x,y) = 
    |x-y|^{2-d}.
\end{equation}
However, it is not clear that such a stopping time will satisfy the condition \begin{align}\label{d-minimal}\E[\tau] = 1/d \textstyle\left(\int |x|^2\, (\mu-\lambda)(dx)\right).\end{align}
 As a result, it is not straightforward to give simple criteria for the existence of a solution in $\RST(\mu)$. 

In the case $d=2$ it follows from Falkner's results \cite{Fa81} that the Skorokhod problem admits a solution (i.e.\ $\RST(\mu)\neq \emptyset$) if \eqref{d-potential} is satisfied for $u(x,y)=-\ln |x-y|$ and then \eqref{d-minimal} applies.   

In either case, assuming that we do have a solution satisfying \eqref{d-minimal}, then the existence result as well as the monotonicity principle  carry over to the present setup (with identical proofs) and we are able to state the following:

\begin{theorem}
  \label{thm:RootRd}
  Suppose $\RST(\mu)$ is non-empty. If $h$ is a strictly convex function and $\hat{\tau} \in \RST(\mu)$ minimizes $\E[h(\tau)]$ over $\tau \in \RST(\mu)$ then there exists a barrier $\mathcal{R}$ such that $\hat{\tau} = \inf \{ t > 0 : (B_t,t) \in \mathcal{R}\}$ on $\{\hat\tau >0\}$ a.s.
\end{theorem}

The proof of this result is much the same as that of Theorem~\ref{RootThm}, except we no longer show that $\tau_\cl = \tau_\op$. In higher dimensions with general initial laws, it is easy to construct examples where there are common atoms of $\lambda$ and $\mu$, but where the size of the atom in $\lambda$ is strictly larger than the atom of $\mu$. By the transience of the process, it is clear that the optimal (indeed, only) behaviour is to stop mass starting at such a point immediately with a probability strictly between $0$ and $1$, however the stopping times $\tau_\cl$ and $\tau_\op$ will always stop either all the mass, or none of this mass respectively. For this reason, we do not say anything about the behaviour of $\hat\tau$ when $\hat\tau = 0$. Trivially, the above result tells us that the solution of the optimal embedding problem is given by a barrier if there exists a set $D$ such that $\lambda(D) = 1 = \mu(D\complement)$.

\begin{proof}[Proof of Theorem~\ref{thm:RootRd}]
  The first part of the proof proceeds similarly to the proof of Theorem~\ref{RootThm}. In particular, the set of stop-go pairs is given by
$$\BP \supset\{((f,s),(g,t))\in S\times S:f(s)=g(t), s>t\}$$
and we define the sets $\mathcal{R}_\cl, \mathcal{R}_\op$ and the stopping times $\tau_\cl, \tau_\op$ as above. We then fix $\delta>0$, and consider the set $\{\hat\tau\ge \delta\}$.
  Given $\eta \ge 0$, we define $B^{-\eta}_t = B_{t+\eta}$, for $t \ge -\eta$ and set 
\[  \textstyle
\tau^{\eta,\delta}_\cl := \inf\{t \ge \delta: (t,B_t^{-\eta}) \in \mathcal R_\cl\}.\] 
Then $\tau_\cl^{\eta,\delta} \ge \delta$, and for any $\eps >0$, there
exists $\eta >0$ sufficiently small that 
  $d_{TV}(B^{-\eta}_{\delta},B_{\delta}) < \eps,$ where $d_{TV}$ denotes the total variation distance. By the Strong Markov property of Brownian motion, it follows that
   $
     d_{TV}(B^{-\eta}_{\tau_\cl^{\eta,\delta}}, B_{\tau_\cl^{0,\delta}}) < \eps$.
   In particular, the law of $B^{-\eta}_{\tau_\cl^{\eta,\delta}}$ converges weakly to the law of $B_{\tau_\cl^{0,\delta}}$ as $\eta \to 0$. Thus
   \begin{equation*} \textstyle
     \tau_\cl^{\eta,\delta} = \inf \{ t \ge \eta+\delta: (t-\eta,B_t) \in \mathcal{R}_\cl\},
   \end{equation*}
   so $\tau_\cl^{\eta,\delta} \ge \tau_{R}^{0,\delta}$, and moreover, $\tau_\cl^{\eta,\delta} \to \tau_\op^{0,\delta}$ a.s.~as $\eta \to 0$.  Hence, $B^{-\eta}_{\tau_\cl^{\eta,\delta}} \to B_{\tau_\op^{0,\delta}}$ in probability, as $\eta \to 0$, so we have weak convergence of the law of $B^{-\eta}_{\tau_\cl^{\eta,\delta}}$ to the law of $B_{\tau_\op^{0,\delta}}$, and hence $\textstyle{B_{\tau_\op^{0,\delta}} \sim B_{\tau_\cl^{0,\delta}}}$.
   We now observe that, by an essentially identical argument to that in the proof of Theorem~\ref{RootThm}, we must have $\tau_\cl^{0,\delta} \le \hat{\tau} \le \tau_\op^{0,\delta}$ on $\{\hat{\tau} \ge \delta\}$. However, in the argument above, we know that $\tau_\cl^{0,\delta} \le \hat \tau \le \tau_\op^{0,\delta}$, and $\tau_\cl^{\eta,\delta} \to_{\mathcal{D}} \tau_\cl^{0,\delta}$ and $\tau_\cl^{\eta,\delta} \to _{\mathcal{D}} \tau_\op^{0,\delta}$  as $\eta \to 0$ (where $\mathcal{D}$ denotes convergence in distribution). It follows that $\tau_\cl^{0,\delta} =_{\mathcal{D}} \tau_\op^{0,\delta}$ and hence $\tau_\cl^{0,\delta} = \tau_\op^{0,\delta}$ a.s. In particular, $B_{\tau_\cl^{0,\delta}} = B_{\tau_\op^{0,\delta}} = B_{\hat \tau}$ on $\{\hat\tau \ge \delta\}$. Letting $\delta \to 0$ we observe that $\tau_\op^{0,\delta} \to \tau_\op$, and hence the required result holds on taking $\mathcal{R}=\mathcal{R}_\op$.
\end{proof}

We now consider the generalization of the Rost embedding. Recall that $(\min(\lambda, \mu))(A) := \inf_{B \subseteq A} \left(\lambda(B)+ \mu(A\setminus B)\right)$ defines a measure.
\begin{theorem}
  \label{thm:Rost_General}
  Suppose $\lambda, \mu$ are measures in $\R^d$ and $\hat{\tau} \in \RST(\mu)$ maximizes $\E[h(\tau)]$ over all stopping times in $\RST(\mu)$, for a convex function $h: \R_+ \to \R$, with $\E[h(\tau)]<\infty$. Then $\P(\hat{\tau}=0, B_0 \in A) = (\min(\lambda, \mu))(A)$, for $A \in \mathcal{B}(\R)$, and on $\{\hat{\tau}>0\}$, $\hat{\tau}$ is the first hitting time of an inverse barrier.
\end{theorem}

\begin{proof}
  We follow the proof of Theorem~\ref{RostThm} to recover the set of stop-go pairs given by
$$\BP\supset\{((f,s),(g,t))\in S\times S:f(s)=g(t), s<t\}$$
and the sets $\mathcal{R}_\op$ and $\mathcal{R}_\cl$, and their corresponding hitting times $\tau_\op, \tau_\cl$. For $0 \le \eta \le \delta$, we define in addition the stopping times
  \begin{align*}
    \tau_\cl^{\eta,\delta}  := \inf \{ t \ge \delta: (t,B_t^\eta) \in \mathcal{R}_\cl\},\ 
    \tau_\op^{\eta,\delta}  := \inf \{ t \ge \delta: (t,B_t^\eta) \in \mathcal{R}_\op\},
  \end{align*}
  where $B_t^\eta = B_{t-\eta}$, for $t \ge \eta$.

  It follows from an identical argument to that in the proof of Theorem~\ref{RostThm} that $\tau_\cl^{0,\delta} \le \hat{\tau} \le \tau_\op^{0,\delta}$ on $\{\hat{\tau} \ge \delta\}$. However, by similar arguments to those used above, we deduce that $\tau_\op^{0,\delta}$ and $\tau_\cl^{0,\delta}$ have the same law on $\{\hat{\tau} \ge \delta\}$, and hence that $\hat{\tau} = \tau_\op^{0,\delta}$ on this set, and then by taking $\delta \to 0$, we get $\hat{\tau} = \tau_\op$ on $\{\hat{\tau}>0\}$.

To see the final claim, we note that trivially $\P(\hat{\tau}=0, B_0
\in A) \le (\min(\lambda, \mu))(A)$. If there is strict inequality,
then there exist some paths in $\Gamma$ which start at $x \in A$, and
paths in $\Gamma$ which stop at $x$ at strictly positive time, constituting a stop-go pair and therefore violating the monotonicity principle.
\end{proof}

\begin{remark}
  We observe that the arguments of Remark~\ref{rem:Loynes} can be applied again in this context. However, one needs to be a little more careful, since it is necessary to take the \emph{fine closure} of the barriers with respect to the fine topology for the processes $(t,B_t)_{t\geq 0}$. With this modification in place, the argument of Loynes can be easily adapted to show that the (finely closed versions) of the barriers in Theorems~\ref{thm:RootRd} and \ref{thm:Rost_General} are unique in the sense of Remark~\ref{rem:Loynes}.
\end{remark}

\subsection{An optimal Skorokhod embedding problem which admits only randomized solutions.}

By analogy with optimal transport, we might interpret a `natural stopping time'  (i.e.\ a stopping time wrt to the Brownian filtration) which solves \eqref{IntPri} as a Monge-type solution whereas stopping times which depend on additional randomization are of  Kantorovich-type. With the exception of the Rost solution, all optimal stopping times encountered in the previous section are natural stopping times, and in the Rost case external randomization is only needed at time $0$. One might ask whether the optimal Skorokhod embedding problem always admits a solution $\tau$ which is natural on $\{\tau>0\}$. We sketch an example, showing that this is not the case:
\begin{example}
There exist an absolutely continuous probability $\mu$ and a continuous adapted process $\gamma_t=\gamma((B_s)_{s\leq t})$ with values in $[0,1]$ such that \eqref{IntPri} admits only randomized solutions. 
\begin{proof} 
Define the stopping time $\sigma:= \inf \{t \ge 0: B_t^2 + t^2 \ge 1\}$, the first time the  Brownian path leaves the right half of the unit disc.
Write $(C(0,\sigma), \W_\sigma)$ for the space of continuous functions up to time $\sigma$, equipped with the corresponding projection of Wiener measure. Pick an isomorphism 
$$l:(C(0,\sigma), \W_\sigma)\to ([2,3], \leb)$$ of standard Borel probability spaces. Using some extra randomization (independent of $\F^B$) we define a  stopping time $\tau$ such that
\begin{enumerate}
\item $\tau=\sigma$ with probability $1/2$, 
\item otherwise $\tau$ stops the first time the Brownian path reaches the level $\pm l((B_s)_{s\leq \sigma})$.
\end{enumerate}
We then define $\mu:= \law(B_\tau)$ and pick $\gamma$ to be a function which equals $0$ on paths which are stopped by $\tau$ and is strictly positive otherwise; clearly we can  do this in such a way that  $ \gamma$ has continuous paths. 

Write $\hat\tau$ for the randomized stopping time $\RST(\mu)$ corresponding to $ \tau$. 
It is then straightforward to see that $\hat \tau $ is the unique solution of \eqref{IntPri}. Thus,   the optimal Skorokhod embedding problem admits no (non-randomized) solution in the natural filtration of $B$.
\end{proof}
\end{example}
In optimal transport it is a difficult and interesting problem to understand under which conditions transport problems admit solutions of Monge-type. An interesting subject for future research would be to understand when Monge-type solutions exist for the optimal Skorokhod embedding problem.

\section{Skorokhod Embedding for Feller processes}\label{sec:feller}
In this section we discuss the extension of our results to the embedding problem for a continuous Feller process $Z$, with values in $\R^d$ and  $Z_0\sim \lambda$.
Throughout we suppose that $Z$ is defined on a stochastic basis $(\Omega,\G,(\G_t)_{t\geq0},\P)$ which is sufficiently rich to support  a uniformly distributed $\G_0$-random variable independent of $Z$. 
Given a probability $\mu\in\mathcal P(\R^d)$ the analogue of \eqref{SkoSol} is to construct a stopping time $\tau$ such that 
\begin{align}\label{eq:FellerSEP}\tag{SEP$^Z$}
 Z_\tau \sim \mu, \quad \tau \mbox{ is minimal.}
\end{align}
Recall from \eqref{eq:minimalDefn} that a stopping time $\tau$ is
minimal iff for any stopping time $\tau'$ such that
$Z_{\tau'}\sim Z_{\tau}$ then $\tau'\leq\tau$ implies $\tau'=\tau$ a.s. If $Z$ is a one dimensional Brownian motion and $\mu$ has second moment, minimality of $\tau$ is equivalent to $\E[\tau]<\infty.$ 
Working in higher dimensions with general starting law we redefine
$$\S:=\{(f,s): f\in C([0,s],\R^d) \}.$$  
Given a function $\gamma:\S\to\R$ the optimal Skorokhod embedding problem for $Z$ is to construct a stopping time optimizing
\begin{align}\label{OptSEPZ}\tag{OptSEP$^Z$}
 P_\gamma^Z:=\inf\{\E[\gamma((Z_s)_{s\leq\tau},\tau)] : \tau \mbox{ solves } \eqref{eq:FellerSEP}\}.
\end{align}
(As above, the value of $P_\gamma^Z$ does not depend on the underlying stochastic basis provided it supports a uniformly distributed random variable independent of $Z$.)

Most of the arguments required to establish our main results are abstract and carry over to the present setup.  In fact, only the parts building on the condition $\E[\tau]<\infty$ need to be adjusted to account for the more general condition of $\tau$ being minimal.  Therefore, to establish Theorems \ref{MinimizerExists}, \ref{DualityIntro}, and \ref{GlobalLocal} in the general Feller setup, we need the crucial Assumption \ref{ass:CompactPlus} below which we verify in a number of natural examples in Section \ref{sec:ex}.

\begin{assumption} \label{ass:CompactPlus}
From now on we \emph{assume} that \eqref{eq:FellerSEP} admits a solution and either
\begin{enumerate}
\item[(1)] that there exist continuous functions $h:\R^n\to \R$ and 
$\zeta:S\to\R$  such that: 

\begin{itemize} 
\item $\zeta_t:=\zeta((Z_s)_{s\leq t},t)$ is strictly increasing, $\zeta_0=0$, $\lim_{t\to \infty} \zeta_t=\infty$,  $\P$-a.s.\ and
\item $X_t := h(Z_t)-\zeta_t$ is a martingale and
  $(X^\tau_{ t})_{t\geq 0}$ is uniformly integrable for all $\tau$ solving \eqref{eq:FellerSEP}, or
    \end{itemize}

   \comment{DC(51) changed as suggested. MB: And now rewritten completely.}
 \item[(2)] that whenever $\tau$ is a finite stopping time satisfying $Z_\tau \sim \mu$ then $\tau$ is minimal\comment{DC(52) Notation introduced below Definition ??. }  and there is an increasing function $G:\R_+\to \R$, $\lim_{t\to \infty}G(t) =\infty$ which satisfies
\begin{align}\label{TimeLike2}
\sup\{\E[G(\tau)]: \tau \mbox{ solves \eqref{eq:FellerSEP}}\} =: V <\infty.
\end{align}

\end{enumerate}
\end{assumption}
The existence of a function $G$ such that \eqref{TimeLike2} holds is equivalent to 
\begin{align}\label{TimeLike2.5}
\{\tau: \tau \mbox{ solves \eqref{eq:FellerSEP}}\}  \mbox{ is bounded in probability}.
\end{align}
In fact, it is straightforward to see that we would arrive at an equivalent condition when replacing the deterministic function $G$ by a stochastic process $(\zeta_t)_{t\geq 0}$ as in Case (1). 

Note also that in Case (1) of Assumption \ref{ass:CompactPlus},  $\tau $ with $Z_\tau \sim \mu$ is minimal if and only if 
\begin{align}\label{TimeLike}
\textstyle \E[ \zeta_\tau]  = \int h\,d\mu - \int h\, d\lambda=:V \ \ (< \infty).
\end{align}
Under Assumption~\ref{ass:CompactPlus}, our main results extend to  continuous Feller processes: 
\begin{theorem}\label{ZMinimizerExists}
  If $\gamma:\S\to \R$ is lsc and bounded from below,  \eqref{OptSEPZ} admits a minimizer.
\end{theorem}

\begin{theorem}\label{ZDuality}
  Let $\gamma: S \to \R$ be lsc and bounded from below. Then we have the duality relation $P^Z_\gamma= D^Z_\gamma$ for $D^Z_\gamma :=\sup\int \psi(y) \, d\mu(y)$, where the supremum is taken over all continuous $\psi\in L^1(\mu)$ such that there exists a continuous bounded martingale $M$ with $\E[M_0] = 0$ and a decreasing process $A$ with $\E[A_\tau]\geq 0$ for all solutions $\tau$ of \eqref{eq:FellerSEP} and almost surely for all $t\geq 0$
\begin{align}\label{ZDualConstraint}
M_t + A_t + \psi(Z_t)\leq \gamma((Z_s)_{s\leq t},t).
\end{align}
Moreover, in Case (1) of Assumption~\ref{ass:CompactPlus}, the process $A$ may be assumed to be zero at the expense of assuming that $(M_{\tau \wedge t})_{t\geq 0}$ is only uniformly integrable for all $\tau$ solving \eqref{eq:FellerSEP}.
\end{theorem}

\begin{theorem}\label{ZMP}
  Let $\gamma:\S\to \R$ be Borel measurable.  If \eqref{OptSEPZ} is well posed and $\tau$ is an optimizer, there exists a $\gamma$-monotone Borel set $\Gamma\subset S$ such that $\P$-a.s.
$$((Z_t)_{t\leq \tau}, \tau)\in\Gamma.$$
\end{theorem}
\begin{remark}
  \begin{enumerate}
  \item  Of course, the analogues  of the secondary optimization results, Theorems  \ref{MinimizerExists2} (on existence of a minimizer) and \ref{GlobalLocal3} (monotonicity principle), carry over to the present setup with the obvious changes.
  \item The continuity of $\zeta$ on $S$ which was imposed in Assumption~\ref{ass:CompactPlus} (1) is not required in Theorems \ref{ZMinimizerExists} and \ref{ZMP}.
  \item The condition $0 < \E[\sigma] < \infty$ in
    Definition~\ref{def:SG} should be replaced by considering all
    stopping times with $0 < \E[\zeta((B_s)_{s \leq
      \sigma},\sigma)]<\infty$ in case (1) of
    Assumption~\ref{ass:CompactPlus}, or $0<\E[G(\tau)] < \infty$ in
    case (2).  In addition, the expectation should be taken over the law of the Feller process started at $f(s) = g(t)$.
  \end{enumerate}
\end{remark}

\subsection{Sketch of proofs}\label{FellerProofs}

As in Section 3 we consider the canonical 
 setup $(C(\R_+,\R^d),\cF^0,\Q)$ (where $\Q$ denotes the law of the Feller process) and we write $Y$ for the canonical process.
 It follows from continuity of $Y$ (resp.\ $Z$) and the Feller property that the $\F^a$-optional and the $\F^a$-predictable $\sigma$-algebra on the canonical space agree; similarly Proposition \ref{prop:very cont} on the definition of $S$\!-continuous martingales extends to the present context. 
We define  $\RST$, $\TRST$ and related notions as before with $\Q$ replacing $\W$.  We say that $\xi\in\RST$ is a minimal embedding of $\mu$ if the corresponding stopping time $\rho$ (cf.~\eqref{eq:rhodefn}) on the enlarged probability space $(C(\R_+,\R^d)\times [0,1],\oQ)$ constitutes a minimal embedding. (Representing randomized stopping times as in Theorem 
\ref{thm:equiv RST} (1), the stopping time $\xi$ constitutes a minimal embedding iff there is no randomized stopping time $\xi'\neq \xi$ embedding the same measure which satisfies $A^{\xi'}\geq A^\xi$.) For $\mu\in\mathcal P(\R^d)$ we define $\RST(\mu)$   to be the set of all minimal  randomized stopping times embedding the measure $\mu$.

Recalling the argument from Theorem \ref{ComSol}, we see that the existence of a function $\zeta : S \to \R$ such that $\zeta \circ r$ increases to $\infty$ and $\sup_{\xi\in \RST(\mu)} \xi (\zeta \circ r)<\infty$  implies that $\RST(\mu)$ is compact. (Vice versa, if $\RST(\mu)$ is compact then such a function exists and can be chosen so that $\zeta \circ r$ is deterministic). Hence, by \eqref{TimeLike} resp.\ \eqref{TimeLike2},  $\RST(\mu)$ is compact.

\begin{proof}[Proof of Theorem \ref{ZMinimizerExists}]
 The argument follows the proof of Theorem~\ref{MinimizerExistsP} line by line.
\end{proof}  

\begin{proof}[Proof of Theorem \ref{ZDuality}]
We give the argument in the case $\lambda=\delta_0$ for ease of exposition. Setting $h= \zeta \circ r$ resp.\ $h= G\circ T$ (and using identical arguments as previously) we obtain the following extension of Proposition \ref{super hedging}: 

For $c:\CRo\times \R_+\times \R\to \R \cup \{\infty\}$ lsc, predictable and bounded from below
\begin{align}\label{NewDualityZ} 
 \inf_{\pi  } {\textstyle \int c(\omega,t,y)} \, d\pi(\omega, t,y)=\sup_{(\phi,\psi)} {\textstyle \int \phi \, d\Q+\int \psi\, d\mu} ,
\end{align}
where the infimum is taken over the set 
$\TRST^{1,V} (\mu)=\{\pi\in \TRST^1(\mu): \pi(h)\leq V\}$
 and the supremum is taken over  $\phi\in C_b(\CRo)$, $\psi\in C_b(\R)$ for which
\begin{align*}
\exists{\alpha\geq 0} \mbox{ s.t. } \phi^M_t(\omega)\!+\! \psi (y)\!-\!\alpha(h_t\!-\!V) \leq c(\omega, t, y) \mbox{ for } \omega \in \CRo, t\in\R_+, y\in \R.
\end{align*}
The argument used to derive Theorem \ref{HedgingDual} from Proposition \ref{super hedging} then implies the desired duality relation $P^Z_\gamma=D^Z_\gamma$, with a decreasing process  (in \eqref{ZDualConstraint}) of the form $A_t= -\alpha (h_t-V)$ for some $\alpha \geq0$. In Case (1) of Assumption \ref{ass:CompactPlus}, $A$ can be `hidden' in $M$ / $\psi$ as in \eqref{ToMart}.
\end{proof}

\begin{proof}[Proof of Theorem \ref{ZMP}]
  Apart from the abstract theory the ingredients of the proof of Theorem \ref{GlobalLocal2} are Proposition \ref{Invisible2} and Proposition \ref{KeLe}. The only stage where the proof of Proposition \ref{Invisible2} has to be altered is when establishing that the randomized stopping time $\xii^\pi$ is minimal. Under Assumption \eqref{ass:CompactPlus} (1) this follows using the minimality characterization given in \eqref{TimeLike}, under Assumption \eqref{ass:CompactPlus} (2) this is of course trivial.

  Proposition \ref{KeLe} only uses transport duality, the Feller property to construct $S$\!-continuous martingales and Choquet's capacitability theorem.
\end{proof}

\subsection{Examples}\label{sec:ex}
We now provide a list of Examples in which Assumption \ref{ass:CompactPlus} is satisfied and Theorems  \ref{ZMinimizerExists}, \ref{ZDuality}, and \ref{ZMP} apply. 
\subsubsection{Let $Z$ be a one-dimensional Brownian motion and assume that $\lambda$ and $\mu$ have first moments and are in convex order.
 Then Assumption \ref{ass:CompactPlus} (1) holds.} \label{lem:fin1moment}
\begin{proof}
  By the de la Vall\'ee-Poussin theorem (see e.g.\ \cite[Thm.\ II 22]{DeMeA}) there exists a positive, smooth and symmetric function $F:\R\to\R_+$ with strictly positive, bounded second derivative and $\lim_{x\to\infty} F(x)/x=\infty$ such that $V:=\int F(x)\, \mu(dx) <\infty.$
  We set
$$ \textstyle  \zeta_t := 1/2 \int_0^t F''(Z_s) \, ds\,  $$
and note that $\zeta_t$ increases to $\infty$ since $\Q\big(\int_0^\infty \1_{[-1,1]} (Z_t)\, dt= \infty\big)=1$ and $F''$ is bounded away from $0$ on $[-1,1]$. Using It\^o's formula and our conditions on $F$ we define the martingale
$$\textstyle X_t := F(Z_t)- 1/2\int_0^t F''(Z_s)\,ds
= F(Z_t)-\zeta_t.$$ In the present Brownian case, it is known that the minimality of a finite stopping time $\tau$ is equivalent to $(Z_{\tau \wedge t})_{t\geq 0}$ being a uniformly integrable martingale. This follows (in the case of a general starting law) from Lemma~12 and Theorem~17 of \cite{Co08}.

If $Z_\tau \sim \mu$ and $(Z^\tau_{ t})_{t\geq 0}$ is uniformly integrable, then for each $t$, the law of $Z_{\tau\wedge t}$ is bounded by $\mu$ in the convex order and in particular $\E [F(Z_{\tau\wedge t})] \leq V, t\geq 0$. 
Uniform integrability of $X$ then follows upon noting
$$ \E[\zeta_\tau] = \lim_{t\to \infty} \E[\zeta^\tau_t]  = \lim_{t\to \infty}\E [F(Z^\tau_t)] - \E[F(Z_0)] \leq V - \E[F(Z_0)] < \infty. \qedhere
$$
\end{proof} 

\subsubsection{One-dimensional regular diffusions} 
Let $Z$ be a regular (time-homogeneous) one-dimensional diffusion on an interval $I \subseteq \R$, with inaccessible or absorbing endpoints (see \cite{RoWi00} for the relevant definitions and terminology) and $Z_0 \sim \lambda$, $\lambda(I^\circ)=\mu(I^\circ)=1$. In particular, $Z$ is a continuous Feller process (\cite[Proposition~V.50.1]{RoWi00}). Then (on a possibly enlarged probability space) there exists a scale function $s$ and a continuous, strictly increasing time change $A_t$ such that $B_t = s(Z_{A_t})$ is a Brownian motion up to the exit of $s(I^\circ)$. Recalling the discussion in \cite[Section~5]{CoHo05b}, with the obvious extension of our notation, it is clear that there exists a minimal stopping time $\tau$ embedding $\mu$ in $Z$ if and only if there exists a stopping time $\tau'$ embedding $s(\mu)$ in $B$ such that 
\begin{align}\label{eq:exitbound} 
  \tau' \le \tau_{s(I)} := \inf \{ t \ge 0: B_t \not \in s(I^\circ)\}. 
\end{align}
Moreover, write $A^{-1}_t$ for the inverse of $A_t$, so $A^{-1}_{A_t} = t$. Since $A$ and $A^{-1}$ are continuous and strictly increasing $\tau$ is a minimal embedding of $\mu$ in $Z$ if and only if $\tau':=A^{-1}_\tau$ is a minimal embedding of $s(\mu)$ in $B$.

We now consider three cases. In the first two we verify Assumption \ref{ass:CompactPlus} (2) and in the last case we verify Assumption \ref{ass:CompactPlus} (1) under some additional smoothness assumptions. Subsequently we give some concrete examples.

\begin{itemize} 
\item[(i)] Suppose $s(I^\circ) = (a,b)$ for $a, b \in \R$. Then it follows from \cite[Theorems~17 and 22]{Co08} that  a solution to \eqref{eq:FellerSEP} exists if and only if $s(\lambda)$ precedes $s(\mu)$ in convex order, and in fact, any finite $\tau$ with $Z_\tau \sim \mu$ is minimal.

Moreover we note  that \begin{itemize}
\item $
\{\tau': B_{\tau'}\sim s(\mu),  \tau' \mbox{ is a minimal} \}
$
is bounded in probability
\item $A_t<\infty$ provided the path $(B_s)_{s\leq t}$ stays inside an interval $[c,d]\subseteq (a,b)$.
\item Given $\eps>0$ there exists an interval $[c,d]\subseteq (a,b)$ such that $(B_s)_{s\leq \tau'}$ stays inside $[c,d]$ with probability $> 1-\eps$ for each minimal $\tau'$, $B_{\tau'}\sim s(\mu)$.  
\end{itemize}
It follows that $ \{A_{\tau'}: B_{\tau'}\sim s(\mu), \tau' \mbox{ is minimal} \} $ is bounded in probability, hence \eqref{TimeLike2.5} and then Assumption \ref{ass:CompactPlus} (2) holds.
\item[(ii)] Suppose $s(I^\circ) = (a,\infty)$ for $a \in \R$, and that  $s(\lambda)$ and  $s(\mu)$ are in convex order and that the moments $m_\lambda = \int s(y) \,\lambda(dy)$, $m_\mu = \int s(y)\, \mu(dy)$ exist. Then it follows from Theorems~17 and 22 and the discussion at the top of p.~245 of \cite{Co08} that a solution to \eqref{eq:FellerSEP} exists if and only if for all $x \ge a$, 
\begin{align}\label{FinalMoments}\textstyle-\int |s(y)-x|\,\mu(dy) \le -\int |s(y) - x|\, \lambda(dy) + (m_\lambda-m_\mu)\end{align}
 Again, any finite $\tau$ with $Z_\tau \sim \mu$ is minimal and  \eqref{TimeLike2.5} follows as above. 
 
 An analogous result holds if $s(I^\circ) = (-\infty,b)$ for $b \in \R$.

\item[(iii)] Suppose $s(I^\circ) = (-\infty, \infty)$ and that $s(\lambda), s(\mu)$ are in convex order, $ \int s(y)^2 \,\mu(dy) < \infty$. Then we are in the classical case, and a stopping time $\tau$ with $Z_\tau \sim \mu$ is minimal if and only if $\E[A^{-1}_\tau] < \infty$. If the process $Z$ is sufficiently well-behaved (as in the examples below) one can show that $X_t = s(Z_t)^2 - A_t^{-1}$ is a martingale and that $A^{-1}$ depends continuously on the path $(Z_s)_{s\leq t}$. For all $\tau$ solving \eqref{eq:FellerSEP}, $\E[A^{-1}_\tau] < \infty$; hence $(X^\tau_t)_{t\geq 0}$ is uniformly integrable and Assumption~\ref{ass:CompactPlus} (1) is satisfied.

More generally, when only the integrals $\int s(y) \,\lambda(dy)$,   $\int s(y) \,\mu(dy)$ are finite, (assuming sufficient regularity of the diffusion), Assumption \ref{ass:CompactPlus} (1) follows as in Section \ref{lem:fin1moment}.
\end{itemize}

\begin{remark}
  Observe that none of the constructions described in Sections~\ref{sec:recov-class-embedd} and \ref{ValoisSection} rely on fine properties of Brownian motion --- the main properties used are the continuity of paths, the strong Markov property, and the regularity and diffusive nature of paths (that the process started at $x$ immediately returns to $x$, and immediately enters the sets $(x,\infty)$ and $(-\infty,x)$). It follows that all the given constructions extend to the case of regular diffusions described above.
\end{remark}
\begin{example}[Brownian motion with drift]
Let $Z_t=B_t+at$ for some $a<0$ with $Z_0\sim \lambda$, and $I=(-\infty,\infty)$. Then a possible choice of the scale function is $s(x) = \exp(-2 a x)$.  Let $\lambda, \mu\in\mathcal P(\R)$ be such that $s(\lambda), s(\mu)$  are integrable and satisfy \eqref{FinalMoments}. Then Assumption \ref{ass:CompactPlus} holds by (ii) above.
\end{example}

\begin{example}[Geometric Brownian motion]
  Let $Z$ be a geometric Brownian motion, given through the SDE $dZ_t= Z_tdB_t$, $Z_0\sim \lambda~.$ A possible choice of scale function is $s(x)=x.$ Let $\lambda,\mu\in\mathcal P(0,\infty)$ be such that $s(\lambda),s(\mu)$ are integrable and satisfy the corresponding version of \eqref{FinalMoments}. Then Assumption \ref{ass:CompactPlus} holds by (ii) above. (More general versions of geometric Brownian motion can be treated similarly.) \end{example}

\begin{example}[Three-dimensional Bessel process] 
Let $Z=|B|$ for a three-dimensional Brownian motion $(B_t)_{t\geq 0}$ with $Z_0\sim \lambda.$ A possible choice of scale function is $s(x)=1-1/x,$ and $s(I^\circ)=(-\infty, 1)$. Let $\lambda,\mu\in\mathcal P(0,\infty)$ be such that $s(\lambda),s(\mu)$ are integrable and satisfy the corresponding version of \eqref{FinalMoments}. Then Assumption \ref{ass:CompactPlus} holds by (ii) above. Similar results hold for $d$-dimensional Bessel processes, with $d > 2$.
\end{example}

\begin{example}[Ornstein-Uhlenbeck process]
Let $Z$ be an Ornstein-Uhlenbeck process, given for example as the solution to the SDE $dZ_t = -Z_t \, dt + dW_t, Z_0\sim\lambda$. Then $Z_t$ is a regular diffusion on $I=(-\infty,\infty)$ with scale function given (up to constants) by $s'(x) = \exp(x^2)$, and $s(I^\circ) = (-\infty,\infty)$. Suppose $\lambda, \mu$ are measures on $\R$ such that $s(\lambda), s(\mu)$ are in convex order and $ \int s(y)^2 \,\mu(dy) < \infty$. Then $A_t^{-1} = \int_0^t \exp\{2Z_t^2\} \, ds$ is continuous as a function of $(Z_s)_{s \le t}$, and hence
Assumption \ref{ass:CompactPlus} holds by (iii) above.
\end{example}

\subsubsection{The Hoeffding-Frechet coupling as a very particular Root solution} Let $Z$ be the deterministic process given by $dZ_t= dt$ started in $Z_0 \sim \lambda$. $Z$ is not a regular diffusion, however Assumption~\ref{ass:CompactPlus} (2) is easily checked. Let $\mu$ be another probability and assume for simplicity that $\max \supp \lambda \leq \min \supp \mu$.  Then the Root solution minimizes $\E[\tau^2]$. But note also that since $\tau= Z_\tau-Z_0$, this minimization problem corresponds precisely to finding the joint distribution $(Z_0, Z_\tau)$ which minimizes $\E[(Z_\tau-Z_0)^2]$: the classical transport problem in the most simple setup. Specifically, the Root solution for the particular case of the process $Z$ corresponds precisely to the monotone (Hoeffding-Frechet) coupling. In the same fashion the Rost solution corresponds to the co-monotone coupling between $\lambda$ and $\mu$.

\bibliography{joint_biblio}{}
\bibliographystyle{plain}
\end{document}